\documentclass{article}
\usepackage[T1]{fontenc}
\usepackage{appendix}
\usepackage{calligra}
\usepackage[utf8]{inputenc}
\usepackage[english]{babel}
\usepackage{graphicx,graphics}
\usepackage{rotating}
\usepackage[twoside]{geometry}
\usepackage{epsfig}
\usepackage{indentfirst}
\usepackage[usenames,dvipsnames,svgnames,table]{xcolor}
\usepackage{cmap}
\usepackage{graphicx,graphics}
\usepackage{enumerate}
\usepackage{hyperref}
\hypersetup{
pdftitle={LOG},    
pdfauthor={},    
pdfnewwindow=true,      
colorlinks=true,      
linkcolor=black,         
citecolor=blue}

\usepackage{amssymb,amsfonts,amstext,amsthm}
\usepackage{mathrsfs}
\usepackage[intlimits]{amsmath}
\newenvironment{proof1}[1][\textbf{Proof of Theorem~\ref{theo4.5}}]{\textit{#1.} }{\hfill $\Box$}

\newenvironment{proof2}[1][\textbf{Proof of Theorem~\ref{teo4.5}}]{\textit{#1.} }{\hfill $\Box$}

\newenvironment{proof3}[1][\textbf{Proof of Theorem~\ref{theo4.4}}]{\textit{#1.} }{\hfill $\Box$}

\newenvironment{proof4}[1][\textbf{Proof of Theorem~\ref{teo4.4}}]{\textit{#1.} }{\hfill $\Box$}

\newenvironment{proof5}[1][\textbf{Proof of Theorem~\ref{theor4.7}}]{\textit{#1.} }{\hfill $\Box$}

\newcommand{\zerarcounters}
{
\setcounter{equation}{0}
\setcounter{theorem}{0}
}

\newcommand{\OBSI}{\begin{remark}\begin{rm}}
\newcommand{\OBSF}{\end{rm}\end{remark}}
\newcommand{\DEFI}{\begin{definition}\begin{rm}}
\newcommand{\DEFF}{\end{rm}\end{definition}}

\newcommand{\dint}{\displaystyle\int}

\newcommand{\be}{\begin{eqnarray}}
\newcommand{\en}{\end{eqnarray}}
\newcommand{\bee}{\begin{eqnarray*}}
\newcommand{\ene}{\end{eqnarray*}}

\DeclareMathOperator*{\Ran}{Ran}

\DeclareMathOperator*{\dom}{\mathcal{D}}
\DeclareMathOperator*{\dist}{dist}
\newtheorem{definition}{\bf Definition}[section]
\newtheorem{proposition}{\bf Proposition}[section]

\newtheorem{lemma}{\bf Lemma}[section]
\newtheorem{theorem}{\bf Theorem}[section]

\newtheorem{corollary}{\bf Corollary}[section]

\newtheorem{remark}{Remark}[section]

\newtheorem{example}{Example}[section]

\title{Refined decay rates of $C_0$-semigroups on Banach spaces}
\author{Genilson S. de Santana and Silas Luiz de Carvalho}

\date{\today}

\begin{document}

\maketitle

\begin{abstract} 
We study rates of decay for $C_0$-semigroups on Banach
spaces under the assumption that the norm of the resolvent of the semigroup generator grows with $|s|^{\beta}\log(|s|)^b$, $\beta, b \geq 0$, as $|s|\rightarrow\infty$, and with $|s|^{-\alpha}\log(1/|s|)^a$, $\alpha, a \geq 0$, as $|s|\rightarrow 0$. Our results do not suppose that the semigroup is bounded. In particular, for $a=b=0$, our results improve the rates involving Fourier types obtained by Rozendaal and Veraar (J. Funct. Anal. 275(10): 2845–2894, 2018).
\end{abstract}  

%%%%%%%%%%%%%%%%%%%%%%%%%%%%%%%%%%%%%%%%%%%%%%%%%%%%%%%%%%%%%%%%%%%%%%%%%%%%%%%%%%%%%%%%%%%%%%%%%%%%%%%%%%%%%%%%%%%%%%%%%%%%%%%%%%%%%%%%%%%%%%%%%%%%%%%%%%%%%%%%%%%%%%%%%%%%%%%%%%%%--Introduction--%%%%%%%%%%%%%%%%%%%%%%%%%%%%%%%%%%%%%%%%%%%%%%%%%%%%%%%%%%%%%%%%%%%%%%%%%%%%%%%%%%%%%%%%%%%%%%%%%%%%%%%%%%%%%%%%%%%%%%%%%%%%%%%%%%%%%%%%%%%%%%%%%%%%%%%%%%%%%%%%%%%%%%%%%%%%%%%%%%%%%%%%%

\section{Introduction}
\zerarcounters

\subsection{Historical background}
An important question in the theory of differential equations refers to the asymptotic behavior (in time) of their solutions; more specifically, if they reach an equilibrium and, if so, with
which speed. For those linear partial
differential equations which can be conveniently analyzed by rewriting them as evolution
equations, it is well known that the long-term behavior of the solutions of each one of these equations is related to some spectral properties (and behavior of the resolvent) of the generator of the associated semigroup.

The asymptotic theory of semigroups provides tools for investigating the convergence to zero of mild and classical solutions to the abstract Cauchy problem 
\begin{equation}\label{1.0}
\left\{\begin{array}{ll}u'(t)+Au(t)=0, \ \ \  t\geq 0 \\ u(0)=x, \end{array}\right.
\end{equation}
We know that $\eqref{1.0}$ has a unique mild solution for every $x\in X$, and that the solution depends continuously on $x$ if, and only if, $-A$ generates a $C_0$-semigroup  $(T(t))_{t\geq 0}$ on $X$ (see \cite{valued, Engel}). In this case,
the unique solution $u$ to $\eqref{1.0}$ is given by $u(t)=T(t)x, \  \ \forall~t\geq 0$, and if $x\in \mathcal{D}(A)$, then
$u \in C^{1}([0,\infty),X)$ (see $\cite{Engel}$, Proposition~II.6.2).

For the classic theory of ODEs in finite dimension, the Lyapunov stability criterion (see $\cite{Engel}$, Theorem 2.10) 
is an excellent tool in the study of the asymptotic behavior of solutions to $\eqref{1.0}$, but this criterion is in general not valid if $X$ has infinite
dimension. However, in this case, the asymptotic behavior can be deduced from of the norm of the resolvent of the operator $A$. For example, on a Hilbert
space $X$, one has the Gearhart(1978)-Pr\"uss(1984)-Greiner(1985) Theorem. 

In what follows, $\rho(A):=\{\lambda\in\mathbb{C}\mid \Vert (\lambda-A)^{-1}\Vert_{\mathcal{L}(X)}<\infty\}$ and $\sigma(A):=\mathbb{C}\setminus\rho(A)$ stand, respectively, for the resolvent set and the spectrum of $A$, a densely defined linear operator in a Banach (Hilbert) space $X$. 

\begin{theorem}[Theorem~I1.10 in \cite{Engel}]
\begin{rm}
A $C_0$-semigroup $(T(t))_{t\geq 0}$ on a Hilbert space $X$ is uniformly
exponentially stable if, and only if, its generator $-A$ satisfies $\mathbb{C}_{-}\subset \rho(A)$ and %of its generator $-A$, with %the resolvent satisfying
\begin{equation*}
    \sup_{\text{Re} \lambda<0} \|(\lambda+A)^{-1}\|_{\mathcal{L}(X)}<\infty.
\end{equation*}
\end{rm}
\end{theorem}

\OBSI A uniform bound for the resolvent is not sufficient to ensure exponential stability
on general Banach spaces; see Counterexample~IV.2.7 in $\cite{Engel}$.
\OBSF

The works of Lebeau $\cite{labeau1, labeau2}$ and Burq $\cite{burq}$
raised the question of what is the relation between the growth rates for norm of the resolvent and the
decay rates of the norm of semigroup orbits.
More precisely, assuming a spectral condition under the generator, $\sigma(A)\subset \mathbb{C}_{+}$ in $\eqref{1.0}$, and $\|R(is, A)\|_{\mathcal{L}(X)} \rightarrow \infty$ as $|s|\rightarrow \infty$, then $(T(t))_{t\geq 0}$ is not exponentially stable and one typically obtains other
asymptotic behavior.  

Until 2010, much attention has been paid
to polynomial decay rates of the norm of semigroup orbits. In the work of ~\cite{batki}, Bátkai, Engel, Pr\"uss and Schnaubelt proved that for uniformly bounded semigroups, a polynomial growth rate of the norm of the resolvent implies a specific polynomial decay rate for
classical solutions to $\eqref{1.0}$. 

\begin{theorem}[Theorem 3.5 in~\cite{batki}]\label{TBEPS}
\begin{rm}
Let $(T(t))_{t\geq 0}$ be a \textit{bounded} semigroup on
a Banach space $X$ with infinitesimal generator $-A$ such that $\sigma(A)\cap i\mathbb{R}=\emptyset$. Let $s\geq 0 $ and set
\begin{equation}\label{eq2}
M(s):=\sup_{|\xi|\leq s}\|(i\xi+A)^{-1}\|_{\mathcal{L}(X)}.
\end{equation}
If there exist constants $C,\beta>0$ such that $M(s) \leq C(1+s)^{\beta}$, then  for each $\varepsilon>0$, there exists a positive constant $C_{\varepsilon}$ such that for each $t>0$,
\begin{equation}\label{ETBEPS}
    \|T(t)(1+A)^{-1}\|_{\mathcal{L}(X)}\leq C_{\varepsilon} t^{-\frac{1}{\beta}+\varepsilon}.%, \ \ \forall ~t>0.
\end{equation}
\end{rm}
\end{theorem}

Liu and Rão obtained in \cite{liu} sharper estimates than those given by \eqref{ETBEPS} in case $X$ is a Hilbert space. 
\begin{theorem}[Theorem 2.1 in~\cite{liu}]
 \begin{rm}
Let $X$ be a Hilbert space, and let $(T(t))_{t\geq 0}$, $A$ and $M$ be as in the statement of Theorem~\ref{TBEPS}. %, with $X$  % be a \textit{bounded} semigroup on 
%a Hilbert space. % $X$ with as the infinitesimal generator $-A$ such that $\sigma(A)\cap i\mathbb{R}=\emptyset$. Let $M:(0,\infty)\rightarrow (0,\infty)$ defined in $\eqref{eq2}$ and
Then, if there exist constants $C,\beta>0$ such that $M(s) \leq C(1+s)^{\beta}$, then %there exists a positive constant $C_\varepsilon$ such that
\begin{equation*}
    \|T(t)(1+A)^{-1}\|_{\mathcal{L}(X)}=O\left(\frac{\log(2+t)^{\frac{1}{\beta}+1}}{t^{\frac{1}{\beta}}}\right), \ t\rightarrow \infty.
\end{equation*}   
\end{rm}
\end{theorem}

In~\cite{Duykaerts}, Batty and Duyckaerts
extended this correspondence to the case where the  resolvent growth is arbitrary; they were also able to reduce the loss $\varepsilon>0$ (see relation~\eqref{ETBEPS}) to a logarithmic scale.

\begin{theorem}[Theorem 1.5 in~\cite{Duykaerts}]\label{batty}
\begin{rm}
Let $(T(t))_{t\geq 0}$ be a \textit{bounded} semigroup on
a Banach space $X$ with infinitesimal generator $-A$ such that $\sigma(A)\cap i\mathbb{R}=\emptyset$. Let $M:(0,\infty)\rightarrow (0,\infty)$ be given by $\eqref{eq2}$; then, there exists a positive constant $C$ such that
\begin{equation*}
    \|T(t)(1+A)^{-1}\|_{\mathcal{L}(X)}=O\left(\frac{1}{M^{-1}_{\log}(Ct)}\right), \ t\rightarrow \infty,
\end{equation*}
where $M^{-1}_{\log}$ is the right inverse of $M_{\log}(s):=M(s)(\log(1+M(s))+\log(1+s))$. In particular, if $M(s)\leq C (1+s)^{\beta}$ for any $\beta >0$ and $C>0$, then
\begin{equation*}
  \|T(t)(1+A)^{-1}\|_{\mathcal{L}(X)}=O\left(\frac{\log(t)}{t}\right)^{1/\beta}, \ t\rightarrow \infty.  
\end{equation*}
\end{rm}
\end{theorem}

Still in \cite{Duykaerts}, Batty and Duyckaerts conjectured that the
logarithmic correction may be dropped in the case of Hilbert spaces, but one cannot expect rates better than $(M^{-1}_{\log}(Ct))^{-1}$ for general Banach spaces. Then, Borichev and Tomilov partially solved the conjecture in \cite{tomilov}; namely, they have shown that in case of a power-law resolvent growth, the logarithmic correction loss is sharp on general Banach spaces (it is worth noting that this optimality is also valid for sub-polynomial functions, as recently shown by Dubruyne and Seifert in~\cite{Seifert}), but that it is not necessarily true on Hilbert spaces. 
\begin{theorem}[Theorem 2.4 in \cite{tomilov}]
\begin{rm}
Let $(T(t))_{t\geq0}$ be a \textit{bounded} $C_0$-semigroup on a
Hilbert space $X$ with generator $-A$ so that $i\mathbb{R}\subset \rho(A)$. Then, given $\beta>0$, the following
assertions are equivalent:
\begin{enumerate}
\item $\|(is+A)^{-1}\|_{\mathcal{L}(X)}=O(|s|^{\beta})$, $|s|\rightarrow \infty$.
\item $\|T(t)(1+A)^{-1}\|_{\mathcal{L}(X)}=O(t^{-1/\beta})$, $t\rightarrow \infty$.
\end{enumerate}

\end{rm}
\end{theorem}

By seeking to answer the conjecture of Batty and Duyckaerts for a larger class of functions than power-law type, Batty, Chill and Tomilov have obtained in \cite{chill} the following result.
\begin{theorem}[Theorem 1.1 in~\cite{chill}]\label{BCT}
\begin{rm}
Let $(T(t))_{t\geq0}$ be a \textit{bounded} $C_0$-semigroup on a
Hilbert space $X$ with generator $-A$ so that $i\mathbb{R}\subset \rho(A)$. Let $\beta>0$ and $b>0$.
\begin{enumerate}
\item The following
assertions are equivalent:
\begin{enumerate}
\item $\|(is+A)^{-1}\|_{\mathcal{L}(X)}=O(|s|^{\beta}\log(|s|)^{-b})$, $|s|\rightarrow \infty$.
\item $\|T(t)(1+A)^{-1}\|_{\mathcal{L}(X)}=O(t^{-1/\beta}\log(t)^{-b/\beta})$, $t\rightarrow \infty$.
\end{enumerate}

\item If \[\|(is+A)^{-1}\|_{\mathcal{L}(X)}=O(|s|^{\beta}\log(|s|)^{b}), \quad |s|\rightarrow \infty,\] then for each $\varepsilon>0$,
\[\|T(t)(1+A)^{-1}\|_{\mathcal{L}(X)}=O(t^{-1/\beta}\log(t)^{b/\beta+\varepsilon}),\quad t\rightarrow \infty\]

\end{enumerate}
\end{rm}
\end{theorem}
\OBSI Theorem~\ref{BCT} remains valid when replacing $|s|^{\beta}\log(|s|)^{-b}$ with a function of the type $|s|^{\beta}\ell(|s|)^{-1}$, where $\ell$ is an increasing and slowly varying function (see Theorem 5.6 in~\cite{chill}).
\OBSF

Finally, Rozendaal, Seifert and Stahn in \cite{Stahn} have extended the previous results to a larger class of functions, namely, those of positive increase:  a continuous increasing function $M:(0,\infty)\rightarrow (0,\infty)$ is said to be of \textit{positive increase} if
there exist positive constants $\alpha>0$, $c\in (0,1]$ and $s_0>0$ such that
\begin{equation*}
    \frac{M(\lambda s)}{M(s)}\geq c\lambda^{\alpha}, \ \ \ \lambda \geq 1, s\geq s_0.
\end{equation*}

\begin{theorem}[Theorem 3.2 in \cite{Stahn}] \label{rss}
\begin{rm}
Let $(T(t))_{t\geq 0}$ be a \textit{bounded} $C_0$-semigroup on a Hilbert space
$X$, with generator $-A$, and let $M:(0,\infty)\rightarrow (0,\infty)$ be a function of positive increase. The following assertions
are equivalent:
\begin{enumerate}
\item $i\mathbb{R}\subset \rho(A) $ and $\|(is+A)^{-1}\|_{\mathcal{L}(X)}=O(M(|s|))$, $|s|\rightarrow \infty$.
\item $\|T(t)(1+A)^{-1}\|_{\mathcal{L}(X)}=O(M^{-1}(t))$, $t\rightarrow \infty$.
\end{enumerate}  
\end{rm}    
\end{theorem}

So far we have presented a compilation of the main results for the situation in which $A$ has only singularity at infinity, i.e, $\|R(is,A)\|_{\mathcal{L}(X)}\rightarrow \infty$ as $|s|\rightarrow \infty$; there are also some other works in the literature that study the decay rates of $C_0$-semigroup for this situation, for example~\cite{batki, chill2,Paunonen, roze}. Nevertheless, there are many other works that study decay rates for the situation in which $A$ has a singularity at zero~\cite{chill,chill1, rozendaal, Stahn}, or even when $A$ has singularity at zero and infinity~\cite{chill,Martinez,rozendaal,Stahn}. In the present work, we consider all of these scenarios.

Until this point, we have presented some of the main results of the asymptotic theory of bounded $C_0$-semigroups. Nevertheless, there are many natural classes of examples where the norm of the resolvent of the generator grows with a power-law rate as $|s|\rightarrow\infty$, for example, but the semigroup is
not uniformly bounded, or where it is unknown whether the semigroup is in fact bounded. For example, this happens with some concrete partial
differential equations, like the standard wave equation with periodic
boundary conditions; here, uniform boundedness fails (see~\cite{rozendaal} for a more complete discussion on these examples).

The currently available literature on polynomial
or other types of decay deals almost exclusively with uniformly bounded semigroups. To the best
of our knowledge, the following result due to Bátkai, Engel, Pr\"uss and Schnaubelt is the first in the literature that proves polynomial decay for not necessarily bounded semigroups. In what follows, $\omega_0(T):=\lim_{t\to\infty}(\log\Vert T(t)\Vert_{\mathcal{L}(X)})/t$.

\begin{theorem}[Proposition 3.4 in~\cite{batki}]\label{batkai}
\begin{rm}
Let $(T(t))_{t\geq 0}$ be a semigroup defined in a Banach space $X$ with generator $-A$ such that there exists $\beta>0$ so that the map $\lambda \mapsto (\lambda+A)^{-1}(1+A)^{-\beta}$, with $\text{Re} \ \lambda>\omega_0(T)$, has a bounded holomorphic extension to $\text{Re} \ \lambda \geq 0$. Then, there exists a positive constant $C_{n,\delta}$ such that for each $n\in \mathbb{N}$, $\delta\in (0,1]$ and $t>0$,
\begin{equation*}
    \|T(t)(1+A)^{-\beta(n+1)-1-\delta}\|_{\mathcal{L}(X)}\leq C_{n,\delta}t^{-n}.
\end{equation*}
\end{rm}
\end{theorem}

Then, by using geometrical properties of the underlying Banach space (like its Fourier type), Rozendaal and Veraar have shown the following result (see Theorem 4.9 in~\cite{rozendaal}).
\begin{theorem}\label{theoem 1.9}
\begin{rm}
Let $(T(t))_{t\geq 0}$ be a $C_0$-semigroup with generator $-A$ defined in a Banach space
$X$ with Fourier type $p\in [1,2]$, and let $\frac{1}{r}=\frac{1}{p}-\frac{1}{p'}$ (where $\frac{1}{p}+\frac{1}{p'}=1$). Suppose that  $\overline{\mathbb{C}_{-}}\subset \rho(A)$ and that there exist $\beta,C \geq 0$ such that  $\|(\lambda+A)^{-1}\|_{\mathcal{L}(X)}\leq C
(1+|\lambda|)^{\beta}$ for each $\lambda \in \overline{\mathbb{C}_{-}} $. Let $\tau>\beta+1$; then, for each $\rho \in \left[0, \frac{\tau-1/r}{\beta}-1\right)$, there exists $C_\rho \geq 0$ such that for each $t \ge 1$,
\begin{equation}\label{eeq3}
    \|T(t)(1+A)^{-\tau}\|_{\mathcal{L}(X)}\leq C_{\rho}t^{-\rho}.% \ \ \ t\geq 1.
\end{equation}
\end{rm}    
\end{theorem}

 In case $X$ is a Hilbert space (which corresponds to $p=p'=2$ and $r=\infty$), they have shown the following result. 
 \begin{corollary}[Theorem 1.1 in~\cite{rozendaal}]\label{TRV}
\begin{rm}
Let $(T(t))_{t\geq 0}$ be a $C_0$-semigroup with generator $-A$ defined in a Hilbert space $X$. Suppose that $\overline{\mathbb{C}_{-}}\subset \rho(A)$ and that there exist $\beta,C\geq 0$ such that $\|(\lambda+A)^{-1}\|_{\mathcal{L}(X)}\leq C
(1+|\lambda|)^{\beta}$ for each $\lambda \in \overline{\mathbb{C}_{-}} $. Then, there exists a positive constant $C_{\tau}$ such that for each $t\ge 1$,
\begin{equation*}
\|T(t)(1+A)^{-\tau}\|_{\mathcal{L}(X)}\leq C_\tau t^{1-\tau/\beta}.
\end{equation*}
\end{rm}    
 \end{corollary}

 \subsection{Main results}

By using the techniques developed in \cite{rozendaal} that involve Fourier Multipliers and also inspired by the techniques developed by Batty, Chill and Tomilov in \cite{chill} that involve functional calculus of sectorial operators, we have obtained decay rates for $C_0$-semigroups as defined in the statement of Theorem~\ref{TRV} by assuming that the norm of the resolvent of the generator behaves as a function of type $|s|^{\beta}\log(|s|)^b$ as $|s|\to\infty$ (a particular example of a \textit{regularly varying} function). Under these assumptions on the resolvent and without the assumption of boundedness of the semigroup, to the best knowledge of the authors, these estimates are new and constitute one of the  main results in this work.

\begin{theorem}\label{theo4.5}
\begin{rm}
Let $\beta>0$, $b\geq 0$ and let $(T(t))_{t\geq 0}$ be a $C_0$-semigroup  defined in the Banach space $X$ with Fourier type $p\in [1,2]$, with $-A$ as its generator. Suppose that $\overline{\mathbb{C}_{-}}\subset \rho(A)$ and that for each $\lambda \in \mathbb{C}$ with $\text{Re}(\lambda)\le 0$,
\[\|(\lambda+A)^{-1}\|_{\mathcal{L}(X)}\lesssim (1+|\lambda|)^{\beta}(\log(2+|\lambda|))^b.\]
 Let $r\in [1,\infty]$ be such that $\frac{1}{r}=\frac{1}{p}-\frac{1}{p'}$, and let $\tau > 0$ be such that $\tau>\beta+\frac{1}{r}$. Then, for each $\delta>0$, there exist constants $c_{\delta,\tau}\in [0,\infty)$ and $t_0\ge 1$ such that for each $t\geq t_0$,
\begin{equation}\label{eqTh1.1}
 \|T(t)(1+A)^{-\tau}\|_{\mathcal{L}(X)}\leq c_{\delta,\tau}t^{1-\frac{\tau-r^{-1}}{\beta}}\log(1+t)^{\frac{b(\tau-r^{-1})}{\beta}+\frac{1+\delta}{r}}.
\end{equation}
\end{rm}
\end{theorem}

The next result is the particular case of Theorem~\ref{theo4.5} where $X$ is a Hilbert space.

\begin{corollary}\label{theo4.6}
 \begin{rm}
Let $\beta$, $b$, $A$  and $(T(t))_{t\geq 0}$ be as in the statement of Theorem~\ref{theo4.5} and let $X$ be a Hilbert space. Let $\tau>\beta$. Then, %let $A$ be an injective sectorial operator on a Hilbert space $X$ such that $-A$ generates a $C_0$-semigroup $(T(t))_{t\geq 0}$ on $X$. Suppose that $\|(is+A)^{-1}\|_{\mathcal{L}(X)}\lesssim (1+|s|)^{\beta}(\log(2+|s|))^b$ and $\overline{\mathbb{C}_{-}}\subset \rho(A)$. Then
   there exist constants $c_{\tau}\geq 0$ and $t_0\ge 1$ such that for each $t\geq t_0$,
\begin{equation}\label{eq1.6}
 \|T(t)(1+A)^{-\tau}\|_{\mathcal{L}(X)}\leq c_{\tau}t^{1-\frac{\tau}{\beta}}\log(1+t)^{\frac{b\tau}{\beta}}.
\end{equation}
\end{rm}   
\end{corollary}

Note that in case $b=0$, one obtains from Theorem $\ref{theo4.5}$ the following result. %we obtained:
\begin{corollary}\label{theo4.7}
\begin{rm}
Let $\beta>0$ and let $(T(t))_{t\ge 0}$ be a $C_0$-semigroup defined in a Banach space $X$ with Fourier type $p\in [1,2]$, whose generator is given by $-A$. Suppose that $\overline{\mathbb{C}_{-}}\subset \rho(A)$ and that for each $\lambda \in \overline{\mathbb{C}_{-}}$, $\|(\lambda+A)^{-1}\|_{\mathcal{L}(X)}\lesssim (1+|\lambda|)^{\beta}$. Let $r\in [1,\infty]$ be such that $\frac{1}{r}=\frac{1}{p}-\frac{1}{p'}$, and let $\tau> 0$ be such that $\tau>\beta+\frac{1}{r}$. Then, for each $\delta >0$ and each $\rho\in[0,1-(\tau-r^{-1})/\beta]$, there exist constants $c_{\delta,\tau}\in [0,\infty)$ and $t_0\ge 1$ such that for each $t\geq t_0$,
\begin{equation}\label{eqqq5}
 \|T(t)(1+A)^{-\tau}\|_{\mathcal{L}(X)}\leq c_{\delta,\tau}t^{-\rho}\log(1+t)^{\frac{1+\delta}{r}}.
\end{equation}
\end{rm}  
\end{corollary}

\OBSI
\begin{enumerate}
  \item Note that relation \eqref{eqqq5} presents a sharper bound to $\|T(t)(1+A)^{-\tau}\|_{\mathcal{L}(X)}$ than the one presented in relation~\eqref{eeq3}; namely, in relation~\eqref{eqqq5}, the exponent in $t$ is precisely the unattained upper-bound of $\rho$ in Theorem~\ref{theoem 1.9}. This partially solves the question posed by Rozendaal and Veraar in~\cite{rozendaal} if whether~\eqref{eeq3} is valid  for $\rho=\dfrac{\tau-1/r}{\beta}-1$ or not, given that the bound presented in~\eqref{eqqq5} has a logarithmic correction. Note that if one lets $b=0$ in~\eqref{eq1.6}, then Corollary~\ref{theo4.6} coincides with Corollary~\ref{TRV} for $\tau>\beta$.

\item We also note that the power law in the logarithmic factor presented in~\eqref{eqqq5} depends on the geometry of the space (that is, its Fourier type): a greater value of $r$ (which means that the space is ``closer'' to a Hilbert space) results in a lesser logarithm correction. %most  influences this logarithmic factor presented in~\eqref{eqqq5}.

\item Furthermore, such logarithm factor is not optimal, even in case $b=0$. %the norm of the resolvent of $A$ behaves as a power-law as a function of $s$.
  Namely, it is possible to obtain a version of Proposition~\ref{theor3.1} and Theorem~\ref{theo4.4} (these two results are central in the proof of Theorem~\ref{theo4.5}, which consists of "eliminating" the operator $\log(2+A)^{-b\frac{(\tau-r^{-1})}{\beta}-\frac{1+\delta}{r}}$ from $\|T(t)(1+A)^{-\tau}\log(2+A)^{-b\frac{(\tau-r^{-1})}{\beta}-\frac{1+\delta}{r}}\|_{\mathcal{L}(X)}$), where $\log(1+t)^{\frac{1+\delta}{r}}$ is replaced by
  \linebreak $\log(1+t)\log(1+\log(1+t))^{\frac{1+\delta}{r}}$; we do not present a proof of this statement, given that the techniques discussed here seem to be far from optimal. We just stress that such replacement is possible given that the functions $\log(1+t)$ and $\log(1+\log(1+t))$ are both complete Bernstein functions (see Definition~\ref{DCBF}).% In \cite{rozendaal} the question is raised whether $\eqref{eeq3}$ it is valid  for $\rho=\dfrac{\tau-1/r}{\beta}-1$. Our results show that this power can be achieved, however, logarithmic correction is required as show in $\eqref{eqqq5}$.
\end{enumerate}
\OBSF

We have also obtained similar decay rates for the situation in which $0\in \sigma(A)$. In the following result, as in Theorem~\ref{theo4.5}, let us assume that the norm of the resolvent grows with order $|s|^{-\alpha}\log(1/|s|)^a$ as $|s|\to 0$ and with order $|s|^{\beta}\log(|s|)^b$ as $|s|\to\infty$. %(in neighborhood of infinity).
\begin{theorem}\label{teo4.5}
\begin{rm}
Let $(T(t))_{t\geq 0}$ be a $C_0$-semigroup  defined in the Banach space $X$ with Fourier type $p\in [1,2]$, with $-A$ as its generator. Suppose $A$ injective, $\overline{\mathbb{C}_{-}}\setminus\{0\} \subset \rho(A)$ and that there exist $\alpha\ge1$, $\beta,a,b>0$ and positive constants $C_1$ and $C_2$ such that %for each $\lambda \in \overline{\mathbb{C}_{-}}\setminus\{0\}$,% a Banach space $X$ with Fourier type $p\in[1,2]$ and suppose 
\begin{equation}\label{eqqq51}
 \|(\lambda+A)^{-1}\|_{\mathcal{L}(X)}\le \left\{\begin{array}{cc}
  C_1|\lambda|^{-\alpha}\log(1/|\lambda|)^{a}, & |\lambda|\leq 1 \\
C_2|\lambda|^{\beta}\log(|\lambda|)^{b}, & |\lambda|\geq 1,
\end{array}\right.
\end{equation}
with $\lambda \in \overline{\mathbb{C}_{-}}\setminus\{0\}$. % Let $r\in [1,\infty]$ be such that $\frac{1}{r}=\frac{1}{p}-\frac{1}{p'}$ , and
Let $\sigma, \tau$ be such that $\sigma>\alpha-1$ and $\tau> \beta+1/r$. Then, for
each $\rho \in \left[0, \min\left\{\frac{\sigma+1}{\alpha}-1,\frac{\tau-r^{-1}}{\beta}-1\right\}\right]$  and each $\delta>1-1/r$, where $r\in [1,\infty]$ is such that $\frac{1}{r}=\frac{1}{p}-\frac{1}{p'}$, there exist  $C_{\delta,\rho}>0$ and $t_0\ge 1$ so that for each $t\ge1$,
\begin{equation}\label{APRINCIPAL}
\|T(t)A^{\sigma}(1+A)^{-\sigma-\tau}\|_{\mathcal{L}(X)}\leq C_{\delta,\rho} t^{-\rho}\log(1+t)^{c(\left \lceil{\rho}\right \rceil +1)+1/r+\delta}, 
\end{equation}
with $c=\max\{a,b\}$.
%
%If $p=2$, for each $\tau>\beta$ ($\tau \geq \beta$ if $a=b=0$) and $\sigma>\alpha-1$, then for each $\delta>1$ and for each $\rho\leq \frac{\tau}{\beta}-1$ and $\rho\le\frac{\sigma+1}{\alpha}-1$  exists a constant $c_{\delta,\rho}>0$ such that
%\begin{equation*}
 % \|T(t)A^{\sigma}(1+A)^{-\sigma-\tau}\|_{\mathcal{L}(X)}\leq C_{\delta,\rho}t^{-\rho}\log(1+t)^{c(\left \lceil{\rho}\right \rceil +1)+\delta} \ \ \ (t\geq 1).  
%\end{equation*}
%  
    \end{rm}
\end{theorem}

\OBSI
By assuming~\eqref{eqqq51}, it is natural to let $\alpha \geq 1$. Indeed, suppose that $\alpha\in[0,1)$; then, 
\begin{equation*}
  \frac{1}{\dist(\lambda,\sigma(A))} \leq \|R(\lambda,A)\|_{\mathcal{L}(X)}\lesssim |\lambda|^{-\alpha}(\log(1/|\lambda|))^a.
\end{equation*} 

Since $0\in \sigma(A)$, it follows that $\dist(\lambda,\sigma(A))\ge |\lambda|$, so $|\lambda|^{\alpha-1}(\log(1/|\lambda|))^{-a}\lesssim C$. On the other hand, since $\alpha\in[0,1)$, it follows that $\displaystyle{\lim_{\lambda \to 0}}|\lambda|^{\alpha-1}(\log(1/|\lambda|))^{-a}=\infty$, from which follows that $\alpha\geq 1$ if $0\in \sigma(A)$.
\OBSF

\OBSI In case $X$ is a Hilbert space (that is, when $p=2$), one has $r=\infty$, and so~\eqref{APRINCIPAL} is just
\begin{equation*}
 \|T(t)A^{\sigma}(1+A)^{-\sigma-\tau}\|_{\mathcal{L}(X)}\leq C_{\delta,\rho}t^{-\rho}\log(1+t)^{c(\left \lceil{\rho}\right \rceil +1)+\delta}.  
\end{equation*}
\OBSF

In case $a=b=0$, one has the following result. %Theorem~\ref{teo4.5} improves the result stated in %obtained by Rozendaal and Veraar in \cite{rozendaal} (see 
%Theorem 4.9 in \cite{rozendaal}.
\begin{corollary}\label{cor1.4}
\begin{rm}
Let $(T(t))_{t\geq 0}$ be a $C_0$-semigroup  defined in the Banach space $X$ with Fourier type $p\in [1,2]$, with $-A$ as its generator. Suppose $A$ injective,  $\overline{\mathbb{C}_{-}}\setminus\{0\} \subset \rho(A)$ and that there exist $\alpha\ge 1$, $\beta>0$ and positive constants $C_1$ and $C_2$ such that %$\lambda \in \overline{\mathbb{C}_{-}}\setminus\{0\}$% a Banach space $X$ with Fourier type $p\in[1,2]$ and suppose 
\begin{equation*}
 \|(\lambda+A)^{-1}\|_{\mathcal{L}(X)}\le \left\{\begin{array}{cc}
  C_1|\lambda|^{-\alpha}, & |\lambda|\leq 1 \\
C_2|\lambda|^{\beta}, & |\lambda|\geq 1,
\end{array}\right.
\end{equation*}
with $\lambda \in \overline{\mathbb{C}_{-}}\setminus\{0\}$. %Let $r\in [1,\infty]$ be such that $\frac{1}{r}=\frac{1}{p}-\frac{1}{p'}$ , and
Let $\sigma, \tau$ be such that $\sigma>\alpha-1$ and $\tau> \beta+1/r$. Then, for
each $\rho \in \left[0, \min\left\{\frac{\sigma+1}{\alpha}-1,\frac{\tau-r^{-1}}{\beta}-1\right\}\right]$  and each $\delta>1-1/r$, where $r\in [1,\infty]$ is such that $\frac{1}{r}=\frac{1}{p}-\frac{1}{p'}$, there exist  $C_{\rho,\delta}>0$ and $t_0\ge 1$ so that for each $t\ge1$,
\begin{equation}\label{eq1.8}
\|T(t)A^{\sigma}(1+A)^{-\sigma-\tau}\|_{\mathcal{L}(X)}\leq C_{\delta,\rho} t^{-\rho}\log(1+t)^{1/r+\delta}. 
\end{equation}
\end{rm}
\end{corollary}

\OBSI
\begin{enumerate}
\item  Note that relation \eqref{eq1.8} improves the estimates obtained by Rozendaal and Veraar in \cite{rozendaal} (see Theorem 4.9 in \cite{rozendaal}). More precisely, we show that it is possible to replace the factor $t^\varepsilon$, with $\varepsilon$ any positive number, by $\log(1+t)^{1/r+\delta}$ in their estimate.
\item Note that even in case $X$ is a Hilbert space, the estimate obtained in Corollary 4.11 in~\cite{rozendaal} still has a factor $t^\varepsilon$; Corollary~\ref{cor1.4} shows that it is possible to replace it by $\log(1+t)^{\delta}$, with $\delta>1$.

\item Corollary~\ref{cor1.4} partially solves the question posed by Rozendaal and Veraar in~\cite{rozendaal}, if whether estimate~\eqref{eq1.8} is valid  for 
$\rho=\min\left\{\dfrac{\sigma+1}{\alpha}-1,\dfrac{\tau-r^{-1}}{\beta}-1\right\}$ 
or not, given that the bound presented in~\eqref{eq1.8} has a logarithmic factor. 
\end{enumerate}
\OBSF

We also studied the situation in which there is only a singularity at zero (but not at infinity); this situation is also discussed in \cite{chill, chill1, rozendaal,Stahn}. As in \cite{chill1, rozendaal,Stahn}, we suppose that the $C_0$-semigroup is asymptotically analytic on the Banach space $X$ (see Definition~\ref{ULTIMA} and Section~\ref{sec5} for more details).

\begin{theorem}\label{theor4.7}
\begin{rm}
Let $A$ be an injective sectorial operator defined in the Banach space $X$ such that $-A$ generates $(T(t))_{t\ge0}$, an asymptotically analytic $C_0$-semigroup on $X$. Suppose that there exist $\alpha\ge 1$ and $a>0$ such that for each $\lambda\in \overline{\mathbb{C}_{-}}\setminus\{0\}$,
\begin{equation}\label{eeq42a}
  \|(\lambda+A)^{-1}\|_{\mathcal{L}(X)} \lesssim %\left\{
%  \begin{array}{cc}
    |\lambda|^{-\alpha}(\log(1/|\lambda|))^{a}.
%    (1+|\xi|)^{\beta}\log(2+|\xi|)^{b}; & |s|\gg 1 .
 % \end{array}\right.
\end{equation}
Let $\sigma>\alpha-1$. Then, for each $\delta>0$ there exists $c_{\delta,\sigma}>0$ such that for each $t\geq 1$,
\begin{equation}\label{teo43}
\|T(t)A^{\sigma}(1+A)^{-\sigma}\|_{\mathcal{L}(X)}\leq c_{\delta,\sigma}t^{1-\frac{\sigma+1}{\alpha}}\log(1+t)^{\frac{a(\sigma+1)}{\alpha}+1+\delta}.
\end{equation}

\end{rm}
\end{theorem}

In case $a=0$, the estimate presented in Theorem~\ref{theor4.7} improves the one presented in Theorem 4.16 in \cite{rozendaal}. More precisely, as in the previous cases, we have shown that it is possible to replace the factor $t^\varepsilon$ by $\log(1+t)^{1+\varepsilon}$, where $\varepsilon>0$ (see equation~\eqref{eq1.12}).

\begin{corollary}\label{cor3.2}
\begin{rm}
Let $A$, $X$ and $(T(t))_{t\ge0}$ as in Theorem~\ref{theor4.7}. Suppose that $\overline{\mathbb{C}_{-}}\setminus\{0\}\subset\rho(A)$ and that there exists $\alpha\ge 1$ such that $\|R(\lambda,A)\|_{\mathcal{L}(X)}\lesssim |\lambda|^{-\alpha}$ for each $\lambda \in \overline{\mathbb{C}_{-}}\setminus\{0\}$.
Let $\sigma>\alpha-1$. Then, for each $\delta>0$, there exists $c_{\delta,\sigma}>0$ such that for each $t\geq 1$,
\begin{equation}\label{eq1.12}
\|T(t)A^{\sigma}(1+A)^{-\sigma}\|_{\mathcal{L}(X)}\leq c_{\delta,\sigma}t^{1-\frac{\sigma+1}{\alpha}}\log(1+t)^{1+\delta}.
\end{equation}
\end{rm}
\end{corollary}

%We consider the situation in that resolvent grows as logarithmic and obtained a decay for $\|T(t)(1+A)^{-1}\|_{\mathcal{L}(X)}$ is Theorem~\ref{teo3.22} . In this case, we also obtain an estimate in terms of the fourier type of the Banach space, without anything existing of the $C_0$-semigroup, just spectral properties of its generator. Our estimate is new for unbounded semigroups, but we believe it is not optimal.

\

The text is organized as follows. In Section~\ref{SPre} we present some important definitions and auxiliary results needed throughout the text. In Sections~\ref{three},~\ref{four} and~\ref{sec5}, we present and prove some auxiliary results needed in the 
the proofs of Theorems~\ref{theo4.5},~\ref{teo4.5} and~\ref{theor4.7}, respectively. In Subsection~\ref{finalmente} we also study the situation for which the resolvent grows as $\log(|s|)^b$, $b>0$, where we obtain new decay rates for (possibly) unbounded $C_0$-semigroups in this context. In the Appendix we present the proofs of some of these auxiliary results (with the idea of presenting only the main arguments for the proof of the main results in the bulk of the text).

\section{Preliminaries}\label{SPre}
\zerarcounters
In this section we set the main notation and present some definitions, objects and auxiliary results needed in the next sections.

\subsection{Notation}

We set $\mathbb{N}:=\{1,2,3,\cdots\}$, $\mathbb{C}_{\pm}:=\{z\in \mathbb{C}\mid \text{Re}(\lambda)\gtrless 0\}$. The H\"older conjugate of  $p\in [0,\infty]$ is denoted by $p'$, so that $\frac{1}{p}+\frac{1}{p'}=1$.

Let $\Omega$ be an open connected subset of $\mathbb{C}$; we denote by $H^{\infty}_0(\Omega)$ the set of holomorphic functions $f:\Omega\rightarrow \mathbb{C}$ for which there exist constants $C>0, s>0$ such that for each $z\in\Omega$, $|f(z)|\leq C \min\{|z|^s,|z|^{-s}\}$. %, \ \forall~z\in \Omega$, we denoted by $H^{\infty}_0(\Omega)$.

We denote by $\mathcal{L}(X,Y)$ the (Banach) space of bounded linear operators from $X$ to $Y$ (both $X$ and $Y$ are non-trivial Banach spaces over $\mathbb{C}$), with $\mathcal{L}(X):=\mathcal{L}(X,X)$. 

For a linear operator $A$ defined in $X$, we denote by $\mathcal{D}(A)$, $\Ran(A)$, $\sigma(A)$ and $\rho(A)$ the domain, the range, the spectrum, and the resolvent set of $A$, respectively. We denote by $R(\lambda, A)=(\lambda-A)^{-1}$ %(where wereplace $\lambda I$, with $I$ being the identity operator, by $\lambda$)
the resolvent operator of $A$ at $\lambda \in \rho (A)$.

We denote by $\mathcal{S}(\mathbb{R};X)$ and $\mathcal{S}'(\mathbb{R};X)$ the spaces of $X$-valued Schwarz functions and tempered distributions, respectively. % $\mathbb{R}$ are $\mathcal{S}(\mathbb{R};X)$ and $\mathcal{S}'(\mathbb{R};X)$, respectively.

We say that the Banach space $X$ has Fourier type $p\in [1,2]$ if the Fourier transform $\mathcal{F}:L^p(\mathbb{R}; X)\rightarrow L^{p'}(\mathbb{R}; X)$ is bounded. We then set $\mathcal{F}_{p,X}:= \|\mathcal{F}\|_{\mathcal{L}(L^p(\mathbb{R}; X),L^{p'}(\mathbb{R}; X))}$. A Banach space $X$ has Fourier cotype $q\in [2, \infty]$ if $X$ has Fourier type $q^{\prime}$. Each Banach space has Fourier type $1$, and $X$ has Fourier type $2$ if, and only if, $X$ is isomorphic to a Hilbert space (See Theorem 2.1.18 in~\cite{tuomas}).

We write $f(t)\lesssim g(t)$ to indicate that there exist a positive constant $C$ and $t_0\ge 0$ such that for each $t\ge t_0$, $f(t)\leq C g(t)$. % for all $t$ and a constant $C$ independent of t.

\subsection{Some important classes of functions}

\subsubsection{Complete Bernstein functions}

In this subsection, we recall the definitions and some properties of some special functions that appear throughout the text. We refer to~\cite{ber} for details (see also~\cite{chill}).

\begin{definition}[Definition 1.3 in \cite{ber}]
{\rm A function $f\in C^{\infty}(0,\infty)$ is called {\bf{completely monotone}} if
\begin{equation*}
    (-1)^nf^{(n)}(\lambda) \geq 0 \  \text {for each}  \ n\in  \mathbb{N}\cup\{0\} \  \text{and each}  \ \lambda>0.
\end{equation*}}    
\end{definition}

By Theorem 1.4 in \cite{ber}, which is known as Bernstein's Theorem,  every completely monotone function $f$ is the Laplace transform of a positive Radon measure on $\mathbb{R}_{+}$. Recall that $f\in C^{\infty}(0,\infty)$  is called a {\bf{Bernstein function}} if
\begin{equation*}
    f\geq 0 \  \text{and} \ f' \ \text{is completely monotone}.
\end{equation*}
It is easy to see from this definition that the fractional powers $\lambda\mapsto \lambda^{\alpha}$, with $0\leq \alpha\leq 1$, and $\lambda\mapsto\log(1+\lambda)$, are Bernstein functions. By Lévy-Khintchine Representation Theorem (see Theorem 3.2 in \cite{ber}), a function $f$ is a Bernstein function if, and only if, there exist
constants $a,b\geq 0$ and a positive Radon measure $\mu_{LK}$ (this notation is used in \cite{chill}) defined over the Borel subsets of $(0,\infty)$ such that for each $\lambda>0$,
\begin{eqnarray*}
f(\lambda)= a+b\lambda+ \int_{0^{+}}^{\infty}(1-e^{-\lambda s}) d\mu_{LK}(s), 
\end{eqnarray*}
with
\begin{eqnarray*}
\int_{0^{+}}^{\infty}\frac{s}{s+1}d\mu_{LK}(s)<\infty.
\end{eqnarray*}
The triple $(a,b,\mu_{LK})$ determines $f$ uniquely and vice versa (see Theorem 3.2 in \cite{ber}), and it is called the Lévy-Khintchine triple of $f$. Every Bernstein function can also be extended to a holomorphic function in $\mathbb{C}_+$ (this is Proposition 3.6 in \cite{ber}).

Now we consider a subclass of the Bernstein functions, the so-called complete Bernstein functions.
\begin{definition}[Definition~6.1 in~\cite{ber}]\label{DCBF}
 {\rm A function $f\in C^{\infty}(0,\infty)$ is called
a {\bf{complete Bernstein function}} if it is a Bernstein function and the measure
$\mu_{LK}$ in the Lévy-Khintchine triple has a completely monotone density with
respect to Lebesgue measure. The set of all complete Bernstein functions is denoted by $\mathcal{CBF}$.}
\end{definition}

By Theorem 6.2-(vi) in \cite{ber}, every  $f\in\mathcal{CBF}$ admits a representation of the form
\begin{equation}\label{eqq5}
    f(\lambda)=a+b\lambda+ \int_{0^{+}}^{\infty} \frac{\lambda}{\lambda+s}d\mu(s), \ \ \lambda>0,
\end{equation}
with $a,b\geq 0$ constants and $\mu$ a positive Radon measure defined over the Borel subsets of $(0, \infty)$ that satisfies
\begin{equation*}
  \int_{0^{+}}^{\infty}  \frac{1}{s+1}d\mu(s)<\infty.
\end{equation*}

As discussed in Remark~2.1 in \cite{chill}, complete Bernstein functions admit other representations than the one given by~\eqref{eqq5}. In particular, one has
\begin{equation*}
f(\lambda)=a+\int_{0}^{\infty} \frac{\lambda}{1+\lambda t}d\nu(t)=a+\nu(\{0\})\lambda+\int_{0+}^{\infty} \frac{\lambda}{1+\lambda t}d\nu(t),
\end{equation*}
where $\nu$ is a positive Radon measure defined over the Borel subsets of $(0,\infty)$ that satisfies
\begin{equation*}
    \int_0^\infty\frac{1}{1+t}d\nu(t)<\infty,
\end{equation*}
and the pair $(a,\nu)$ is unique.

The representation formula $\eqref{eqq5}$  is  unique (that is, the triple $(a,b,\mu)$ is unique), and it is called the Stieltjes
representation for $f$ (see Chapter 6 in \cite{ber} for details). Note that
\begin{eqnarray*}
    a=\lim_{\lambda \rightarrow 0^{+}} f(\lambda) \ \ \text{and} \ \ b=\lim_{\lambda \rightarrow 0^{+}} \frac{f(\lambda)}{\lambda}.
\end{eqnarray*}

\begin{example}\label{ex3.1} 
\begin{enumerate}[(a)]
\begin{rm}
\item The function $f:(0,\infty)\rightarrow \mathbb{R}$ given by $f(\lambda)=\lambda^{\alpha}$, with $\alpha \in[0,1]$, is a complete Bernstein function whose Stieltjes representation is given by
\begin{equation*}
f(\lambda)=  \frac{\sin (\alpha \pi)}{\pi}\int_{0+}^{\infty} s^{\alpha} \frac{\lambda}{s+\lambda} \frac{ds}{s}, \qquad\qquad\lambda>0.
\end{equation*}
\item The function $\lambda\mapsto (1+\lambda)^{\alpha}-1$, with $\alpha \in (0,1)$, is a complete Bernstein function whose Stieltjes representation is given by
\begin{equation*}\label{eeq7}
(1+\lambda)^{\alpha}-1=\frac{\sin (\alpha \pi)}{\pi}\int_{0+}^{\infty} (s-1)^{\alpha} \chi_{(1,\infty)}\frac{\lambda}{s+\lambda} \frac{ds}{s}, \qquad\qquad\lambda>0.
\end{equation*}

\item The function $\lambda\mapsto \log(1+\lambda)$ is a complete Bernstein function whose Stieltjes representation is given by% with the Stieltjes representation 
\begin{equation}\label{log(1+s)}
    \log(1+\lambda)= \int_{0+}^{\infty} \chi_{(1,\infty)}(s)\frac{\lambda}{\lambda+s} \frac{ds}{s}, \qquad \qquad \lambda>0.
\end{equation}
\item The function $\lambda\mapsto  \dfrac{\lambda-1}{\log(\lambda)}$ is a complete Bernstein function  whose Stieltjes representation is given by%with the Stieltjes representation 
\begin{equation}\label{loginver}
    \frac{\lambda-1}{\log(\lambda)}=\int_{0+}^{\infty} \frac{s+1}{s(\pi^2+\log(s)^2)}\frac{\lambda}{\lambda+s}ds, \qquad \qquad \lambda>0.
\end{equation}
\item The function $\lambda\mapsto  \dfrac{\lambda\log(\lambda)-\lambda+1}{\log(\lambda)^2}$ is a complete Bernstein function  whose Stieltjes representation is given by
\begin{equation}\label{logquad}
\dfrac{\lambda\log(\lambda)-\lambda+1}{\log(\lambda)^2}=\int_{0+}^{\infty} \frac{\pi^2-2(1+1/s)\log(s)+\log(s)^2}{(\pi^2+\log(s)^2)^2}\frac{\lambda}{\lambda+s}ds, \qquad \qquad \lambda>0.
\end{equation}
\item The function $\lambda\mapsto  \dfrac{(-2 + 2\lambda - 2\lambda\log(\lambda)+\lambda \log(\lambda)^2)}{\log(\lambda)^3}$ is a complete Bernstein function  whose Stieltjes representation is given by
\begin{equation}\label{cubo}
    \frac{(-2 + 2\lambda - 2\lambda\log(\lambda)+\lambda \log(\lambda)^2)}{\log(\lambda)^3}=\int_{0+}^{\infty} \frac{f(s)}{s(\pi^2+\log(s)^2)^3} \frac{\lambda}{(\lambda+s)}ds, \qquad \qquad \lambda>0,
\end{equation}
where $f(s)=\pi^2((-2+\pi^2)s-2)+s\log(s)^4-4s\log(s)^3+2((3+\pi^2)s+3)\log(s)^2-4\pi^2s\log(s).$

\end{rm}
\end{enumerate}
\end{example}

The next results play an important role in the proofs of Theorems~\ref{theo4.5} and~\ref{teo4.5}; items (a) and (b) are Theorem 2.2 in \cite{chill}, and item (c) is Proposition~7.13 in \cite{ber}. 

\begin{proposition}%[Theorem 2.2 in \cite{chill}]
  \label{theor2.3}
\begin{rm}
Let $f,g:(0,\infty)\rightarrow \mathbb{R}$ be non-zero functions.
\begin{enumerate}[(a)]
\item If $f\in\mathcal{CBF}$, then $\dfrac{\lambda}{f(\lambda)},\lambda f\left(\dfrac{1}{\lambda}\right)\in\mathcal{CBF}$. Conversely, if $\dfrac{\lambda}{f(\lambda)}\in\mathcal{CBF}$ or $\lambda f\left(\dfrac{1}{\lambda}\right)\in\mathcal{CBF}$, then $f\in\mathcal{CBF}$.
\item If $f,g\in\mathcal{CBF}$, then $g\circ f\in\mathcal{CBF}$.
 %   \end{enumerate}
  %  \end{rm}
%\end{proposition}
%
%Like Theorem~\ref{theor2.3}, the next result plays an important role in the proofs of our main results.  %For more details see Proposition 7.13 in \cite{ber}.
%\begin{theorem}[\textbf{Proposition 7.13} in \cite{ber}]\label{propor3.1}
%\begin{rm}
\item Let $a_1,a_2\in (0,1)$ be such that $a_1+a_2\leq 1$. Then, for each $f,g\in \mathcal{CBF}$, one has $f^{a_1}\cdot g^{a_2} \in \mathcal{CBF}$.
\end{enumerate}
\end{rm}
\end{proposition}

\subsubsection{Slowly varying functions} %In this subsection we briefly review a class of functions that will help us in our income statement. For more details we suggest~\cite{chill}(see Subsection 2.2) and \cite{bingham}(see Chapter 1).}}

\begin{definition}
\begin{rm}
 Let $a\in\mathbb{R}$ and let $\ell:[a,\infty)\rightarrow \mathbb{R}$ be a strictly positive measurable function %defined on $[a,\infty)$ 
 such that for each $\lambda>0$,
\begin{equation*}
    \lim_{s\rightarrow \infty} \frac{\ell(\lambda s)}{\ell(s)}=1.
\end{equation*}
Then, $\ell$ is said to be \textbf{slowly varying}. 
\end{rm}
\end{definition}

\begin{example}
\begin{rm}
    \begin{enumerate}[(a)]
\item The function $s\mapsto \log(1+s)$ is a slowly varying function.
\item If $\ell$ is
a slowly varying function, then the following ones are also slowly varying functions:
$s\mapsto \ell(s)^{\alpha}$, with $\alpha \in \mathbb{R}$; $s\mapsto \ell(s)\log(s)$.

\end{enumerate}
\end{rm}
    
\end{example}

The next result also plays an important role in the proof of Theorems~\ref{theo4.5},~\ref{teo4.5} and~\ref{theor4.7}.

\begin{proposition}[{Corollary 2.8-(a)} in $\cite{chill}$]\label{cor2}
\begin{rm}
Let $\ell$ be a slowly varying function and let $\gamma>0$. Then, there are positive constants $C, c$ such that for each sufficiently large $s, t$ with $t\geq s$, 
\begin{equation*}
   c\left(\frac{s}{t}\right)^{\gamma}\leq \frac{\ell(t)}{\ell(s)}\leq C\left(\frac{t}{s}\right)^{\gamma}.
\end{equation*}
%for all sufficiently large $s, t$ with $t\geq s$.
\end{rm}

\end{proposition}

\subsection{Functional Calculus}

%Here, we recall some basic properties of the functional calculus that will help us in the proofs of our results.

We begin with the 
\textit{Riesz-Dunford functional calculus of bounded operators}: for each $A \in \mathcal{L}(X)$, let $U$ be an open connected of $\sigma(A)$, let $\gamma$ be a path in $U$ around $\sigma(A)$ and let $f$ be a complex function whose restriction to $U$ is holomorphic; then, one may define the bounded linear operator $f(A):X\rightarrow X$ by the law % we have a functional calculus
\begin{equation}\label{eqq14}
f(A):= \frac{1}{2\pi i} \int_{\gamma} f(z)R(z,A) dz.
\end{equation}
%where $f$ is holomorphic in a neighbourhood $U$ of $\sigma(A)$ and $\gamma$ is a contour in $U$ around $\sigma(A)$.

\subsubsection{Sectorial operators}

For each $\omega\in (0,\pi)$, set $S_{\omega}:= \{z\in \mathbb{C}\mid 0<|\text{arg}(z)|<\omega\}$; set also $S_0:=(0,\infty)$.

\begin{definition}
  \rm{A linear operator $A:\mathcal{D}(A)\subset X\rightarrow X$ is called {\bf{sectorial}} of angle $\omega$ if $\sigma(A)\subset \overline{S_{\omega}}$ and $M(A, \omega):=\sup\{\|\lambda R(\lambda,A)\|_{\mathcal{L}(X)} \mid \lambda \in \mathbb{C}\setminus \overline{S_{\omega'}}, \  \omega'\in (\omega, \pi)\}<\infty$. One denotes this by $A\in \text{Sect}_X(\omega)$.}
\end{definition}  

Set $\omega_A:=\min\{\omega\in (0,\pi)\mid A\in \text{Sect}_X(\omega)\}$, which is the minimal angle for which $A$ is sectorial.  For the required background on sectoral operators, we refer
to \cite{haase}. 
\begin{remark}
\begin{rm}
Let $A:\mathcal{D}(A)\subset X\rightarrow X$ be a linear operator for which $(-\infty,0) \subset \rho(A)$ and
\begin{equation*}
    M_A:= M(A,\pi)=\sup_{t>0}t\|(t+A)^{-1}\|_{\mathcal{L}(X)}<\infty;
\end{equation*}
then, it follows that $A\in \text{Sect}_X(\pi-\arcsin{\left(1/M_A\right))}$. 
    \end{rm}
\end{remark}

The following result presents some useful properties of sectorial operators. For more details, see \cite{chill,haase}.
\begin{lemma}\label{lemma21}
\begin{rm}
Let $A\in\text{Sect}_X(\omega_A)$. Then, 
\begin{enumerate}[(a)]
\item $(1+A)^{-1},A(1+A)^{-1}\in  \text{Sect}_X(\omega_A)$. If $A$ is injective, then $A^{-1}\in \text{Sect}_X(\omega_A)$, and the identity
\begin{equation}\label{eqq2.7}
\lambda(\lambda+A^{-1})^{-1}=1-\frac{1}{\lambda}\left(\frac{1}{\lambda}+A\right)^{-1},
\end{equation}
holds for each $0\neq \lambda \in \mathbb{C}$ .
\item Let $\sigma\in (0,1)$ and set $A_\sigma:=(A+\sigma)(1+\sigma A)^{-1}\in \mathcal{L}(X)$; then, $A_\sigma$ is a sectorial operator, $\displaystyle{\sup_{\sigma \in (0,1)} M_{A_{\sigma}}<\infty}$, and for each $\lambda \in \rho (A)$, $R(\lambda,A_\sigma)$ converges to $R(\lambda,A)$ in $\mathcal{L}(X)$ as $\sigma\rightarrow 0^+$.
\end{enumerate}
\end{rm}
\end{lemma}

Now we recall some basic properties of the functional calculus of sectorial
operators based on complete Bernstein functions. We use \cite{chill} as a reference in our discussion (see also~\cite{gomilko,batty1,batty33}).

%Let $f\in\mathcal{CBF}$, with Stieltjes representation $(a,b,\mu)$.

\begin{definition}[Definition 3.3 in \cite{chill}]\label{Dchill}
\begin{rm}
Let $A\in\text{Sect}_X(\omega_A)$ be densely defined and let $f\in\mathcal{CBF}$, with Stieltjes representation $(a,b,\mu)$. One defines the linear operator $f_0(A):\mathcal{D}(A)\rightarrow X$  by the law
\begin{equation}\label{cbf}
f_0(A)=ax+bAx+\int_{0+}^{\infty} A(A+\lambda)^{-1}x d\mu(\lambda), \ \ \ x\in \mathcal{D}(A).
\end{equation}
Set $f(A):= \overline{f_0(A)}$. We call the linear operator $f(A)$ a complete Bernstein function of $A$. 
\end{rm}
\end{definition}

\begin{theorem}[Theorem 3.6 in \cite{chill}]\label{The2.3}
\begin{rm}
 Let $A$ be a sectorial operator on a Banach space $X$ and let
$f\in\mathcal{CBF}$. Then, $f(A)$ is sectorial.
\end{rm}
\end{theorem}

We consider now the situation where $A\in \text{Sect}_X(\omega_A)$, $\varphi\in (\omega_A,\pi)$ and $f\in H^{\infty}_0(S_{\omega_A})$. Set $f(A)\in \mathcal{L}(X)$ given by
\begin{equation}\label{eq3}
f(A):=\frac{1}{2\pi i} \int_{\Gamma_{\omega'}}f(z)R(z,A) dz,
\end{equation}
where $\Gamma_{\omega'}$ stands for the positively oriented boundary of $S_{\omega'}$ for $\omega'\in (\omega_A,\varphi)$. A standard argument using Cauchy’s Integral Theorem shows that this definition is actually independent of $\omega'$. An interesting reference for this Functional Calculus and its applications is \cite{haase}.

\begin{remark}\label{remark2.2}
{\rm Let $\alpha, \beta>0$, $\upsilon_1,\upsilon_2\ge 0$, $\varphi\in (0,\pi)$ and
\[f_{\alpha,\beta,\upsilon_1,\upsilon_2}(z)=\dfrac{z^{\alpha}}{(1+z)^{\alpha+\beta}\log(2+z)^{\upsilon_1}(2\pi-i\log(z))^{\upsilon_2}},\qquad z\in S_\varphi;\] it is straightforward to show that $f_{\alpha,\beta,\upsilon_1,\upsilon_2}\in H^{\infty}_0(S_{\varphi})$. Therefore,  by $\eqref{eq3}$, one may define
\begin{eqnarray}\label{eq1}
  \nonumber f_{\alpha,\beta,\upsilon_1,\upsilon_2}(A)&:=&\frac{1}{2\pi i} \int_{\Gamma_{\omega'}}f_{\alpha,\beta,\upsilon_1,\upsilon_2}(z)R(z,A) dz\\
  &=&\frac{1}{2\pi i} \int_{\Gamma_{\omega'}}\frac{z^{\alpha}}{(1+z)^{\alpha+\beta}\log(2+z)^{\upsilon_1}(2\pi-i\log(z))^{\upsilon_2}}R(z,A) dz,
\end{eqnarray}
where $\Gamma_{\omega'}$ is the positively oriented boundary of $S_{\omega'}$ for $\omega'\in (\omega_A,\varphi)$. If $A$ is invertible, then one may let $\alpha=0$ in the expression $\eqref{eq1}$. This operator will play an important role in the proofs of Propositions~\ref{prop3.1},~\ref{prop3.2} and~\ref{prop3.3}.}
\end{remark}

\subsubsection{Logarithm operator}

Given the nature of our problem, an investigation involving the definition of the logarithm of an injective sectorial operator is required. Such operator was first defined by Nollau~\cite{nollau} and was subsequently studied by Okazawa~\cite{okaz} and Haase~\cite{haas}. 

Let $A$ be an injective operator over the Banach space $X$ such that $A\in\text{Sect}_X(\omega_A)$. Let $\varphi\in (\omega_A,\pi)$ and set $\tau(z):=z(1+z)^{-2}$; note that $\tau \in H^{\infty}_0(S_{\varphi})$  and $\tau(A)=A(1+A)^{-2}$, by relation~\eqref{eq1} (with $\upsilon_1=\upsilon_2=0$, $\alpha=1$ and $\beta=1$). Set $\mathcal{B}(S_\varphi):=\{f:S_\varphi\rightarrow \mathbb{C} \mid \exists~n\in \mathbb{N}~\text{such that}~ \tau^{n}f\in H^{\infty}_0(S_{\varphi})\}$. Since $A$ is injective, $\tau(A)$ is also injective, and so one may define for each $f\in \mathcal{B}(S_\varphi)$
\begin{equation}\label{eqq17a}
f(A):=(\tau(A)^{-1})^{n}[(\tau^{n}(z)f(z))](A),    
\end{equation}
 with $n$ large enough so that $\tau^{n}f\in H^{\infty}_0(S_{\varphi})$. 
\begin{remark}
\begin{rm}
Definition~\eqref{eqq17a} is independent of the choice of $n$ (see Proposition 2.1 in~\cite{haas}). Note that $f(A)$ is a closed operator with domain $\mathcal{D}(f(A))=\{x\in X \mid (\tau^n(z)f(z))(A)x\in \mathcal{D}(\tau(A)^{-1})^{n}\}$. We refer to~\cite{haas} for more details.
\end{rm}
\end{remark}

\begin{definition}[\textbf{Haase},~\cite{haas}] 
\begin{rm}\label{haase}
Let $A\in\text{Sect}_X(\omega_A)$ and injective. Let $f:S_\pi\rightarrow \mathbb{C}$ be given by the law $f(z)=\log(z)$. Since $f\in \mathcal{B}(S_\varphi)$, then
\begin{equation}\label{eqq16}
    \log(A):=f(A).
\end{equation}    
\end{rm}    
\end{definition}

\begin{remark}\label{remark2.1}
\begin{rm}
Let $A\in\text{Sect}_X(\omega_A)$ be densely defined. It follows from~$\eqref{log(1+s)}$ and from Definition~\ref{Dchill} that for each $x\in \mathcal{D}(A)$,
\begin{equation}  \label{eq6} 
\log(1+A)x=\int_{1}^{\infty} A(A+t)^{-1}x \frac{dt}{t}.
\end{equation}
This representation for $\log(1+A)$ was presented for the first time in \cite{nollau}. 
\end{rm}
\end{remark}

\begin{definition}[{\bf{Okazawa}}, see \cite{okaz}]\label{oka} 
{\rm Let $A\in \text{Sect}_X(\omega_A)$ and injective. Suppose that $\mathcal{D}(A)$ and $\Ran(A)$ are dense in $X$. Then, $\log (A)$ is defined as the closure of 
\begin{equation*}
    \log(1+A)-\log(1+A^{-1}).
\end{equation*}}    
\end{definition}

\begin{remark} {\rm Naturally, the Definitions~\ref{haase} and~\ref{oka} for $\log (A)$ when $A$ is an injective operator must coincide when $D(A)$ and $R(A)$ are both dense; for details, see~\cite{clark}.} %such that $A\in \text{Sect}_X(\omega_A)$ and injective 
\end{remark}

The following result is a direct consequence of Definition~\ref{oka}. %By Definition above the result follows
\begin{lemma}\label{lemma2.2}
{\rm{Let $A\in \text{Sect}_X(\omega_A)$ be injective and densely defined (with not necessarily dense range). Then, for each $x\in \mathcal{D}(A)\cap \Ran(A)$}},
\begin{equation*}
    \log(A)x=\log(1+A)x-\log(1+A^{-1})x.
\end{equation*}
\end{lemma}
\begin{proof} Let $\sigma\in (0,1)$ and set $A_\sigma:=(A+\sigma)(1+\sigma A)^{-1}\in\mathcal{L}(X)$. It follows from  Definition \ref{oka} that for each $x\in X$, and in particular, for each $\mathcal{D}(A)\cap \Ran(A)$,
\begin{eqnarray}\label{eqlemaa2.1}  \log(A_\sigma)x=\log(1+A_\sigma)x-\log(1+A^{-1}_\sigma)x.
\end{eqnarray}
Now, it follows from Lemma 3(c) in \cite{nollau} that for each $x\in \mathcal{D}(A)\cap \Ran(A)$, $\displaystyle{\lim_{\sigma\to 0^{+}} \log(A_\sigma)x=\log(A)x}$ and $\displaystyle{\lim_{\sigma\to 0^{+}} \log(1+A_\sigma)x=\log(1+A)x}$; thus, by $\eqref{eqlemaa2.1}$, one has for each  $x\in \mathcal{D}(A)\cap \Ran(A)$ that %the limit
\begin{equation*}\label{eqq19}
    \lim_{\sigma\to 0^{+}} \log(1+A^{-1}_\sigma)x=\lim_{\sigma\to 0^{+}} \log(1+A_\sigma)x-\lim_{\sigma\to 0^{+}} \log(A_\sigma)x=\log(1+A)x-\log(A)x.
\end{equation*}%\begin{equation*}
 %   \lim_{\sigma \to 0^{+}} \log(1+A^{-1}_\sigma)x,
%\end{equation*}
%with $x\in \mathcal{D}(A)\cap \Ran(A)$, exists.
%Note that $\log(1+A^{-1})$ is well defined in $\mathcal{D}(A)\cap \Ran(A)$ by $\eqref{eqq16}$ and that
%\begin{equation}\label{eqq19}
 %   \lim_{\sigma\to 0^{+}} \log(1+A^{-1}_\sigma)x=\lim_{\sigma\to 0^{+}} \log(1+A_\sigma)x-\lim_{\sigma\to 0^{+}} \log(A_\sigma)x=\log(1+A)x-\log(A)x.
%\end{equation}

Since $A$, $A+1$, $A^{-1}+1$ are sectorial operators, $\log(1+A^{-1})$ is well-defined by~\eqref{eqq16}; thus, it follows from Proposition 3.1.3 in~\cite{clark} and Lemma 3.1 in \cite{haas} that for each $x\in\mathcal{D}(A)\cap \Ran(A)$,
\begin{equation*}
    \log(1+A)x-\log(A)x=\log(1+A)x+\log(A^{-1})x=\log((1+A)A^{-1})x=\log(1+A^{-1})x.
\end{equation*}
Then, it follows from the previous relations %, %$\eqref{eqq19}$, it follows 
that for each $x\in \mathcal{D}(A)\cap \Ran(A)$,
\begin{equation*}
   \lim_{\sigma\to 0^{+}} \log(1+A^{-1}_\sigma)x= \log(1+A^{-1})x,
\end{equation*}
and so, for each $x\in \mathcal{D}(A)\cap \Ran(A)$, one gets
\begin{equation*}
    \log(A)x=\log(1+A)x-\log(1+A^{-1})x.
\end{equation*}
\end{proof}

Let us now recall
 some properties of the logarithm and fractional power.

\begin{lemma}\label{lemma2.3}
  {\rm Let $A \in \text{Sect}_X(\omega_A)$. Then, the following assertions hold:
\begin{enumerate}[(a)]
\item $A^{\sigma}$ is sectorial, with $\sigma \in (0,1)$.
\item If $A\in \mathcal{L}(X)$, then for each $\sigma>0$, $A^{\sigma} \in \mathcal{L}(X)$.
\item If $A$ is injective, then for each $\sigma \in [0,1]$, $\log(A^{\sigma})=\sigma \log(A)$. %, $\forall~\sigma \in [0,1]$.
\item Let $(T(t))_{t\geq 0}$ be a $C_0$-semigroup on the Banach space $X$, with $-A$ its infinitesimal generator. Let, for each $\varepsilon\in(0,1)$, $f_{\varepsilon}(A)=(1+A)^{\varepsilon}-1$. Then, for each $t,s\geq 0$,
\begin{equation*}
T(t)f_{\varepsilon}(A)(s+f_{\varepsilon}(A))^{-1}=f_{\varepsilon}(A)(s+f_{\varepsilon}(A))^{-1}T(t).
\end{equation*}
%where $f_{\varepsilon}(A)=(1+A)^{\varepsilon}-1$, $\forall~\varepsilon\in (0,1)$.
\end{enumerate}
}
\end{lemma}
%\noindent \textbf{Proof:} 
\begin{proof}
(a) Given that %Note that
  for each $\sigma \in (0,1)$, %the function $f_{\sigma}(s)=
  $[s\mapsto s^{\sigma}]\in\mathcal{CBF}$ (see Example~\ref{ex3.1}-(a)), it follows from Theorem~\ref{The2.3} that the operator $f_{\sigma}(A)=A^{\sigma}$ is sectorial. (b) This is Proposition~3.1.1~(a) in \cite{haase}. (c)  This is Satz~5 in~\cite{nollau}. (d) It follows from Theorem~3.9~(a) in~\cite{chill} that for each $t\geq 0$, $T(t)f_{\varepsilon}(A)\subset f_{\varepsilon}(A) T(t)$, and so, by Proposition B.3 in \cite{valued}, one has $T(t)(s+f_{\varepsilon}(A))^{-1}=(s+f_{\varepsilon}(A))^{-1}T(t)$ for each $s,t\geq 0$.    
\end{proof}

\subsection{Fourier Multipliers and Stability for $C_0$-Semigroups}

\subsubsection{Growth at infinity}

Let $X$ and $Y$ be Banach spaces and let $m: \mathbb{R} \rightarrow \mathcal{L}(X,Y)$ be a $X$-strongly measurable map (i.e. the map $\xi\mapsto m(\xi)x$ is a strongly measurable $Y$-valued map for every $x\in X$). One says that $m$ is of {\it{moderate growth at infinity}} if there exist $\beta \geq 0$ and $g\in L^{1}(\mathbb{R})$ such that for each~$\xi\in \mathbb{R}$,
\begin{equation*}
    \frac{1}{(1+|\xi|)^{\beta}}\|m(\xi)\|_{\mathcal{L}(X,Y)}\lesssim g(\xi).
\end{equation*}

For such measurable $m$, one defines the {\bf{\it{Fourier multiplier operator}}} associated with $m$, $T_m:\mathcal{S}(\mathbb{R};X)\rightarrow \mathcal{S}'(\mathbb{R};Y)$, by the law
\begin{equation*}
   T_m(f):= \mathcal{F}^{-1}(m \cdot \mathcal{F}{f}), \qquad \qquad \forall~f\in \mathcal{S}(\mathbb{R};X);
\end{equation*}
$m$ is called the {\it{symbol}} of $T_m$. For $p\in [1,\infty)$ and $q\in [1, \infty]$, let $\mathcal{M}_{p,q}(\mathbb{R};\mathcal{L}(X,Y))$ denote the set of all $X$-strongly measurable maps $m: \mathbb{R}\rightarrow \mathcal{L}(X,Y)$ of moderate growth such that $T_m\in \mathcal{L}(L^{p}(\mathbb{R};X),L^{q}(\mathbb{R};Y))$  and $\|m\|_{\mathcal{M}_{p,q}(\mathbb{R};\mathcal{L}(X,Y))}:=\|T_m\|_{\mathcal{L}(L^{p}(\mathbb{R};X),L^{q}(\mathbb{R};Y))}$.

\subsubsection{Growth at zero and infinity}

Let $\Dot{\mathcal{S}}(\mathbb{R},X):=\{f\in\mathcal{S}(\mathbb{R};X)\mid \hat{f}^{(k)}(0)=0$ for each $k\in \mathbb{N}\cup\{0\}\}$ and $m:\mathbb{R}\setminus\{0\}\rightarrow \mathcal{L}(X,Y)$ be a $X$-strongly measurable map. One says that $m$ is of {\it{moderate growth at zero and infinity}} if there exist $\alpha \geq 0$ and $g\in L^{1}(\mathbb{R})$ such that for each~$\xi\in \mathbb{R}$,
\begin{equation*}
    \frac{|\xi|^{\alpha}}{(1+|\xi|)^{2\alpha}}\|m(\xi)\|_{\mathcal{L}(X,Y)}\lesssim g(\xi).
\end{equation*}
For such measurable $m$, one defines the {\bf{\it{Fourier multiplier operator}}} associated with $m$, $\Dot{T}_m:\Dot{\mathcal{S}}(\mathbb{R};X)\rightarrow \Dot{\mathcal{S}}'(\mathbb{R};Y)$, by the law
\begin{equation*}
  \Dot{T}_m(f):= \mathcal{F}^{-1}(m \cdot \mathcal{F}{f}), \qquad \qquad \forall~f\in \Dot{\mathcal{S}}(\mathbb{R};X);
\end{equation*}
$m$ is called the {\it{symbol}} of $T_m$. For $p\in [1,\infty)$ and $q\in [1, \infty]$, let $\mathcal{M}_{p,q}(\mathbb{R};\mathcal{L}(X,Y))$ denote the set of all $X$-strongly measurable maps $m: \mathbb{R}\setminus\{0\}\rightarrow \mathcal{L}(X,Y)$ of moderate growth such that $T_m\in \mathcal{L}(L^{p}(\mathbb{R};X),L^{q}(\mathbb{R};Y))$  and $\|m\|_{\mathcal{M}_{p,q}(\mathbb{R};\mathcal{L}(X,Y))}:=\|T_m\|_{\mathcal{L}(L^{p}(\mathbb{R};X),L^{q}(\mathbb{R};Y))}$. For more details about discussion above, see~\cite{Veraar}.

The next result will be used in the proofs of Theorems~\ref{theo4.5} and~\ref{teo4.5}. For more details, see~\cite{Veraar}.
\begin{proposition}[Proposition 3.3 in~\cite{Veraar}]\label{prop2.1}
\begin{rm}
Let $X$ be a Banach space with Fourier type $p\in [1, 2]$, let $Y$ be a Banach space with Fourier cotype $q\in [2, \infty]$, and let $r\in [1, \infty]$ be such that $\frac{1}{r}=\frac{1}{p}-\frac{1}{q}$. Let $m:\mathbb{R}\setminus\{0\}\rightarrow \mathcal{L}(X,Y)$ (or $m:\mathbb{R}\rightarrow \mathcal{L}(X,Y)$) be an $X$-strongly measurable map such that
$\|m(\cdot)\|_{\mathcal{L}(X,Y)}\in L^{r}(\mathbb{R})$. Then, $\mathcal{M}_{p,q}(\mathbb{R},\mathcal{L}(X,Y))$.
\end{rm}  
\end{proposition}

The theory of $(L^p,L^q)$ Fourier multipliers has proven to be an important tool for the stability theory of $C_0$-semigroups~\cite{Lastushkin1,Veraar,rozendaal,roze, Weis1, Weis2}.  In particular, by using it, Rozendaal and Veraar have obtained the following result that characterizes polynomial stability. We stress that this result is a necessary tool in our analysis. (see also Theorem 5.1 in~\cite{roze}). % (see the proof of Theorem 4.6 in \cite{rozendaal}).

\begin{theorem}[Theorem 4.6 in \cite{rozendaal}]\label{theor4.3}
\begin{rm}
Let $-A$ be the generator of a $C_0$-semigroup $(T(t))_{t\geq 0}$ defined on a Banach space $X$ such that $\overline{\mathbb{C}_{-}}\setminus\{0\}\subset \rho(A)$ and such that there exist $\alpha,\beta \geq 0$ so that $\|(\lambda+A)^{-1}\|_{\mathcal{L}(X)}\lesssim |\lambda|^{-\alpha}(1+|\lambda|)^{\beta}$, with $\text{Re}(\lambda)\leq 0$. Let $Y$ be a Banach space which is continuously embedded in $X$ and suppose that there exists a constant $C\geq 0$ such that, for each $t\geq 0$, $T(t)Y\subset Y$, $\|T(t)\big{|}_Y\|_{\mathcal{L}(Y)}\leq C\|T(t)\|_{\mathcal{L}(X)}$, and  that there exists a dense subspace $Y_0\subset Y$ such that for each $y\in Y_0$, $[t\mapsto t^{n}T(t)y] \in L^{1}([0,\infty), Y)$. Then, the following statements are equivalent:
\begin{enumerate}[a)]
\item $\displaystyle{\sup_{t\geq 0}\left\{t^{n}\|T(t)\|_{\mathcal{L}(Y,X)}\right\}<\infty}$.
\item There exist $\psi \in C^{\infty}_c(\mathbb{R})$, $p\in [1,\infty)$ and $q \in [p,\infty]$ such that for each $k\in\{n-1,n,n+1\}$,
\begin{equation*}
\psi(\cdot)R(i\cdot,A)^{k}\in \mathcal{M}_{1,\infty}(\mathbb{R}; \mathcal{L}(Y,X)) \ \ \ \textrm{and}\ \ \ 
(1-\psi(\cdot))R(i\cdot,A)^{k}\in \mathcal{M}_{p,q}(\mathbb{R}; \mathcal{L}(Y,X)).
\end{equation*}
%for all $k\in\{n-1,n,n+1\}$.
\end{enumerate}
\end{rm}
\end{theorem}

\section{Singularity at Infinity}\label{three}
\zerarcounters

We begin introducing some notation that will be useful throughout this subsection.

Let $\nu,\upsilon\geq 0$ and $A\in \text{Sect}_{X}(\omega_A)$; since $\lambda\mapsto \log(1+\lambda)\in\mathcal{CBF}$ (see Example~\ref{ex3.1}-(b)), it follows from  Theorem~\ref{The2.3} that the operator $\log(2+A)$ is sectorial, and so $(\log(2+A))^{-\upsilon}$ is well-defined and bounded (see definition of fractional powers of sectorial operators in \cite{haase,mart}). Define the operator
\begin{equation*}
\Phi_{\nu}(\upsilon)=\Phi_{\nu}(A,\upsilon):= (1+A)^{-\nu}\log(2+A)^{-\upsilon} \in \mathcal{L}(X),
\end{equation*}
and set $X_{\nu}(\upsilon):=\Ran(\Phi_{\nu}(\upsilon))$. The space $X_{\nu}(\upsilon)$  is a Banach space with respect to the norm 
 \begin{eqnarray*}
\|x\|_{X_{\nu}(\upsilon)}&=&\|x\|+\|\Phi_{\nu}(\upsilon)^{-1}x\|=\|x\|+\|\log(2+A)^{\upsilon}(1+A)^{\nu}x\|, \ \ \ x \in X_{\nu}(\upsilon).
 \end{eqnarray*}
% (given that $\Phi_{\nu}(\upsilon)^{-1}$ is closed).
 
 Note that $\Phi_{\nu}(\upsilon):X\rightarrow  X_{\nu}(\upsilon)$ is  an isomorphism, so for each $T\in \mathcal{L}(X_{\nu}(\upsilon),X)$,
 \begin{equation*}
     \|Tx\|=\|T\Phi_{\nu}(\upsilon)y\|\leq \|T\Phi_{\nu}(\upsilon)\|_{\mathcal{L}(X)}\|y\|\leq \|T\Phi_{\nu}(\upsilon)\|_{\mathcal{L}(X)}\|x\|_{X_{\nu}(\upsilon)} %\ \ \ T\in \mathcal{L}(X_{\nu}(\upsilon),X),
 \end{equation*}
(here, $y:=\Phi_{\nu}(\upsilon)^{-1}x$) and 
 \begin{equation*}
     \|T\Phi_{\nu}(\upsilon)x\|\leq \|T\|_{\mathcal{L}(X_{\nu}(\upsilon),X)}\|\Phi_{\nu}(\upsilon)x\|\leq \|T\|_{\mathcal{L}(X_{\nu}(\upsilon),X)}\|\Phi_{\nu}(\upsilon)\|_{\mathcal{L}(X)}\|x\|;
 \end{equation*}
 therefore, for each $ T\in \mathcal{L}(X_{\nu}(\upsilon),X)$, one has
 \begin{equation}\label{eqq22}
     \|T\|_{\mathcal{L}(X_{\nu}(\upsilon),X)} \leq  \|T\Phi_{\nu}(\upsilon)\|_{\mathcal{L}(X)}\leq \|\Phi_{\nu}(\upsilon)\|_{\mathcal{L}(X)}\|T\|_{\mathcal{L}(X_{\nu}(\upsilon),X)}. %, \ \ \  T\in \mathcal{L}(X_{\nu}(\upsilon),X).
 \end{equation}
 
Note that $\Phi_{\nu}(0)=\Phi_{\nu}(A)$ and $X_{\nu}(0)=X_{\nu}$, where $\Phi_{\nu}(A)$ and $X_{\nu}$ are the objects defined in \cite{rozendaal}. %(NÃO ENTENDI ESSA OBSERVAÇÃO){\color{red}{ERA PARA FAZER REFERÊNCIAS AOS ESPAÇOS USADOS POR ELES}} 

\

In this subsection, we discuss the decay rate of a $C_0$-semigroup whose infinitesimal generator $-A$ is such that $\overline{\mathbb{C}_-}\subset \rho(A)$ and such that there exist $\beta>0$ and $b\ge 0$ so that $\Vert (\lambda +A)^{-1}\Vert_{\mathcal{L}(X)}\lesssim (1+|\lambda|)^{\beta}\log(2+|\lambda|)^{b}$, for $\lambda \in \mathbb{C}$ satisfying $\text{Re}(\lambda)\le 0$.

\begin{theorem}\label{theo4.4}
\begin{rm}
Let $\beta>0$, $b\ge 0$ and $(T(t))_{t\geq 0}$ be a $C_0$-semigroup defined in the Banach space $X$ with Fourier type
$p\in [1,2]$, with $-A$ as its generator. Suppose $\overline{\mathbb{C}_{-}}\subset \rho(A)$ and for each $\lambda \in\mathbb{C}$ with $\text{Re}(\lambda)\le 0$,
\begin{equation}\label{eeqq26}
\|(\lambda+A)^{-1}\|_{\mathcal{L}(X)}\lesssim (1+|\lambda|)^{\beta}\log(2+|\lambda|)^b.
\end{equation}

Let $r\in [1,\infty]$ be such $\frac{1}{r}=\frac{1}{p}-\frac{1}{p'}$ and let $\tau$ be such that $\tau>\beta+\frac{1}{r}$. Then, for each $\delta>0$ and each $\rho \in [0,\frac{\tau-1/r}{\beta}-1]$, there exists $c_{\rho,\delta}>0$ such that for each $t\geq 1$,
\begin{equation}\label{aindanao1}
 \|T(t)(1+A)^{-\tau}\log(2+A)^{-\frac{b}{\beta}(\tau-1/r)-\frac{1+\delta}{r}}\|_{\mathcal{L}(X)}\leq c_{\rho,\delta}t^{-\rho}.
\end{equation}
\end{rm}
\end{theorem}

The following results are needed in the proof of Theorem~\ref{theo4.4}. Note also that the following proposition is a version of Theorem~\ref{theo4.4} in case $p=1$ (that is, in case $X$ is a Banach space with trivial type).

\begin{proposition}\label{theor3.1}
\begin{rm}
Let $b\ge 0$, $\beta>0$ and let $A$ be an injective sectorial operator on a Banach space $X$ such that
$-A$ generates a $C_0$-semigroup $(T(t))_{t\geq 0}$ on $X$. Suppose $\overline{\mathbb{C}_{-}}\subset \rho(A)$ and for each $\lambda \in\mathbb{C}$ with $\text{Re}(\lambda)\le 0$,
\begin{equation}
\|(\lambda+A)^{-1}\|_{\mathcal{L}(X)}\lesssim (1+|\lambda|)^{\beta}\log(2+|\lambda|)^b.
\end{equation}
Let $\tau\geq \beta+1$. Then, for each  $\delta>0$ and each $\rho \in [0,\frac{\tau-1}{\beta}-1]$, there exists $c_{\rho,\delta}>0$ such that for each $t\geq 1$,
\begin{equation*}
 \|T(t)(1+A)^{-\tau}\log(2+A)^{-\frac{b}{\beta}(\tau-1)-1-\delta}\|_{\mathcal{L}(X)}\leq c_{\rho,\delta}t^{-\rho}.
\end{equation*}
\end{rm}
\end{proposition}
\begin{proof}

We follow the same steps of the proof of Proposition 4.3 in~\cite{rozendaal}. The proposition is equivalent to the following
statement: for each $s\geq 0$ and $\delta>0$ there exists $C_{s,\delta}>0$ such that for each $t\geq 1$,
\begin{equation*}
\|T(t)(1+A)^{-\nu}\log(2+A)^{-\upsilon}\|_{\mathcal{L}(X)}\le C_{s,\delta} t^{-s}, 
\end{equation*}
where $\upsilon:=b(s+1)+1+\delta$, $\nu:=(s+1)\beta+1$.

Firstly, we obtain the result for $s=n\in \mathbb{N}\cup \{0\}$ and then for any $s\ge 0$ by an interpolation argument.

So, let $\delta>0$, $n\in \mathbb{N}\cup\{0\}$, $\upsilon=b(n+1)+1+\delta$, $\nu=(n+1)\beta+1$ and $x \in X_{\nu+1}(\upsilon)$. Set 
\begin{eqnarray*}
y:=[\Phi_{\nu}(\upsilon)]^{-1}x=\log(2+A)^{\upsilon}(1+A)^{\nu}x &=& \log(2+A)^{\upsilon}(1+A)^{\nu}\left((1+A)^{-\nu-1}\log(2+A)^{-\upsilon}z\right)\\
&=& \log(2+A)^{\upsilon}\left((1+A)^{-1}\log(2+A)^{-\upsilon}z\right)\\
&=& (1+A)^{-1}z, % \in \mathcal{D}(A),
\end{eqnarray*}
with $z\in X$, and note that $(1+A)^{-1}z\in \mathcal{D}(A)$; here, we have used that $\log(2+A)^{\upsilon}$ commutes with $(1+A)^{-1}$ (for more details, see Proposition~2.3-(d) in \cite{okaz} and Proposition~3.1.1-(f) in \cite{haase}).

Let $g:[0,\infty)\rightarrow X$ be given by
\begin{equation}\label{eqqq27}
g(t)=\frac{1}{2\pi i} \int_{i\infty}^{-i\infty} e^{-\lambda t} \frac{1}{(1+\lambda)^{\nu}[\log(2+\lambda)]^{\upsilon}}R(\lambda,A)yd\lambda,
\end{equation}
and note that for each $t\ge 0$, $g(t)\in X$; namely, for each~$t\geq 0$, one has
\begin{eqnarray*}
\|g(t)\|&\leq&\left\| \frac{1}{2\pi i} \int_{\mathbb{R}} e^{i\xi t} \frac{1}{(1-i\xi)^{\nu}[\log(2-i\xi)]^{\upsilon}} R(-i\xi,A)y d\xi \right\|\\
&\lesssim& \left(\int_{\mathbb{R}} \frac{1}{(1+|\xi|)^{\nu} [\log(2+|\xi|)]^{\upsilon}}\|(i\xi+A)^{-1}\|_{\mathcal{L}(X)}d\xi\right)\|y\|,
\end{eqnarray*}
Now, by assuming~\eqref{eeqq26}, it follows that the integral above is finite.

Moreover, since $y\in \mathcal{D}(A)$, the function $\lambda\mapsto \dfrac{\lambda}{(1+\lambda)^{\nu}(\log(2+\lambda))^{\upsilon}}R(\lambda,A)y$ is integrable and by dominated convergence,
\begin{equation*}
g'(t)=-\frac{1}{2\pi i}\int_{i\infty}^{-i\infty}e^{-\lambda t} \frac{\lambda}{(1+\lambda)^{\nu}[\log(2+\lambda)]^{\upsilon}}R(\lambda,A)yd\lambda,
\end{equation*}
which proves that $g$ is differentiable everywhere. Now, by Lemma~\ref{lemmaB.1},
\begin{eqnarray*}
g'(t)&=& \frac{1}{2\pi i}\int_{i\infty}^{-i\infty}e^{-\lambda t} \frac{1}{(1+\lambda)^{\nu}[\log(2+\lambda)]^{\upsilon}}(-AR(\lambda,A)y-y)d\lambda\\
&=& -A\left(\frac{1}{2\pi i}\int_{i\infty}^{-i\infty}e^{-\lambda t} \frac{1}{(1+\lambda)^{\nu}[\log(2+\lambda)]^{\upsilon}}R(\lambda,A)yd\lambda\right)-\\
&-&\left(\frac{1}{2\pi i}\int_{i\infty}^{-i\infty}e^{-\lambda t} \frac{1}{(1+\lambda)^{\nu}[\log(2+\lambda)]^{\upsilon}}d\lambda \right)y\\
&=& 0-Ag(t)=-Ag(t),
\end{eqnarray*}

and $\displaystyle{g(0)=\frac{1}{2\pi i} \int_{i\infty}^{-i\infty}  \frac{1}{(1+\lambda)^{\nu}[\log(2+\lambda)]^{\upsilon}}R(\lambda,A)yd\lambda=\Phi_{\nu}(\upsilon)y=x}$, by~\eqref{eq1}. Then, $g'(t)=-Ag(t)$ for each $t\geq 0$, and $g(0)=x$. Therefore, for each $t\ge 0$, $g(t)=T(t)x$, by the uniqueness of the Cauchy problem associated with $-A$.

Integration by parts yields
\begin{equation*}
    t^n T(t)x=\frac{1}{2\pi i} \int_{i\mathbb{R}}e^{-\lambda t} p(\lambda,A)y d\lambda,
\end{equation*}
where $p(\lambda,A)$ is a finite linear combination of terms of the form
\begin{equation*}
    \frac{R(\lambda,A)^{n-j+1}}{(1+\lambda)^{\nu+j}(2+\lambda)^{i}[\log(2+\lambda)]^{\upsilon+i}}, \ \ \   \frac{R(\lambda,A)^{n-j+1}}{(1+\lambda)^{\nu+i}(2+\lambda)^{j}[\log(2+\lambda)]^{\upsilon+j}},
\end{equation*}
with $0\leq i<j\leq n$, each one of them being integrable (see the proof of Proposition~4.3 in~\cite{rozendaal} for details).
Then, there exists a positive constant $d_{n,\delta}$ so that for each $t\ge 1$,
\begin{equation*}
   \|t^{n}T(t)x\|\leq \left(\frac{1}{2\pi} \int_{i\mathbb{R}}|e^{-\lambda t}| \|p(\lambda,A)\|_{\mathcal{L}(X)} d\lambda \right)\|y\| \le d_{n,\delta} \|\log(2+A)^{\upsilon}(1+A)^{\nu}x\|\leq d_{n,\delta} \|x\|_{X_{\nu}(\upsilon)}. 
\end{equation*}

%\noindent{\bf{Claim:}} The subspace
%$X_{\nu+1}(\mu)$ is dense in $X_{\nu}(\mu)$.

%\noindent {\textbf{Proof:}} Let $y\in \text{Ran}((1+A)^{-\nu}\log(1+B)^{-\mu})$. Then $y=\log(2+A)^{-\mu}(1+A)^{-\beta}x$. Since $\mathcal{D}(A)$ is dense, there exists $(y_n)_n\in \mathcal{D}(A)$ such that $y_n\rightarrow x$. Define $z_n=\log(2+A)^{-\mu}(1+A)^{-\beta}y_n\in  X_{\nu+1}(\mu)$, so
 %\begin{equation*}
%\|z_n-y\|_{X_{\nu}(\mu)}=\|z_n-y\|+\|\Phi^{-1}_{\nu}(\mu)(z_n-y)\|\rightarrow 0.
% \end{equation*}
Since $X_{\nu+1}(\upsilon)$ is dense in $X_{\nu}(\upsilon)$, it follows from the previous discussion that for each $t\ge 1$,
\begin{equation}\label{eqtheo3.1}
    \|T(t)\|_{\mathcal{L}(X_{\nu}(\upsilon),X)}\le d_{n,\delta} t^{-n}. % \ \ \ (t\geq 1).
\end{equation}
 
It remains to prove the result for any $s\ge 0$. For each $s\ge 0$, let $n\in \mathbb{N}\cup\{0\}$ be such that $n\le s<n+1$. Let also define $\theta:=\theta(s)\in [0,1)$ by the relation $s=(1-\theta)n+\theta(n+1)$.

  Set $a_1:=\frac{\beta}{\beta+b}$ and $a_2:=\frac{b}{\beta+b}$ and note that $a_1+a_2=1$; then, by Proposition~\ref{theor2.3}-(c), $f(\lambda)=(1+\lambda)^{a_1}\log(2+\lambda)^{a_2} \in \mathcal{CBF}$, where $\lambda>0$. Now, by  Lemma~\ref{lemma21}, the operator
\begin{equation*}
(f(A))^{-1}=(1+A)^{-a_1}\log(2+A)^{-a_2},
\end{equation*}
is sectorial, given that $f(A)$ is sectorial (by Theorem~\ref{The2.3}).

Since $(f(A))^{-1}$ is sectorial, it follows from relation~\eqref{eqtheo3.1}, the moment inequality (see Proposition 4.6 in \cite{haase}) and Theorem 2.4.2 in \cite{haase} that there exists a positive constant $C_{s,\delta}$ such that for each $t\ge 1$,
\begin{eqnarray*}
\|T(t)[(f(A))^{-1}]^{\theta(\beta+b)}\Phi_{\nu}(\upsilon)\|_{\mathcal{L}(X)} &\lesssim& \|T(t)\Phi_{\nu}(\upsilon)\|^{1-\theta}_{\mathcal{L}(X)}\|T(t)[(f(A))^{-1}]^{\beta+b}\Phi_{\nu}(\upsilon)\|^{\theta}_{\mathcal{L}(X)}\\
&=& \|T(t)\Phi_{\nu}(\upsilon)\|^{1-\theta}_{\mathcal{L}(X)}\|T(t)(1+A)^{-\beta}\log(2+A)^{-b}\Phi_{\nu}(\upsilon)\|^{\theta}_{\mathcal{L}(X)}\\
&=& \|T(t)\Phi_{\nu}(\upsilon)\|^{1-\theta}_{\mathcal{L}(X)}\|T(t)\Phi_{\beta(n+2)+1}(b(n+2)+1+\delta)\|^{\theta}_{\mathcal{L}(X)}\\
&\le& (d_{n,\delta}t^{-n})^{1-\theta}(d_{n+1,\delta}t^{-n-1})^\theta =C_{s,\delta}t^{-s},
\end{eqnarray*}
and we are done.    
\end{proof}

Note that for $b=\zeta=0$, the following result is Proposition~3.4 in \cite{rozendaal} (see also Theorem 5.5 in~\cite{chill}).

\begin{proposition}\label{prop3.1}
\begin{rm}
Let $A\in \text{Sect}_X(\omega_A)$ be such that $\overline{\mathbb{C}_{-}}\subset \rho(A)$, and let $\beta,b,\zeta\geq 0$ and $\beta_0\in[0,1)$. If 
\begin{equation}\label{eeq23}
    \|(i\xi+A)^{-1}\|_{\mathcal{L}(X)}\lesssim (1+|\xi|)^{\beta}\log(2+|\xi|)^{b},
\end{equation}
then the family
\begin{equation}\label{eeq24}
\{|\lambda|^{\beta_0}\log(2+|\lambda|)^{\zeta}\|(\lambda+A)^{-1}\|_{\mathcal{L}(X_{\beta_0+\beta}(\zeta+b),X)}\mid \lambda \in i\mathbb{R}, |\lambda|\geq 1\} 
\end{equation}
is uniformly bounded.
\end{rm}
\end{proposition}

\begin{proof}  We  proceed as in the proof of Proposition~3.4-(2) in~\cite{rozendaal}. Fix $\theta \in (\omega_A,\pi)$ and let the path $\Gamma:=\{re^{i\theta} \mid r\in [0,\infty)\}\cup \{re^{-i\theta} \mid r\in [0,\infty)\}$
be oriented from $\infty e^{i\theta}$ to $\infty e^{-i\theta}$. Set  $\Tilde{c}:=b+\zeta$; since $A+\frac{1}{2}\in\text{Sect}_X(\omega_A)$, it follows from Remark~\ref{remark2.2} (by letting $\alpha=0$, $\upsilon_2=0$) and from Lemma~\ref{lemmaB2} that for each $\lambda\in i\mathbb{R}$, $\vert\lambda\vert\ge 1,$ and for each ~$x\in X$ %and let $\Tilde{c}:=b+\zeta$, then
\begin{eqnarray*}
(\lambda+A)^{-1}(1+A)^{-\beta-\beta_0}\log(2+A)^{-\Tilde{c}}x &=& \frac{1}{2\pi i} \int_{\Gamma} \frac{(\lambda+A)^{-1}}{(\frac{1}{2}+z)^{\beta+\beta_0}\log(\frac{3}{2}+z)^{\Tilde{c}}}R\left(z,A+\frac{1}{2}\right) xdz\\
&=& \frac{1}{2\pi i} \int_{\Gamma} \frac{(\lambda+A)^{-1}}{(\frac{1}{2}+z)^{\beta+\beta_0}\log(\frac{3}{2}+z)^{\Tilde{c}}(z+\lambda-\frac{1}{2})}x dz\\
&+& \frac{1}{2\pi i} \int_{\Gamma} \frac{R\left(z,A+\frac{1}{2}\right)}{(\frac{1}{2}+z)^{\beta+\beta_0}\log(\frac{3}{2}+z)^{\Tilde{c}}(z+\lambda-\frac{1}{2})}x dz\\
&=& \frac{1}{(1-\lambda)^{\beta+\beta_0} \log(2-\lambda)^{\Tilde{c}}}(\lambda+A)^{-1}x+T_\lambda x
\end{eqnarray*}
with
\begin{equation*}
    T_\lambda:=\frac{1}{2\pi i} \int_{\Gamma} \frac{1}{(\frac{1}{2}+z)^{\beta+\beta_0}\log(\frac{3}{2}+z)^{\Tilde{c}}(z+\lambda-\frac{1}{2})}R(z,A+\frac{1}{2}) dz.
\end{equation*}
Let $h_{\beta_0,\zeta}(\lambda):=(1-\lambda)^{\beta_0}\log(2-\lambda)^{\zeta}$, with $\lambda\in i\mathbb{R}$, $\vert\lambda\vert\ge 1$; then, for each $x\in X$
\begin{equation}\label{eeq25}
 h_{\beta_0,\zeta}(\lambda)(\lambda+A)^{-1}(1+A)^{-\beta-\beta_0}\log(2+A)^{-\Tilde{c}}x=\frac{(\lambda+A)^{-1}}{(1-\lambda)^{\beta}\log(2-\lambda)^b}x+h_{\beta_0,\zeta}(\lambda)T_\lambda x.
\end{equation}

Let $\varepsilon\in (0,\beta+\beta_0)$ and note that the function $z\mapsto \dfrac{1}{(z+1/2)^{\varepsilon}}R(z,A+1/2)$ is integrable on $\Gamma$. Note also that, by Lemma~A.1 in \cite{rozendaal}, for each $z\in \Gamma$ and each $\lambda \in i\mathbb{R}$, one has 
\begin{equation}\label{eeq26}
 \frac{1}{|z+\frac{1}{2}|^{\beta+\beta_0-\varepsilon}|\log\left(\frac{3}{2}+z\right)|^{\Tilde{c}}|z+\lambda-\frac{1}{2}|}\lesssim \frac{1}{1+|\lambda|}.
\end{equation}
Therefore, by relations~\eqref{eeq25} and~\eqref{eeq26}, it follows that
\begin{eqnarray*}
 \|h_{\beta_0,\zeta}(\lambda)(\lambda+A)^{-1}(1+A)^{-\beta-\beta_0}\log(2+A)^{-\Tilde{c}}\|_{\mathcal{L}(X)} %&\leq& %\left\|\frac{1}{(1-\lambda)^{\beta}\log(2-\lambda)^b}(\lambda+A)^{-1}\right\|_{\mathcal{L}(X)}\\
 % &+&\|h_{\beta_0,\upsilon}(\lambda)T_\lambda \|_{\mathcal{L}(X)}\\
 &\lesssim& \left\|\frac{1}{(1-\lambda)^{\beta}\log(2-\lambda)^b}(\lambda+A)^{-1}\right\|_{\mathcal{L}(X)}\\
 &+&\frac{|\log(2-\lambda)|^{\zeta}}{(1+|\lambda|)^{1-\beta_0}}.
\end{eqnarray*}

By relation $\eqref{eeq23}$ and since $\displaystyle{\lim_{|\lambda|\to \infty} \dfrac{|\log(2-\lambda)|^{\zeta}}{(1+|\lambda|)^{1-\beta_0}}=0}$ (recall that $\beta_0\in [0,1)$), one concludes that
\begin{equation*}
\{|\lambda|^{\beta_0}\log(2+|\lambda|)^{\zeta}\|(\lambda+A)^{-1}\|_{\mathcal{L}(X_{\beta_0+\beta}(\zeta+b),X)}\mid \lambda \in i\mathbb{R}, |\lambda|\geq 1\} 
\end{equation*}
is uniformly bounded.
\end{proof}

\begin{remark}
\begin{rm}
Note that by relations $\eqref{eeq25}$ and $\eqref{eeq26}$, for each $\lambda \in i\mathbb{R}$,
\begin{eqnarray*}
 \left\|\frac{(\lambda+A)^{-1}}{(1-\lambda)^{\beta}\log(2-\lambda)^b}\right\|_{\mathcal{L}(X)} &\lesssim& \|(1-\lambda)^{\beta_0}\log(2-\lambda)^{\zeta} (\lambda+A)^{-1}(1+A)^{-\beta-\beta_0}\log(2+A)^{-\Tilde{c}}\|_{\mathcal{L}(X)}\\
 &+& \|(1-\lambda)^{\beta_0}\log(2-\lambda)^{\zeta}T_\lambda\|_{\mathcal{L}(X)}\\
 &\lesssim& \|(1-\lambda)^{\beta_0}\log(2-\lambda)^{\zeta} (\lambda+A)^{-1}(1+A)^{-\beta-\beta_0}\log(2+A)^{-\Tilde{c}}\|_{\mathcal{L}(X)}\\
 &+&\frac{|\log(2-\lambda)|^{\zeta}}{(1+|\lambda|)^{1-\beta_0}};
\end{eqnarray*}
thus, by assuming that the condition~\eqref{eeq24} is valid, one gets.
\begin{equation*}
  \left\|\frac{(\lambda+A)^{-1}}{(1-\lambda)^{\beta}\log(2-\lambda)^b}\right\|_{\mathcal{L}(X)}\lesssim C. 
\end{equation*}
This shows that the converse of Proposition~\ref{prop3.1} is also valid. 
\end{rm}
\end{remark}

\begin{proof3}%noindent{\bf{Proof.}}
The case $p=1$ corresponds to Proposition~\ref{theor3.1}. Let $n\in \mathbb{N}\cup\{0\}$ and set $\nu:=(n+1)\beta+\frac{1}{r}$, $\upsilon:=b(n+1)+\frac{1+\delta}{r}$ in case $p\in(1,2)$ ($1< r<\infty$), and $\nu:=(n+1)\beta$, $\upsilon:=b(n+1)$ if $p=2$ (that is, if $r=\infty$). Set also $B:=A+1$. By letting $\beta_0=\zeta=0$ in Proposition~\ref{prop3.1}, it follows that for each $k\in\{1,\ldots,n\}$,
\begin{equation}\label{eq22}
    \sup_{\xi \in \mathbb{R}}\|R(i\xi,A)^{k}\|_{\mathcal{L}(X_{n\beta}(bn),X)}<\infty.
\end{equation}

%Since for each $\xi\in\mathbb{R}$, $R(i\xi,A)$ commutes with $B^{-\nu}\log(1+B)^{-\upsilon}$, (POR QUÊ? CONFIRA A MODIFICAÇÃO QUE EU FIZ NOS EXPOENTES) it follows that for each $x\in X_{\nu}(\upsilon)$, 
%\begin{eqnarray}\label{ETH3.1}
%\nonumber\|R(i\xi,A)x\|_{X_{\beta n}(bn)}&=&\|R(i\xi,A)x\|+\|(\Phi_{\beta n}(bn))^{-1}R(i\xi,A)x\|\\%
%\nonumber&=& \|R(i\xi,A)x\|+\|R(i\xi,A)B^{-\beta-\frac{1}{r}}\log(1+B)^{-b-\frac{1+\delta}{r}}y\|\\
%\nonumber&\leq& \|R(i\xi,A)\Phi_{\nu}(\upsilon)y\|+\|R(i\xi,A)B^{-\beta-\frac{1}{r}}\log(1+B)^{-b-\frac{1+\delta}{r}}y\|\\
%\nonumber&\lesssim& \|R(i\xi,A)B^{-\beta-\frac{1}{r}}\log(1+B)^{-b-\frac{1+\delta}{r}}\|_{\mathcal{L}(X)}\|y\|\\
%&\leq& \|R(i\xi,A)B^{-\beta-\frac{1}{r}}\log(1+B)^{-b-\frac{1+\delta}{r}}\|_{\mathcal{L}(X)}\|x\|_{X_{\nu}(\upsilon)},
%\end{eqnarray}
%with $y:=(\Phi_{\nu}(\upsilon))^{-1}x$.

Let $\delta>0$ and let $h_{r,\delta}:\mathbb{R}\rightarrow\mathbb{R}$ be given by the law $h_{r,\delta}(\xi)=(1+|\xi|)^{\frac{1}{r}}\log(2+|\xi|)^{\frac{1+\delta}{r}}$; then, it follows from %relation~\eqref{ETH3.1} and
Proposition \ref{prop3.1} (by taking $\beta_0=1/r$ and $\zeta=(1+\delta)/r$) that
\begin{equation}\label{eq23}    \sup_{\xi\in\mathbb{R}} h_{r,\delta}(\xi)\Vert R(i\xi,A)B^{-\beta-\frac{1}{r}}\log(1+B)^{-b-\frac{1+\delta}{r}}\Vert_{\mathcal{L}(X)}<\infty. %\mid\xi\in\mathbb{R}\}\subset \mathcal{L}\left(X_{\nu}(\upsilon),X_{\beta n}(bn)\right)
\end{equation}
%is uniformly bounded. Note that
Thus, for each $k\in\{1,\ldots,n+1\}$, it follows from relations~\eqref{eq22} and~\eqref{eq23} that 
\begin{equation*}    
\sup_{\xi\in\mathbb{R}}h_{r,\delta}(\xi)\|R(i\xi,A)^{k}\|_{\mathcal{L}(X_{\nu}(\upsilon),X)}\lesssim \sup_{\xi\in\mathbb{R}} \left(h_{r,\delta}(\xi)\Vert R(i\xi,A)^{k}B^{-\beta(n+1)-\frac{1}{r}}\log(1+B)^{-b(n+1)-\frac{1+\delta}{r}}\Vert_{\mathcal{L}(X)}\right)
\end{equation*}
\begin{eqnarray}\label{eq24} 
\nonumber&&\lesssim \sup_{\xi\in\mathbb{R}} \left(h_{r,\delta}(\xi)\|R(i\xi,A)B^{-\beta-\frac{1}{r}}\log(1+B)^{-b-\frac{1+\delta}{r}}\|_{\mathcal{L}(X)} \|R(i\xi,A)^{k-1}B^{-\beta n}\log(1+B)^{-bn}\|_{\mathcal{L}(X)}\right)\\
\nonumber&&\le \sup_{\xi\in\mathbb{R}} \left(h_{r,\delta}(\xi)\|R(i\xi,A)B^{-\beta-\frac{1}{r}}\log(1+B)^{-b-\frac{1+\delta}{r}}\|_{\mathcal{L}(X)}\right) \sup_{\xi\in\mathbb{R}}\left(\|R(i\xi,A)^{k-1}\|_{\mathcal{L}(X_{\beta n}(bn),X)}\right)<\infty.\\
&& 
\end{eqnarray}
%then, by \eqref{eq22} and \eqref{eq23} we concluded that for $k\in\{1,\cdots,n+1\}$,
%\begin{equation}\label{eq24}
 %   \sup_{\xi \in \mathbb{R}} h_{r,\delta}(\xi)\|R(i\xi,A)^{k}\|_{\mathcal{L}(X_{\nu}(\upsilon),X)} <\infty.
%\end{equation}

It follows from Proposition~\ref{theor3.1} that the space $X_\nu(\upsilon)$ satisfies the conditions presented in the statement of Theorem~\ref{theor4.3}. By proceeding as in the proof of this Theorem~\ref{theoem 1.9} (see Theorem 4.9 in \cite{rozendaal}), let $\psi \in C_{c}(\mathbb{R})$ be such that $\psi \equiv 1$ on $[-1, 1]$. One has, by \eqref{eq24}, that for each $k\in \{1,\ldots, n+1\}$,
\begin{equation*}    \psi(\cdot)R(i\cdot,A)^{k}\in L^{1}\left(\mathbb{R},\mathcal{L}\left(X_{\nu}(\upsilon),X\right)\right)\subset \mathcal{M}_{1,\infty}\left(\mathbb{R},\mathcal{L}\left(X_{\nu}(\upsilon),X\right)\right),
\end{equation*}
and
\begin{equation}\label{SERLR}
  \|(1-\psi(\cdot))R(i\cdot, A)^{k}\|_{\mathcal{L}(X_{\nu}(\upsilon),X)} \in L^{r}(\mathbb{R}).
\end{equation}
 Note that $X_{\nu}\left(\upsilon\right)$ has Fourier type $p$, since $X_{\nu}(\upsilon)$ is isomorphic to $X$. Then, by Proposition~\ref{prop2.1} and by~\eqref{SERLR}, one concludes that for each $k\in \{1,\ldots, n+1\}$,
\begin{equation*}
    (1-\psi(\cdot))R(i\cdot, A)^{k} \in \mathcal{M}_{p,p'}\left(\mathbb{R},\mathcal{L}(X_{\nu}(\upsilon),X)\right).
\end{equation*}

Now, by Theorem~\ref{theor4.3}, for each $n\in\mathbb{N}\cup\{0\}$ there exists $c_n\geq 0$ such that for each $t\ge 1$,
\begin{equation}\label{eqq23}
\|T(t)(1+A)^{-\nu}\log(2+A)^{-b(n+1)-\frac{1+\delta}{r}}\|_{\mathcal{L}(X)}\leq c_n t^{-n}.
\end{equation}

 Let $s\geq 0$, $\nu=\beta(s+1)+1/r$ and let $n\in \mathbb{N}\cup\{0\}$ be such that $n\le s<n+1$. Let $\theta\in[0,1)$ be such that $s=(1-\theta)n+\theta(n+1)$. Then, by following the same arguments presented in the proof of Proposition~\ref{theor3.1}, it follows that for each $t\ge1$,
\begin{eqnarray*}
\|T(t)B^{-\nu}\log(2+A)^{-b(s+1)-\frac{1+\delta}{r}}\|_{\mathcal{L}(X)}
&\lesssim& t^{-s}.
\end{eqnarray*}
\end{proof3}

\begin{remark}[]
\begin{rm}
 Let $A$ be a linear operator defined in a Banach space $X$, not necessarily sectorial, such that
\begin{enumerate}
\item $-A$ generates a $C_0$-semigroup on $X$;
\item $\overline{\mathbb{C}_{-}}\subset\rho(A)$ and $\|(\lambda+A)^{-1}\|_{\mathcal{L}(X)}\lesssim (1+|\lambda|)^{\beta}\log(2+|\lambda|)^{b}$, for each $\lambda \in \overline{\mathbb{C}_{-}}$. 
\end{enumerate}
Under the above assumptions, note that for each $\varepsilon>0$, $A+\varepsilon$ is sectorial. Then, the operator
\begin{equation*}
    (2+A)^{-\beta}\log(3+A)^{-b}=(1+1+A)^{-\beta}\log(2+1+A)^{-b}
\end{equation*}
is well-defined through the sectorial functional calculus for $A+1$. Note that previous results are still valid. So, in this context, we are able to remove the hypothesis of sectorially of $A$ (see Theorem~\ref{theo4.4}).
\end{rm}
\end{remark}

\begin{lemma}
\begin{rm}
Let $-A$ be the generator of a $C_0$-semigroup $(T(t))_{t\geq 0}$ on a Banach
space $X$. Suppose that there exist $\beta>0$, $\delta\in [0,1)$, $\eta\in \rho(-A)$ that such $1\not \in \sigma(A+\eta)$, and a sequence $(t_n)_{n\in \mathbb{N}}\subset [0,\infty)$ such that $t_n\rightarrow \infty$ and
\begin{equation}\label{eq45a}  
\lim_{n\rightarrow \infty}\|T(t_n)(\eta+A)^{-\beta}[\log(A+\eta)]^{-\delta}\|_{\mathcal{L}(X)}=0.
\end{equation}
Then, $\overline{\mathbb{C}_{-}}\subset \rho(A)$.
\end{rm}
\end{lemma}

\begin{proof}
  We begin with the following remarks:
  \begin{itemize}
  \item One may let $\eta\in \mathbb{R}$ be such that $\eta-1>\omega_0(T)$; namely, it follows from the definition of $\omega_0(T)$ that $1-\eta \in \rho(A)$ (see also~\cite{Neerven}).
  \item  Since $\eta-1>\omega_0(T)$,  $\eta+A$ is sectorial, and so $(\eta+A)^{-\beta}$, $\log(A+\eta)^{-\delta}$ are well-defined.
  \item One has for each $t>0$, 
\begin{equation}\label{maisumpouco}
\|T(t)(\eta+A)^{-\beta}[\log(A+\eta)]^{-2}\|_{\mathcal{L}(X)}\leq\|[\log(A+\eta)]^{-2+\delta}\|_{\mathcal{L}(X)} \|T(t)(\eta+A)^{-\beta}[\log(A+\eta)]^{-\delta}\|_{\mathcal{L}(X)},%\rightarrow 0.
\end{equation}

and so it is sufficient to assume that there exists a sequence $(t_n)_{n\in \mathbb{N}}\subset [0,\infty)$ such that $t_n\rightarrow \infty$ and
\begin{equation*}
\lim_{t_n\to \infty}\|T(t_n)(\eta+A)^{-\beta}[\log(A+\eta)]^{-2}\|_{\mathcal{L}(X)}=0.
\end{equation*}
  \end{itemize}
  
For each $\lambda\in \mathbb{C}$, with $\text{Re}{\lambda}>-\eta$, and each $a>0$, set $f_{t,a}(\lambda):=e^{-\lambda t}(\eta+\lambda)^{-a}[\log(\lambda+\eta)]^{-2}$. Let $a\ge 0$ and for each $\xi \in \mathbb{R}$, one has $\xi^{-a}\log(\xi)^{-1}=\mathcal{L}(v(\cdot,a-1))(\xi)$ (see Table 5.7 in $\cite{Bateman}$), where
\begin{equation*}
    v(x,a)=\int_{0}^{\infty} \frac{x^{s+a}}{\Gamma(s+a+1)}ds.
\end{equation*}

Now, for $a>1$, the inverse Laplace transform of $\xi^{-a}\log(\xi)^{-2}$ reads 
\begin{eqnarray*}
\frac{1}{2\pi i}\int_{b-i\infty}^{b+i\infty} e^{\xi \lambda} \frac{1}{\lambda^{a}\log^2(\lambda)}d\lambda &=&- \frac{1}{2\pi i}\int_{b-i\infty}^{b+i\infty} e^{\xi \lambda} \lambda^{-a+1} \frac{d}{d\lambda} \left(\frac{1}{\log(\lambda)} \right)d\lambda\\
&=& - \lim_{r\rightarrow \infty} \frac{1}{2\pi i} e^{\xi \lambda} \lambda^{-a+1}  \frac{1}{\log(\lambda)}\Big]_{b-ir}^{b+ir}+\frac{(-a+1)}{2\pi i}\int_{b-i\infty}^{b+i\infty} e^{\xi \lambda} \frac{1}{\lambda^{a}\log(\lambda)}d\lambda\\
&+& \frac{\xi}{2\pi i}\int_{b-i\infty}^{b+i\infty} e^{\xi \lambda} \frac{1}{\lambda^{a-1}\log(\lambda)}d\lambda\\
&=& (-a+1)v(\xi,a-1)+\xi v(\xi,a-2)
\end{eqnarray*}

Next, set %define $k_a\in L^{1}(\mathbb{R})$, 
\begin{align*}
 k_a(\xi):=\begin{cases}
[(-a+1)v(\xi,a-1)+\xi v(\xi,a-2)]e^{-\eta \xi}, \ \ \ \ \ \xi > 0  \\
   0,  \ \ \ \ \ \ \xi\leq 0,
 \end{cases}
 \end{align*}
and note that for each $\lambda\in \mathbb{C}$ with $\text{Re}{\lambda}>-\eta$, one has 
\begin{eqnarray*}
\mathcal{L}(\delta_{t}*k_a)(\lambda)&=& e^{-\lambda t}\int_{0}^{\infty} e^{-\lambda s}k_a(s)ds\\
&=& e^{-\lambda t}\int_{0}^{\infty} e^{-(\lambda+\eta) s}(-a+1)v(s,a-1)+s v(s,a-2)]ds\\
&=& e^{-\lambda t} (\lambda+\eta)^{-a}(\log(\lambda+\eta))^{-2}=f_{t,a}(\lambda).
\end{eqnarray*}

Now, by Hille-Phillips Functional Calculus, one has
\begin{equation*}
f_{t,a}(A):= T(t)(\eta+A)^{-a}\log(A+\eta)^{-2}.
\end{equation*}

\noindent {\bf{Case $\beta>1$:}}

%For $a=\beta$, one gets $\mathcal{L}(\delta_{t}*k_\beta)(\lambda)=f_t(\lambda)$.

Let $a=\beta$; by the Spectral Mapping Theorem (which is a consequence of Hille-Phillips Functional Calculus; see~\cite{Hille,haase}), one has $f_t(\sigma(A))\subset \sigma(f_t(A))$ for each $t>0$. Let $\lambda \in \sigma(A)$; then, $f_{t_n}(\lambda)\in \sigma(f_{t_n}(A))$ for each $t_n$ and
\begin{equation*}
  \frac{ e^{-\text{Re}{\lambda}t_n}}{|(\eta+\lambda)|^{\beta}|\log(\lambda+\eta))|^2}=|f_{t_n}(\lambda)|\leq \|T(t_n)(\eta+\lambda)^{-\beta}[\log(A+\eta)]^{-2}\|.
\end{equation*}

It follows from relation~\eqref{maisumpouco} that $\lim_{t_n\to\infty}e^{-(\text{Re}{\lambda})t_n}=0$, which is only possible if $-\text{Re}{\lambda}<0$.

\noindent {\bf{Case $\beta\leq1$:}}

Let $a>1$, and note that
\begin{eqnarray*}
\|T(t_n)(\eta+A)^{-a}(\log(\eta+A))^{-1-\delta}\|_{\mathcal{L}(X)}&\leq& \|T(t_n)(\eta+A)^{-(\beta+(a-\beta))}(\log(\eta+A))^{-1-\delta}\|_{\mathcal{L}(X)}\\
&\leq& \|(\eta+A)^{-(a-\beta)}\|_{\mathcal{L}(X)}\|T(t_n)(\eta+A)^{-\beta}(\log(\eta+A))^{-1-\delta}\|_{\mathcal{L}(X)}
\end{eqnarray*} 
\end{proof}

\subsection*{Proof of Theorem $\ref{theo4.5}$}

\begin{proof1} The result is equivalent to the following
statement: for each $s>0$ and each $\delta>0$, there exists $C_{\delta,s}>0 $ such that for each $t\geq 1$,
\begin{equation*}
\|T(t)(1+A)^{-\nu}\|_{\mathcal{L}(X)}\le C_{\delta,s} t^{-s}, 
\end{equation*}
where $\nu:=\beta(s+1)+1/r$ and $\upsilon:=b(s+1)+\dfrac{1+\delta}{r}$ for $p\neq 2$, $\nu:=\beta(s+1)$ and $\upsilon:=b(s+1)$ otherwise. Set $m:=\lfloor\upsilon\rfloor$ and $\eta:=\{\upsilon\}\in (0,1)$. We divide the proof into the cases where $\eta=0$ and $\eta>0$. In both of them, we proceed recursively over $m$.

\

%\begin{equation*}
%\eta:=\left\{ \begin{array}{cc}
 %   \frac{1+\delta}{r},    &\frac{1+\delta}{r} \in [0,1] \\
  %   \frac{1+\delta}{r} -1,   & \frac{1+\delta}{r} \in (1,2);
   % \end{array} \right.
%\end{equation*}
%it follows from this definition that $\eta\in [0,1]$.

\noindent\textbf{Case $\eta>0$.}

\

\textbf{Step 1: removing $\eta>0$}.  Since $(0,\infty)\ni \tau\mapsto \log(1+\tau)^{\eta}\in \mathcal{CBF}$ (by Proposition~\ref{theor2.3}), it follows that
\begin{equation*}    \log(1+\tau)^{\eta}=\int_{0+}^{\infty} \frac{\tau}{\tau+\lambda}d\mu(\lambda).
\end{equation*}

Let $\theta=\theta(s)\in (0,1)$ be such that $s\geq \theta>0$. Let, for each $\sigma\in(0,1)$, $f_\sigma:[0,\infty)\rightarrow\mathbb{R}$ be given by the law $f(\xi)=(1+\xi)^{\sigma}-1$; it is a complete Bernstein function (see Example~\ref{ex3.1}). Then, for $\sigma=\varepsilon:=\frac{\min\{1,\beta\}\theta}{2}$, $f_\varepsilon (B)$ is a sectorial operator, by Theorem~\ref{The2.3}, where $B:=A+1$. Therefore, by Lemma $\ref{lemma2.3}$,
  \[\log(1+B)^{\eta}=\frac{1}{\varepsilon^{\eta}}\log(1+f_{\varepsilon}(B))^{\eta}.\]

 By the choice of $\varepsilon>0$, $\mathcal{D}(B^{\beta s})\subset \mathcal{D}(f_\varepsilon (B))$ (see Proposition 3.1.1 (c) in \cite{haase}). It follows from equation~\eqref{cbf} (with $a=b=0$; see~\eqref{log(1+s)}) and from the previous facts that for each $x\in X$,
\begin{eqnarray*}
T(t)B^{-\nu}\log(1+B)^{-m}x&=&\frac{1}{\varepsilon^{\eta}}T(t)\log(1+f_{\varepsilon}(B))^{\eta}\log(1+B)^{-\eta}B^{-\nu}\log(1+B)^{-m}x\\
&=&\frac{1}{\varepsilon^{\eta}}T(t)\log(1+f_{\varepsilon}(B))^{\eta}B^{-\nu}\log(1+B)^{-m-\eta}x\\
&=& \frac{1}{\varepsilon^{\eta}}T(t)\int_{0+}^{\infty} f_{\varepsilon}(B) (\lambda+f_{\varepsilon}(B))^{-1}B^{-\nu}\log(1+B)^{-m-\eta}xd\mu(\lambda).
\end{eqnarray*}

Let $\phi \in [0,\nu)$, $\zeta>0$, set $P_{\zeta}(B_{\phi}):=B^{-\nu+\phi}\log(1+B)^{-\zeta}\in \mathcal{L}(X)$ and
  \[\tau:= \dfrac{\|T(t)P_{m+\eta}(B_{\varepsilon})\|_{\mathcal{L}(X)}}{\|T(t)P_{m+\eta}(B_{0})\|_{\mathcal{L}(X)}}>0.\]
  Since $f_\varepsilon(B)$ is a sectorial operator and since for each $t\geq 0$, $T(t)$ commutes with $f_\varepsilon(B)$ (see Lemma $\ref{lemma2.3}$), it follows that for each $t\geq 0$,
\begin{eqnarray}\label{eqq27}
\nonumber&&\left\|T(t)\int_{0+}^{\tau}f_{\varepsilon}(B)(\lambda+f_{\varepsilon}(B))^{-1}P_{m+\eta}(B_{0})x d\mu(\lambda)\right\|\\
\nonumber&&\leq \|T(t)P_{m+\eta}(B_{0})\|_{\mathcal{L}(X)}(M_{f_\varepsilon(B)}+1) \int_{0+}^{\tau}d\mu(\lambda)\|x\|\\
&&\leq 4\|T(t)P_{m+\eta}(B_{0})\|_{\mathcal{L}(X)}M_{f_\varepsilon(B)} \int_{0+}^{\tau}\frac{\tau}{\tau+\lambda} d\mu(\lambda)\|x\|
\end{eqnarray}
where $\displaystyle{M_{f_\varepsilon(B)}:=\sup_{\lambda>0}\|\lambda(\lambda+f_{\varepsilon}(B))^{-1}\|_{\mathcal{L}(X)}\geq 1}$. Moreover, by Lemma $\ref{lemma2.3}$,
\begin{eqnarray}\label{eqq28}
\nonumber&&\left\|T(t)\int_{\tau}^{\infty}f_{\varepsilon}(B)(\lambda+f_\varepsilon(B))^{-1}P_{m+\eta}(B_0)x d\mu(\lambda)\right\|\\
\nonumber &&\leq  \|T(t)f_{\varepsilon}(B)P_{m+\eta}(B_0)\|_{\mathcal{L}(X)}\int_{\tau}^{\infty} \|(\lambda+f_{\varepsilon}(B))^{-1}\|_{\mathcal{L}(X)}d\mu(\lambda)\|x\|\\
\nonumber &&\leq \|T(t)f_{\varepsilon}(B)P_{m+\eta}(B_0)\|_{\mathcal{L}(X)}M_{f_\varepsilon(B)}\int_{\tau}^{\infty} \frac{1}{\lambda} d\mu(\lambda)\|x\|\\
&&\leq  2\|T(t)f_{\varepsilon}(B)P_{m+\eta}(B_0)\|_{\mathcal{L}(X)}M_{f_\varepsilon(B)} \int_{\tau}^{\infty} \frac{1}{\lambda+\tau} d\mu(\lambda)\|x\|.
\end{eqnarray}

By combining relations $\eqref{eqq27}$ and $\eqref{eqq28}$, one gets, for each $t\geq 0$, 
\begin{eqnarray}\label{eq27a} 
\nonumber\|T(t)P_m(B_0)x\|&\le& C_{\varepsilon}^{\prime}\|T(t)P_{m+\eta}(B_0)\|_{\mathcal{L}(X)} \left( \int_{0+}^{\tau}\frac{\tau}{\tau+\lambda} d\mu(\lambda)\right.\\
  &+&\left.\frac{\|T(t)f_{\varepsilon}(B)P_{m+\eta}(B_0)\|_{\mathcal{L}(X)}}{\|T(t)P_{m+\eta}(B_0)\|_{\mathcal{L}(X)}}\int_{\tau}^{\infty} \frac{1}{\lambda+\tau} d\mu(\lambda)\right)\|x\|,
\end{eqnarray}
where  $C_{\varepsilon}^{\prime}:=4M_{f_\varepsilon(B)}$.

Note that, for each $t\geq 0$, $T(t)$ commutes with $(1+B)^{\varepsilon}B^{-\varepsilon}-B^{-\varepsilon}$, hence for each $x\in X$,
\begin{eqnarray*}
\|T(t)f_\varepsilon(B)P_{m+\eta}(B_0)x\|&=& \|T(t)((1+B)^{\varepsilon}-1)P_{m+\eta}(B_0)x\|\\
&=& \|T(t)((1+B)^{\varepsilon}-B^{-\varepsilon}B^{\varepsilon})B^{-\nu}\log(1+B)^{-m-\eta}x\|\\
&=& \|T(t)((1+B)^{\varepsilon}B^{-\varepsilon}-B^{-\varepsilon})B^{-\nu+\varepsilon}\log(1+B)^{-m-\eta}x\|\\
&=& \|((1+B)^{\varepsilon}B^{-\varepsilon}-B^{-\varepsilon})T(t)P_{m+\eta}(B_\varepsilon)x\|\\
&\leq& \|((1+B)^{\varepsilon}B^{-\varepsilon}-B^{-\varepsilon})\|_{\mathcal{L}(X)}\|T(t)P_{m+\eta}(B_\varepsilon)\|_{\mathcal{L}(X)}\|x\|,
\end{eqnarray*}
and so
\begin{equation}\label{eq9} \|T(t)f_\varepsilon(B)P_{m+\eta}(B_0)\|_{\mathcal{L}(X)}\leq C_{\varepsilon} \|T(t)P_{m+\eta}(B_\varepsilon)\|_{\mathcal{L}(X)},
\end{equation}
where $C_{\varepsilon}:=\|(1+B)^{\varepsilon}B^{-\varepsilon}-B^{-\varepsilon}\|_{\mathcal{L}(X)}$.

Therefore, it follows from relations~\eqref{eq27a} and~\eqref{eq9} that for each $t\ge 0$,
\begin{eqnarray}\label{eqq29}
  \nonumber\|T(t)B^{-\nu}\log(1+B)^{-m}\|_{\mathcal{L}(X)}&\leq& C'_{\varepsilon} (1+C_{\varepsilon})\|T(t)P_{m+\eta}(B_{0})\|_{\mathcal{L}(X)}\left(\int_{0+}^{\tau}\frac{\tau}{\tau+\lambda} d\mu(\lambda) +\tau\int_{\tau}^{\infty} \frac{1}{\lambda+\tau} d\mu(\lambda)\right)\\
 &=& C'_{\varepsilon} (1+C_{\varepsilon})\|T(t)P_{m+\eta}(B_{0})\|_{\mathcal{L}(X)} \log(1+\tau)^{\eta}.
\end{eqnarray}

On the other hand, by the definition $\varepsilon=\frac{\min\{1,\beta\}\theta}{2}$, one has $\beta(s+1)+1/r-\varepsilon>\beta +1/r$; then, it follows from Theorem~\ref{theo4.4} that there exists a positive constant $C_{s,\varepsilon,\eta}$ such that for each $t\geq 1$,
\begin{equation}\label{eqq30}
   \|T(t)P_{m+\eta}(B_\varepsilon)\|_{\mathcal{L}(X)} \leq C_{s,\varepsilon,\eta}t^{-s+\frac{\varepsilon}{\beta}}.
\end{equation}
One also has from Theorem~\ref{theo4.4} that there exists a positive constant $C_{s,\eta}$ so that for each $t\geq 1$,
\begin{equation}\label{eq27b}
  \|T(t)P_{m+\eta}(B_0)\|_{\mathcal{L}(X)}\le C_{s,\eta}t^{-s}. 
\end{equation}

Now, set $k_{m,\eta}(t):=\|T(t)P_{m+\eta}(B_0)\|_{\mathcal{L}(X)}$; by $\eqref{eq27b}$, one has for each $t\geq 1$,
\begin{equation}\label{eqq31}
\frac{C_{s,\varepsilon,\eta}t^{-s+\frac{\varepsilon}{\beta}}}{k_{m,\eta}(t)}\geq \Tilde{C}_{s,\eta} C_{s,\varepsilon,\eta}t^{\frac{\varepsilon}{\beta}},
\end{equation}
with $\Tilde{C}_{s,\eta}:=(C_{s,\eta})^{-1}$. It follows from relations~\eqref{eqq30},~\eqref{eq27b} and by letting $\gamma=1$, $\ell(w)=\log(1+w)^{\eta}$, $s=\Tilde{C}_{s,\eta} C_{s,\varepsilon,\eta}t^{\frac{\varepsilon}{\beta}}$ and $w=\frac{C_{s,\varepsilon,\eta}t^{-s+\frac{\varepsilon}{\beta}}}{k_{m,\eta}(t)}$ in Proposition~\ref{cor2} (note that $w\ge s$, by~\eqref{eqq31}) that there exists $C>0$ (which depends only on the function $\log$) such that for each sufficiently large $t$, %and by equation $\eqref{eqq31}$,
\begin{eqnarray}\label{eq27c}
\nonumber\log(1+\tau)^{\eta}\leq \log\left(1+\frac{C_{s,\varepsilon,\eta}t^{-s+\frac{\varepsilon}{\beta}}}{k_{m,\eta}(t)}\right)^{\eta}&\leq&  C\frac{C_{s,\varepsilon,\eta}t^{-s+\frac{\varepsilon}{\beta}}}{\Tilde{C}_{s,\eta} C_{s,\varepsilon,\eta}t^{\frac{\varepsilon}{\beta}}k_{m,\eta}(t)}\log\left(1+\Tilde{C}_{s,\delta} C_{s,\varepsilon,\eta}t^{\frac{\varepsilon}{\beta}}\right)^{\eta}\\
\nonumber&=&\frac{C}{k_{m,\eta}(t)\Tilde{C}_{s,\eta}t^{s}
}\log\left(1+\Tilde{C}_{s,\eta} C_{s,\varepsilon,\eta}t^{\frac{\varepsilon}{\beta}}\right)^{\eta}\\
&=& C_1(s,\varepsilon,\eta)\frac{1}{k_{m,\eta}(t)t^{s}}\log\left(1+C_2(s,\varepsilon,\eta)t^{\frac{\varepsilon}{\beta}}\right)^{\eta},
\end{eqnarray}
with $C_1(s,\varepsilon,\eta):=C/\Tilde{C}_{s,\eta}$ and $C_2(s,\varepsilon,\eta):=\Tilde{C}_{s,\eta} C_{s,\varepsilon,\eta}$.

Then, one concludes from~\eqref{eqq29} and~\eqref{eq27c} that for each sufficiently large $t$,
\begin{eqnarray}\label{eqq34}
\nonumber\|T(t)B^{-\nu}\log(1+B)^{-m}\|_{\mathcal{L}(X)}&\leq& \frac{C'_{\varepsilon} (1+C_{\varepsilon})C_1(s,\varepsilon,\eta)}{\varepsilon } \frac{k_{m,\eta}(t)}{k_{m,\eta}(t)t^{s}}\log\left(1+C_2(s,\varepsilon,\eta)t^{\frac{\varepsilon}{\beta}}\right)^{\eta}\\
&= & \Tilde{C}_1(s,\varepsilon,\eta)t^{-s}  \log\left(1+C_2(s,\varepsilon,\eta)t^{\frac{\varepsilon}{\beta}}\right)^{\eta},
\end{eqnarray}
with $\Tilde{C}_1(s,\varepsilon,\eta):=\dfrac{C'_{\varepsilon} (1+C_{\varepsilon})C_1(s,\varepsilon,\eta)}{\varepsilon}$.

\

\noindent {\bf{Step 2: removing $m$}}.  If $m=0$, there is nothing to be done. So, let $m\in\mathbb{N}$. It follows from the discussion presented in the beginning of \textbf{Step 1} that for each $x\in X$, 
\begin{eqnarray*}
T(t)B^{-\nu}\log(1+B)^{-m+1}x&=&\frac{1}{\varepsilon}T(t)\log(1+f_{\varepsilon}(B))B^{-\nu}\log(1+B)^{-m}x\\
&=& \frac{1}{\varepsilon}T(t)\int_{0+}^{\infty}f_{\varepsilon}(B) (\lambda+f_{\varepsilon}(B))^{-1}P_{m}(B_0)xd\mu(\lambda),
\end{eqnarray*}
where $\mu$ now stands for the Borel measure related to the integral representation of $\log(1+\lambda)$ (which is a complete Bernstein function).

Let $\tau:= \dfrac{\|T(t)P_{m}(B_\varepsilon)\|_{\mathcal{L}(X)}}{\|T(t)P_{m}(B_0)\|_{\mathcal{L}(X)}}>0$. By proceeding as in {\bf{Step 1}}, one gets from~\eqref{eqq34} that for each $t\ge 1$,
\begin{equation}\label{eqq35a}
\|T(t)P_{m-1}(B_{\varepsilon})\|_{\mathcal{L}(X)}\lesssim t^{-s+\varepsilon/\beta}  \log\left(1+t\right)^{\eta},
\end{equation}
%with $\Tilde{C}_{s,\varepsilon}$ a positive constant,
and  
\begin{equation}\label{eqq35}
\|T(t)P_{m-1}(B_{0})\|_{\mathcal{L}(X)}\lesssim\|T(t)P_{m}(B_{0})\|_{\mathcal{L}(X)} \log(1+\tau).
\end{equation}

Now, let $n_{m}(t):=\|T(t)P_{m}(B_0)\|_{\mathcal{L}(X)}$; it follows from~\eqref{eqq35} and~\eqref{eqq35a} that there exists a positive constant $\Tilde{c}_{s,\varepsilon}$ so that for each $t\ge 1$,
\begin{equation}\label{eq17}
\frac{t^{-s+\frac{\varepsilon}{\beta}}\log\left(1+t\right)^{\eta}}{n_{m}(t)}\geq \Tilde{c}_{s,\varepsilon} t^{\frac{\varepsilon}{\beta}}.
\end{equation}

By letting $\gamma=1$, $\ell(w)=\log(1+w)$, $s=\Tilde{c}_{s,\varepsilon}t^{\frac{\varepsilon}{\beta}}$ and $w=\frac{t^{-s+\frac{\varepsilon}{\beta}}\log\left(1+ t\right)^\eta}{n_m(t)}$ in Proposition~\ref{cor2} (note that $w\ge s$, by~\eqref{eq17}), it follows from relation~\eqref{eq17} that there exists  $\Tilde{c}>0$ such that for each sufficiently large $t$,%By $\eqref{eq17}$ and it follows from Corollary $\ref{cor2}$, that there exist $\Tilde{c}>0$ such that 
\begin{eqnarray}\label{eqq37}
\nonumber\log(1+\tau)&\leq& \log\left(1+\frac{\Tilde{C}_{\varepsilon, s}t^{-s+\frac{\varepsilon}{\beta}}\log\left(1+ t\right)^{\eta}}{n_{m}(t)}\right)\\
\nonumber&\leq& \Tilde{c} \frac{\Tilde{C}_{\varepsilon, s}t^{-s+\frac{\varepsilon}{\beta}}\log\left(1+t\right)^{\eta}}{\Tilde{C}_{\varepsilon, s} \Tilde{c}_{s,\varepsilon} t^{\frac{\varepsilon}{\beta}}n_{m}(t)}\log\left(1+C_{\varepsilon, s, \delta}\Tilde{c}_{s,\varepsilon} t^{\frac{\varepsilon}{\beta}}\right)\\
&=&\frac{\Tilde{c}\log\left(1+t\right)^{\eta}}{n_{m}(t)\Tilde{c}_{s,\varepsilon}t^{s}
}\log\left(1+\Tilde{C}_{\varepsilon, s} \Tilde{c}_{s,\varepsilon} t^{\frac{\varepsilon}{\beta}}\right).
\end{eqnarray}

Then, by $\eqref{eqq35}$ and $\eqref{eqq37}$, one has for each $t\ge 1$,
\begin{eqnarray*}
\nonumber\|T(t)P_{m-1}(B_{0})\|_{\mathcal{L}(X)}&\lesssim& %\frac{\Tilde{c}M_{f_\varepsilon(B)}\log\left(1+\Tilde{C}_{\varepsilon, s}t^{\frac{\varepsilon}{\beta}}\right)^{\eta}}{\varepsilon\Tilde{c}_{s} } \frac{ n_{\eta}(t)}{n_{\eta}(t)t^{s}}\log\left(1+\Tilde{C}_{s} \Tilde{C}_{\varepsilon, s}t^{\frac{\varepsilon}{\beta}}\right)\\
 t^{-s}\log\left(1+t\right)^{\eta} \log\left(1+\Tilde{C}_{\varepsilon, s} \Tilde{c}_{s}t^{\frac{\varepsilon}{\beta}}\right)\\
%&=& C^{''}_{\varepsilon, s}t^{-s}\log\left(1+ C^{'''}_{\varepsilon, s}t^{\frac{\beta}{\varepsilon}}\right)^{1+\eta}\\
&\lesssim & t^{-s}\log\left(1+ t\right)^{1+\eta}.
\end{eqnarray*}
%where $ C^{''}_{\varepsilon, s}:=\dfrac{\Tilde{c}M_{f_\varepsilon}(B)}{\varepsilon \Tilde{c}_{s,\varepsilon}}$.

By proceeding recursively over $m$, it follows from the previous discussion that for each $t\ge 1$,
\begin{eqnarray*}
\nonumber\|T(t)B^{-\nu}\|_{\mathcal{L}(X)}&\lesssim& t^{-s}\log\left(1+ t\right)^{m+\eta}.
\end{eqnarray*}

\

\textbf{Case $\eta=0$}. Since in this case $m\in\mathbb{N}$, one just needs to proceed as in $\textbf{Step 2}$ of the case $\eta>0$ in order to obtain, for each $t\ge 1$,
\begin{eqnarray*}
\nonumber\|T(t)B^{-\nu}\|_{\mathcal{L}(X)}&\lesssim& t^{-s}\log\left(1+ t\right)^{m}.
\end{eqnarray*}

Hence, in both cases, relation~\eqref{eqTh1.1} follows, and we are done.
\end{proof1}

\subsection{Resolvent growth slower than $\log(|\xi|)^{b}$}
\label{finalmente}
$\bullet$ {\textbf{Case $p\neq 2$}}

By using the same strategy presented in the proof of Theorem~\ref{theo4.4}, we conclude that for $\beta=0<b$ (that is, for $\|(i\xi+A)^{-1}\|_{\mathcal{L}(X)}\lesssim \log(2+|\xi|)^b$, $\xi\in \mathbb{R}$), for each $s\geq 0$ and each $\delta>0$, there exists $c_{s,\delta}> 0$ such that for each $t\geq 1$,
\begin{equation*}
\|T(t)(1+A)^{-1/r}\log(1+B)^{-b(s+1)-\frac{1+\delta}{r}}\|_{\mathcal{L}(X)}\leq c_{s,\delta}t^{-s}.
\end{equation*}

Actually, it is possible to obtain in this setting a better estimate than the previous one. Namely, note that for each $x\in \dom (A)$,
\[T(t)x=x+\int_{0}^{t}T(w)Axdw.\]
Let $x\in X_{1+1/r}(\upsilon)$, with $\upsilon:=b(s+1)+\frac{1+\delta}{r}$. We argue that $[t\mapsto T(t)x]$ is a Lipschitz continuous function: it follows from the previous identity that for each $t,u\geq 0$, $\|T(t)x-T(u)x\|\lesssim |t-u|\|x\|_{X_{1/r}(\upsilon)}$, and since $X_{1+1/r}(\upsilon)$ is dense $X_{1/r}(\upsilon)$, one concludes that $\|T(t)-T(u)\|_{\mathcal{L}(X_1(\upsilon),X)}\lesssim |t-u|$.

Now, note that for each $x\in X$ and each $t>0$,  $f_x(t)=T(t)(1+A)^{-1/r}\log(1+B)^{-\upsilon}x$ satisfies the assumptions of Theorem 2.1 in~\cite{chill1} (with $F_x(s)=R(is,A)(1+A)^{-1/r}\log(1+B)^{-\upsilon}x$ and $M(s)=\log(2+|s|)^{b}$), so for any $c\in (0,1/2)$ and $t_0$ such that for each $t\geq t_0$ and each $x\in X$ with $\|x\|=1$, one has
\[\|T(t)(1+A)^{-1/r}\log(1+B)^{-\upsilon}x\|\lesssim \frac{1}{M^{-1}_{\log(ct)}} \lesssim e^{-ct^{\frac{1}{b+1}}}.\]
Therefore, for each $t\geq t_0$,
\begin{equation*}
\|T(t)(1+A)^{-1/r}\log(1+B)^{-\upsilon}\|_{\mathcal{L}(X)}\leq e^{-ct^{\frac{1}{b+1}}}. 
\end{equation*}

By proceeding as in the proof of Lemma 4.1 in~\cite{rozendaal}, one can show that for each $\tau>1/r$,
\begin{equation}\label{aindanao}
\|T(t)(1+A)^{-\tau}\log(1+B)^{-\upsilon\tau r}\|_{\mathcal{L}(X)}\leq e^{-c\tau r t^{\frac{1}{b+1}}}. 
\end{equation}

\begin{theorem}\label{teo3.22}
\begin{rm}
Let $b\ge 0$ and let $(T(t))_{t\ge 0}$ be a $C_0$-semigroup on a Banach space $X$ whose generator $-A$ satisfies $\overline{\mathbb{C}_{-}}\subset \rho(A)$. Suppose that $X$ has Fourier type $p \in [1,2)$ and that for each $\lambda \in\mathbb{C}$ with $\text{Re}(\lambda)\le 0$,
\begin{equation*}
\|(\lambda+A)^{-1}\|_{\mathcal{L}(X)}\lesssim \log(2+|\lambda|)^b.
\end{equation*}
Let $r\in[1,\infty)$ be such that $1/r=1/p-1/p^{\prime}$. Then, for each $\delta,\varepsilon>0$ and each $\tau>1/r$, there exists $c_{\tau,\delta,r,\varepsilon} \geq 0$ such that for each $t\geq 1$,
\begin{equation*}
 \|T(t)(1+A)^{-\tau}\|_{\mathcal{L}(X)}\leq c_{\tau,\delta,\varepsilon, r}e^{-c\tau r t^{\frac{1}{b+1}}}t^{\frac{\tau r}{b+1}(b(\varepsilon+1)+\frac{1+\delta}{r})}.
\end{equation*}
\end{rm}
\end{theorem}

\begin{proof}
We proceed as in the proof of Theorem~\ref{theo4.5}, by replacing relation~\eqref{aindanao1} by relation~\eqref{aindanao}.
\end{proof}

$\bullet$ {\textbf{Case $p=2$}}

As in the case above, by using the same strategy presented in the proof of Theorem~\ref{theo4.4}, we conclude that for $\beta=0<b$, for each $s\geq 0$ and each $\delta\in (0,1/2)$, there exists $c_{s,\delta}> 0$ such that for each $t\geq 1$,
\begin{equation*}
\|T(t)(1+A)^{-\delta}\log(1+B)^{-b(s+1)}\|_{\mathcal{L}(X)}\leq c_{s,\delta}t^{-s}.
\end{equation*}

Actually, it is possible to obtain in this setting a better estimate than the previous one. Namely, note that for each $x\in \dom (A)$,
\[T(t)x=x+\int_{0}^{t}T(w)Axdw.\]
Let $x\in X_{1+\delta}(\upsilon)$, with $\upsilon:=b(s+1)$. We argue that $[t\mapsto T(t)x]$ is a Lipschitz continuous function: it follows from the previous identity that for each $t,u\geq 0$, $\|T(t)x-T(u)x\|\lesssim |t-u|\|x\|_{X_\delta(\upsilon)}$, and since $X_{1+\delta}(\upsilon)$ is dense $X_{\delta}(\upsilon)$, one concludes that $\|T(t)-T(u)\|_{\mathcal{L}(X_\delta(\upsilon),X)}\lesssim |t-u|$.

Now, note that for each $x\in X$ and each $t>0$,  $f_x(t)=T(t)(1+A)^{-\delta}\log(1+B)^{-\upsilon}x$ satisfies the assumptions of Theorem 2.1 in~\cite{chill1} (with $F_x(s)=R(is,A)(1+A)^{-\delta}\log(1+B)^{-\upsilon}x$ and $M(s)=\log(2+|s|)^{b}$), so there exists $t_0(\delta)\geq 1$ such that for each $t\geq t_0$ and each $x\in X$ with $\|x\|=1$, one has
\[\|T(t)(1+A)^{-\delta}\log(1+B)^{-\upsilon}x\|\lesssim \frac{1}{M^{-1}_{\log(\delta t)}} \lesssim e^{-\delta t^{\frac{1}{b+1}}}.\]
Therefore, for each $t\geq t_0$,
\begin{equation*}
\|T(t)(1+A)^{-\delta}\log(1+B)^{-\upsilon}\|_{\mathcal{L}(X)}\lesssim e^{-\delta t^{\frac{1}{b+1}}}. 
\end{equation*}

Let $\tau>0$ and $\delta \in (0,1/2)$ be such that $\tau>\delta>0$, then by proceeding as in the proof of Lemma 4.1 in~\cite{rozendaal}, one can show ,
\begin{equation*}
\|T(t)(1+A)^{-\tau}\log(1+B)^{-\upsilon\tau/\delta}\|_{\mathcal{L}(X)}\leq e^{-\tau t^{\frac{1}{b+1}}}. 
\end{equation*}

\begin{theorem}%\label{teo3.22}
\begin{rm}
Let $b\ge 0$ and let $(T(t))_{t\ge 0}$ be a $C_0$-semigroup on a Hilbert space $X$ whose generator $-A$ satisfies $\overline{\mathbb{C}_{-}}\subset \rho(A)$. Suppose that for each $\lambda \in\mathbb{C}$ with $\text{Re}(\lambda)\le 0$,
\begin{equation*}
\|(\lambda+A)^{-1}\|_{\mathcal{L}(X)}\lesssim \log(2+|\lambda|)^b.
\end{equation*}
Let $\varepsilon,\tau>0$ and $\delta \in (0,1/2)$ be such that $\tau>\delta>0$. Then, there exists $c_{\tau, \varepsilon,\delta} \geq 0$ such that for each $t\geq 1$,
\begin{equation*}
 \|T(t)(1+A)^{-\tau}\|_{\mathcal{L}(X)}\leq c_{\tau,\varepsilon,\delta}e^{-\tau t^{\frac{1}{b+1}}}t^{\frac{\tau b(\varepsilon+1)}{\delta(b+1)}}.
\end{equation*}
\end{rm}
\end{theorem}

\begin{proof}
We proceed as in the proof of Theorem~\ref{theo4.5}, by replacing relation~\eqref{aindanao1} by relation~\eqref{aindanao}.
\end{proof}

%\end{remark}

\section{Singularity at infinity and zero}\label{four}
\zerarcounters

Let $\mu,\nu,\upsilon\geq 0$ and $A\in \text{Sect}_X(\omega_A)$; it is known that $2\pi-i\log(A)$ is sectorial (see page 92 in~\cite{haase}), and so $(2\pi-i\log(A))^{-\upsilon}$ is well-defined, by the functional calculus of fractional powers (see~\cite{haase,mart}). Define the operator
\begin{equation*}
\Phi^{\mu}_{\nu}(\upsilon)=\Phi^{\mu}_{\nu}(A,\upsilon):= A^{\mu}(1+A)^{-\mu-\nu}(2\pi-i\log(A))^{-\upsilon} \in \mathcal{L}(X),
\end{equation*}
and the space $X^{\mu}_{\nu}(\upsilon):=\Ran(\Phi^{\mu}_{\nu}(\upsilon))$. If $A$ is injective, then the space $X^{\mu}_{\nu}(\upsilon)$ 
 is a Banach space with the norm
 \begin{eqnarray*}
\|x\|_{X^{\mu}_{\nu}(\upsilon)}&=&\|x\|+\|\Phi^{\mu}_{\nu}(\upsilon)^{-1}x\|=\|x\|+\|(2\pi-i\log(A))^{\upsilon}(1+A)^{\nu}A^{\mu}x\|, \ \ \ \forall~x \in X^\mu_{\nu}(\upsilon).
 \end{eqnarray*}
Moreover, $\Phi^\mu_{\nu}(\upsilon):X\rightarrow  X^\mu_{\nu}(\upsilon)$ is an isomorphism and so %there exists $C>0$ such that   
\begin{equation}\label{eqq41}
     \|T\|_{\mathcal{L}(X^\mu_{\nu}(\upsilon),X)} \leq  \|T\Phi^\mu_{\nu}(\upsilon)\|_{\mathcal{L}(X)}\leq\| \Phi^\mu_{\nu}(\upsilon)\|_{\mathcal{L}(X)}\|T\|_{\mathcal{L}(X^{\mu}_{\nu}(\upsilon),X)}, \ \ \  T\in \mathcal{L}(X^\mu_{\nu}(\upsilon),X).
 \end{equation}
 
Note that $\Phi^{\mu}_{\nu}(0)=\Phi^{\mu}_{\nu}(A)$ and $X^{\mu}_{\nu}(0)=X^{\mu}_{\nu}$, where $\Phi^{\mu}_{\nu}(A)$ and $X^{\mu}_{\nu}$ are the objects defined in \cite{rozendaal}.

\begin{theorem}\label{teo4.4}
\begin{rm}
Let $(T(t))_{t\geq 0}$ be a $C_0$-semigroup  defined in the Banach space $X$ with Fourier type $p\in [1,2]$, with $-A$ as its generator. Suppose $A$ injective, $\overline{\mathbb{C}_{-}}\setminus\{0\} \subset \rho(A)$ and that there exist $\alpha \geq 1$, $\beta>0$, $a,b\geq 0$ such that, for each  $\lambda \in \overline{\mathbb{C}_{-}}\setminus\{0\}$, % a Banach space $X$ with Fourier type $p\in[1,2]$ and suppose that
\begin{equation*}
 \|(\lambda+A)^{-1}\|_{\mathcal{L}(X)}\lesssim \left\{\begin{array}{cc}
  |\lambda|^{-\alpha}\log(1/|\lambda|)^{a}, & |\lambda|\leq 1 \\
|\lambda|^{\beta}\log(|\lambda|)^{b}, & |\lambda|\geq 1.
\end{array}\right.
\end{equation*}
%Let $r\in [1,\infty]$ be such that $\frac{1}{r}=\frac{1}{p}-\frac{1}{p'}$ , and
Let $\sigma, \tau$ be such that $\sigma\geq \alpha-1$ and $\tau\geq  \beta+1/r$. Then, for
each $\rho \in \left[0, \min\left\{\frac{\sigma+1}{\alpha}-1,\frac{\tau-r^{-1}}{\beta}-1\right\}\right]$  and each $\delta>1-1/r$, where $r\in [1,\infty]$ is such that $\frac{1}{r}=\frac{1}{p}-\frac{1}{p'}$, there exist  $C_{\rho,\delta}>0$ and $t_0>1$ so that for each $t\ge1$,
\begin{equation}
\|T(t)A^{\sigma}(1+A)^{-\sigma-\tau}(2\pi-i\log(A))^{-c(\left \lceil{\rho}\right \rceil +1)-1/r-\delta}\|_{\mathcal{L}(X)}\leq C_{\rho,\delta} t^{-\rho}, 
\end{equation}
with $c=\max\{a,b\}$.
\end{rm}
\end{theorem}

In order to prove Theorem~\ref{teo4.4}, some preparation is required. The next result is the version of Proposition~\ref{theor3.1} in this setting, and consists in the result stated in Theorem~\ref{teo4.4} in case $p=1$.

\begin{proposition}\label{theor3.3}
\begin{rm}
Let $A$ be an injective sectorial operator defined in the Banach space $X$ such that $-A$ generates the $C_0$-semigroup $(T(t))_{t\ge0}$ on $X$. Suppose that there exist $\alpha\ge 1$,~$\beta>0$, $a,b\ge 0$ such that, for each  $\lambda \in \overline{\mathbb{C}_{-}}\setminus\{0\}$
\begin{equation}\label{eeq42}
  \|(\lambda+A)^{-1}\|_{\mathcal{L}(X)} \lesssim \left\{
  \begin{array}{cc}
    |\lambda|^{-\alpha}(\log(1/|\lambda|))^{a}; & |\lambda|\leq 1 \\
    |\lambda|^{\beta}\log(|\lambda|)^{b}; & |\lambda|\geq 1 .
  \end{array}\right.
\end{equation}
Let $\sigma \geq \alpha-1$ and $\tau\geq \beta+1$. Then, for each $\delta>0$ and each $\rho \in \left[0,\min\left\{\frac{\sigma+1}{\alpha}-1,\frac{\tau-1}{\beta}-1\right\}\right]$, there exists $c_{\rho,\delta}>0$ such that for each $t\geq 1$,
\begin{equation}
\|T(t)A^{\sigma}(1+A)^{-\sigma-\tau}(2\pi-i\log(A))^{-c(\left \lceil{\rho}\right \rceil+1)-1-\delta)}\|_{\mathcal{L}(X)}\leq c_{\rho,\delta}t^{-\rho},
\end{equation}
where $c=\max\{a,b\}$.
\end{rm}
\end{proposition}

\begin{proof} 

Let $n\in \mathbb{N}\cup\{0\}$ and set $\mu:=\alpha(n+1)-1$, $\nu:=(n+1)\beta+1$, $\upsilon:=c(n+1)+1+\delta$. For each $x\in X^{\mu}_{\nu+1}(\upsilon)$, let
\begin{eqnarray*}
y:=(\Phi^{\mu}_{\nu}(\upsilon))^{-1}x&=&(2\pi-i\log(A))^{\upsilon}(1+A)^{\nu+\mu}A^{-\mu}x \\
&=& (2\pi-i\log(A))^{\upsilon}(1+A)^{\nu+\mu}A^{-\mu}\left(A^{\mu}(1+A)^{-\mu-\nu-1}(2\pi-i\log(A))^{-\upsilon}z\right)\\
&=& B^{-1}z\in \mathcal{D}(A),
\end{eqnarray*}
where $z:=(\Phi^{\mu}_{\nu+1}(\upsilon))^{-1}x$. Let $g:[0,\infty)\rightarrow X$ be defined by the law
\begin{equation*}
g(t):=\frac{1}{2\pi i} \int_{i\infty}^{-i\infty} e^{-\lambda t} \frac{\lambda^{\mu}}{(1+\lambda)^{\nu+\mu}(2\pi-i\log(\lambda))^{\upsilon}}R(\lambda,A)yd\lambda.
\end{equation*}
Note that for each $t\ge 0$, $g(t)$ is indeed an element of $X$ and it is differentiable. Namely, since $y\in \mathcal{D}(A)$, then $\lambda\mapsto \dfrac{\lambda^{\mu+1}}{(1+\lambda)^{\nu+\mu}(2\pi-i\log(\lambda))^{\upsilon}}R(\lambda,A)y$ is integrable in $i\mathbb{R}$. Therefore, by dominated convergence, one gets
\begin{equation*}
g'(t)=-\frac{1}{2\pi i}\int_{i\infty}^{-i\infty}e^{-\lambda t} \frac{\lambda^{\mu+1}}{(1+\lambda)^{\nu+\mu}(2\pi-i\log(\lambda))^{\upsilon}}R(\lambda,A)yd\lambda.
\end{equation*}
Moreover, by Lemma \ref{lemmaB.1},
\begin{eqnarray*}
g'(t)&=& \frac{1}{2\pi i}\int_{i\infty}^{-i\infty}e^{-\lambda t} \frac{\lambda^{\mu}}{(1+\lambda)^{\nu+\mu}(2\pi-i\log(\lambda))^{\upsilon}}(-AR(\lambda,A)y-y)d\lambda\\
&=& -A\frac{1}{2\pi i}\int_{i\infty}^{-i\infty}e^{-\lambda t} \frac{\lambda^{\mu}}{(1+\lambda)^{\nu+\mu}(2\pi-i\log(\lambda))^{\upsilon}}R(\lambda,A)yd\lambda-\\
&-&\frac{1}{2\pi i}\int_{i\infty}^{-i\infty}e^{-\lambda t} \frac{\lambda^{\mu}}{(1+\lambda)^{\nu+\mu}(2\pi-i\log(\lambda))^{\upsilon}}yd\lambda \\
&=& -Ag(t)
\end{eqnarray*}

Then, $g'(t)=-Ag(t)$ for each $t\geq 0$, and $g(0)=x$. Therefore, $g(t)=T(t)x$, by the uniqueness of the solution to the Cauchy problem associated with $-A$.

Integration by parts yields
\begin{equation*}
    t^n T(t)=\frac{1}{2\pi i} \int_{i\mathbb{R}}e^{-\lambda t} p(\lambda,A)y d\lambda,
\end{equation*}
where $p(\lambda,A)$ is a finite linear combination of terms of the form
\begin{equation*}
    \frac{\lambda^{\mu-k}R(\lambda,A)^{n-l+1}}{(1+\lambda)^{\mu+\nu+(l-k)}(2\pi-i\log(\lambda))^{\upsilon+j}} \ \ \text{and} \ \ \frac{\lambda^{\mu-k-m}R(\lambda,A)^{n-l+1}}{(1+\lambda)^{\mu+\nu+j}(2\pi-i\log(\lambda))^{\upsilon+(l-k)}},
\end{equation*}
where $0 \le j \le k\le l \leq  n$ and $l-k\leq m\leq l-k+1$.

Then, for each $t>0$,
\begin{eqnarray*}
\|t^{n}T(t)x\|&\leq& \frac{1}{2\pi} \int_{i\mathbb{R}}|e^{-\lambda t}| \|p(\lambda,A)y\| d\lambda\\
&\leq& \frac{1}{2\pi} \int_{i\mathbb{R}} \|p(\lambda,A)\|_{\mathcal{L}(X)} d\lambda \|y\|\leq C\|(\Phi^{\mu}_{\nu}(\upsilon))^{-1}x\|\lesssim\|x\|_{X^{\mu}_{\nu}(\upsilon)}.
\end{eqnarray*}

Since $X^\mu_{\nu+1}(\upsilon)$ is dense in $X^{\mu}_{\nu}(\upsilon)$, it follows from the previous discussion that for each $t\ge 1$,
\begin{equation*}
    \|T(t)\|_{\mathcal{L}(X^\mu_{\nu}(\upsilon),X)}\lesssim t^{-n}. % \ \ \ (t\geq 1).
\end{equation*}

%\noindent{\bf{Claim:}} The subspace
%$X^{\mu}_{\nu+1,(1+\delta)\log}:=\text{Ran}%(A^{\mu}B^{-\nu-\mu}(a+\log(A)^{1+\delta}))^{-1}$ is dense in $X^{\mu}_{\nu,(1+\delta)\log}:=\text{Ran}(B^{-\nu}\log(1+B)^{-1-\delta})$.

In general, for each $s\geq 0$, let $n\in \mathbb{N}\cup\{0\}$ be such that $n<s<n+1$; then, there exists $\theta=\theta(s)\in (0,1)$ so that $s=(1-\theta)n+\theta(n+1)$. Let $\alpha_1:=\alpha(n+1)-1$, $\alpha_2:=\alpha(n+2)-1$, $\beta_1:=\beta(n+1)+1$ and  $\beta_2:=\beta(n+2)+1$, then $\alpha(s+1)-1=(1-\theta)\alpha_1+\theta\alpha_2$ and $\beta(s+1)+1=(1-\theta)\beta_1+\theta\beta_2$.  Set $\Tilde{\upsilon}:=c(\left \lceil{s}\right \rceil+1)+1+\delta $. Then, by a moment-like inequality (Lemma~4.2 in~\cite{rozendaal}), it follows that 
 \begin{eqnarray*}
   \|T(t)\Phi^{\mu}_{\nu}(A)(2\pi-i\log(A))^{-\Tilde{\upsilon}}\|_{\mathcal{L}(X)}&\lesssim&\|T(t)\Phi^{\alpha_1}_{\beta_1}(A)(2\pi-i\log(A))^{-\Tilde{\upsilon}}\|^{1-\theta}_{\mathcal{L}(X)}\cdot\\
   &&\|T(t)\Phi^{\alpha_2}_{\beta_2}(A)(2\pi-i\log(A))^{-\Tilde{\upsilon}}\|^{1-\theta}_{\mathcal{L}(X)}\lesssim t^{-s}.
\end{eqnarray*}
\end{proof}

The following result is analogous to Proposition \ref{prop3.1}; % and will assist us in obtaining the results of this subsection;
its proof is presented in Appendix~\ref{proofprop}.
\begin{proposition}\label{prop3.2}
\begin{rm}
Let $A\in \text{Sect}_X(\omega_A)$, with $\overline{\mathbb{C}_{-}}\setminus\{0\} \subset \rho(A)$. The following statements hold:
\begin{enumerate}[(a)]
\item Let $\alpha\geq 1$ and $a\geq 0$ be such that, for each $\lambda \in \overline{\mathbb{C}_{-}}\setminus\{0\}$ with $|\lambda|<1$, one has
\begin{equation}\label{eqq42}
\|(\lambda+A)^{-1}\|_{\mathcal{L}(X)}\lesssim |\lambda|^{-\alpha}(\log(1/|\lambda|))^a;
\end{equation}
then, for each $\zeta>1$,
\begin{equation*}
\left\{\lambda (2\pi-i\log(\lambda))^{\zeta}(\lambda+A)^{-1}\mid \lambda\in i\mathbb{R}\setminus\{0\}, |\lambda|<1\right\}\subset\mathcal{L}( X^{\alpha-1}(a+\zeta),X)
\end{equation*}
is uniformly bounded.
\item Let $\alpha\geq 1$, $\beta \ge 0$, $\beta_0\in[0,1)$ and $b\ge 0$. If 
\begin{equation}\label{eq12}
 \sup \{|\lambda|^{-\beta}\log(1+|\lambda|)^{-b}\Vert(\lambda+A)^{-1}\Vert\mid \lambda \in \overline{\mathbb{C}_{-}}\setminus\{0\}, |\lambda|\geq 1\}<\infty, % \subset \mathcal{L}(X)
\end{equation}
%is uniformly bounded,
then for each $\zeta>1$,
\begin{equation}\label{eq13}
\left\{\lambda^{\beta_0}(2\pi-i\log(\lambda))^{\zeta}  (\lambda+A)^{-1}\mid \lambda \in i\mathbb{R}, |\lambda|\geq 1\right\}\subset\mathcal{L}( X^{\alpha}_{\beta+\beta_0}(\zeta+b),X)
\end{equation}
is uniformly bounded. 
\end{enumerate}
\end{rm}
\end{proposition}

\begin{proposition}\label{prop3.3}
\begin{rm}
Let $A\in \text{Sect}_X(\omega_A)$ be such that $\overline{\mathbb{C}_{-}}\setminus\{0\} \subset \rho(A)$ and let $\alpha\ge 1$, $\beta, a,b\geq 0$. Then,  
%\begin{enumerate}[(a)]
%   \item
\begin{equation}\label{eqprop3.3}
 \|(\lambda+A)^{-1}\|_{\mathcal{L}(X)}\lesssim \left\{\begin{array}{cc}
  |\lambda|^{-\alpha}\log(1/|\lambda|)^{a}, & |\lambda|\leq 1 \\
|\lambda|^{\beta}\log(|\lambda|)^{b}, & |\lambda|\geq 1
\end{array}\right.
\end{equation}
implies
\begin{equation*}
\sup\{\|(\lambda+A)^{-1}\|_{\mathcal{L}(X^{\alpha n}_{\beta n}(cn),X)}\mid \lambda \in i\mathbb{R}\setminus\{0\}\}<\infty,
    \end{equation*}
where $c=\max\{a,b\}$ and $n\in \mathbb{N}$. 
%\end{enumerate}
\end{rm}    
\end{proposition}
%\noindent \textbf{Proof:} 

\begin{proof}
%The operator $A(1+A)^{-1}$ commutes with $(2\pi-i\log(A))^{-1}$ and by Graphic Closed Theorem $\log(2+A)^c(2\pi-i\log(A))^{-c}\in \mathcal{L}(X)$.
We consider the following cases.

\

$\bullet$~{\bf{Case 1: $\alpha=1$}}.

\

{\bf{Case 1(a): $c\in (0,1)$}}. Note that for each $\lambda \in \rho(-A)$,
\begin{equation}\label{eqq20}
(\lambda+A)^{-1}A(1+A)^{-1-\beta}(2\pi-i\log(A))^{-c}=\frac{1}{1+\lambda}\left((1+A)^{-1-\beta}-\lambda(\lambda+A)^{-1}(1+A)^{-\beta}\right)(2\pi-i\log(A))^{-c}.
\end{equation}

By the moment inequality (recall that $(2\pi-i\log(A))^{-1}$ is sectorial) and by~\eqref{eqq54}, it follows from~\eqref{eqprop3.3} that for each $|\lambda|<1$,
\begin{eqnarray}\label{eqq55a}
\nonumber&&\|\lambda (\lambda+A)^{-1}(1+A)^{-\beta}(2\pi-i\log(A))^{-c}\|_{\mathcal{L}(X)}\lesssim \|\lambda(\lambda+A)^{-1}(1+A)^{-\beta}\|^{1-c}_{\mathcal{L}(X)}\\
\nonumber &\cdot &\|\lambda(\lambda+A)^{-1}(1+A)^{-\beta}(2\pi-\log(A))^{-1}\|^{c}_{\mathcal{L}(X)}\\
\nonumber&&\lesssim \log(1/|\lambda|)^{c(1-c)} \left(\left\|\frac{\lambda(\lambda+A)^{-1}}{(2\pi-i\log(-\lambda))}\right\|_{\mathcal{L}(X)}+\frac{|\lambda|-1}{\log(|\lambda|)}\right)^{c}= \left(\frac{\left\|\frac{\lambda\log(|\lambda|)(\lambda+A)^{-1}}{(2\pi-i\log(-\lambda))}\right\|_{\mathcal{L}(X)}}{|\log(|\lambda|)|^{c}}+\frac{1-|\lambda|}{|\log(|\lambda|)|^{c}}\right)^{c}\\
&&\lesssim \left(\left\|\frac{\lambda(\lambda+A)^{-1}}{\log(|\lambda|)^c}\right\|_{\mathcal{L}(X)}+\frac{1-|\lambda|}{|\log(|\lambda|)|^c}\right)^{c}
\end{eqnarray}
Then, it follows from relations~\eqref{eqq20} and~\eqref{eqq55a} that for each $\lambda \in i \mathbb{R}\setminus\{0\}$, % with $|\lambda|\leq 1$ and by~\eqref{eqq20} and~\eqref{eqq55a}, it is follows 
\begin{equation*}
\sup\{\|(\lambda+A)^{-1}A(1+A)^{-1-\beta}(2\pi-i\log(A))^{-c}\|_{\mathcal{L}(X)}\mid \lambda \in i \mathbb{R}\setminus \{0\}, |\lambda|< 1\}<\infty.
\end{equation*}

Now, note that $A(1+A)^{-1}$ commutes with $(2\pi-i\log(A))^{-1}$, and by Closed Graph Theorem, $\log(2+A)^c(2\pi-i\log(A))^{-c}\in \mathcal{L}(X)$; thus,% for $|\lambda|\geq 1$,
\begin{eqnarray*}
  &&\frac{|\lambda|}{|1+\lambda|}\|(\lambda+A)^{-1}A(1+A)^{-1-\beta}(2\pi-i\log(A))^{-c}\|_{\mathcal{L}(X)}\\
  &&\lesssim \frac{|\lambda|}{|1+\lambda|}\|(\lambda+A)^{-1}(1+A)^{-\beta}(\log(2+A))^{-c}\|_{\mathcal{L}(X)},
\end{eqnarray*}
and so, it follows from Proposition~\ref{prop3.1} that
\begin{equation*}
    \sup_{\lambda \in i \mathbb{R}, |\lambda|\geq 1}\left\{\frac{|\lambda|}{|1+\lambda|}\|(\lambda+A)^{-1}A(1+A)^{-1-\beta}(2\pi-i\log(A))^{-c}\|_{\mathcal{L}(X)}\right\}<\infty.
\end{equation*}

\

\noindent {\bf{Case 1(b): $c=1$}}. 
It follows from~\eqref{eqprop3.3} and~\eqref{eqq54} that for each $\lambda \in i\mathbb{R}$ with $|\lambda|\leq 1$,
\begin{equation*}
    \|\lambda(\lambda+A)^{-1}A(1+A)^{-1-\beta}(2\pi-i\log(A))^{-1}\|_{\mathcal{L}(X)}\lesssim \left\|\frac{\lambda(\lambda+A)^{-1}}{(2\pi-i\log(-\lambda))}\right\|_{\mathcal{L}(X)}+ \frac{|\lambda|-1}{\log(|\lambda|)}\lesssim 1.
\end{equation*}

%By the hypothesis, $ \left\|\frac{\lambda(\lambda+A)^{-1}}{(2\pi-i\log(-\lambda))}\right\|\lesssim 1$
For $\lambda \in i \mathbb{R}$ with $|\lambda|\geq 1$,  one just proceeds as in {\bf{Case 1(a)}}.

\

\noindent {\bf{Case 1(c): $c>1$}}.
Note that for each $\lambda\in \rho(-A)$,
\begin{eqnarray}\label{eqq34a}
\nonumber(\lambda+A)^{-1}A(1+A)^{-1-\beta}(2\pi-i\log(A))^{-c}&=&%(\lambda+A)^{-1}(1+A)A(1+A)^{-2}(2\pi-i\log(A))^{-c}\\
%\nonumber&=& (A(\lambda+A)^{-1}A(1+A)^{-2}+(\lambda+A)^{-1}A(1+A)^{-2})(2\pi-i\log(A))^{-c}\\
%\nonumber&=&
(1-\lambda(\lambda+A)^{-1})A(1+A)^{-2-\beta}(2\pi-i\log(A))^{-c}\\
\nonumber &+& (\lambda+A)^{-1}A(1+A)^{-2-\beta}(2\pi-i\log(A))^{-c}\\
\nonumber &=& A(1+A)^{-2-\beta}(2\pi-i\log(A))^{-c}\\
\nonumber &+&(1-\lambda)(\lambda+A)^{-1}A(1+A)^{-2-\beta}(2\pi-i\log(A))^{-c}.\\
&&
\end{eqnarray}

Now, by Remark~\ref{remark2.2}, one has %equation~\eqref{eq1},
\begin{eqnarray}\label{eqq56a}
\nonumber(\lambda+A)^{-1}A(1+A)^{-2-\beta}(2\pi-i\log(A))^{-c} &=& \frac{1}{2\pi i}\int_{\Gamma} \frac{zR(z,A)}{(1+z)^{2+\beta}(2\pi-i\log(z))^{c}}dz(\lambda+A)^{-1}\\
\nonumber&=&\frac{(-\lambda)}{(2\pi-i\log(-\lambda))^{c}(1-\lambda)^{2+\beta}}(\lambda+A)^{-1}\\
\nonumber &+&\frac{1}{2\pi i}\int_{\Gamma} \frac{z R(z,A)}{(1+z)^{2+\beta}(2\pi-i\log(z))^{c}(\lambda+z)}dz,\\
&&
\end{eqnarray}
where $\Gamma$ is given as in the proof of Proposition $\ref{prop3.1}$. Since $A$ is sectorial, one can replace $\Vert R(z,A)\Vert_{\mathcal{L}(X)}$ by $1/|z|$, and so the function $z\mapsto (2\pi-i\log(z))^{-c}R(z,A)$ is integrable on $\Gamma$ (recall that $c>1$). Now, by letting $\gamma=\delta=1$ in Lemma A.1 in \cite{rozendaal}, it follows that for each $z\in \Gamma$ and $\lambda \in i \mathbb{R}\setminus\{0\}$,
\begin{equation}\label{eqq56b}
\left|\frac{z(1-\lambda)}{(1+z)^{2+\beta}(z+\lambda)}\right|\lesssim 1.
\end{equation}
Now, by~\eqref{eqprop3.3} and relations~\eqref{eqq34a},~\eqref{eqq56a} and \eqref{eqq56b}, it follows that 
\begin{eqnarray*}
\|(\lambda+A)^{-1}A(1+A)^{-1-\beta}(2\pi-i\log(A))^{-c}\|_{\mathcal{L}(X)} &\leq & \|A(1+A)^{-2-\beta}(2\pi-i\log(A))^{-c}\|_{\mathcal{L}(X)}\\
&+& |1-\lambda|\|(\lambda+A)^{-1}A(1+A)^{-2-\beta}(2\pi-i\log(A))^{-c}\|_{\mathcal{L}(X)}\\
&\lesssim& %\|A(1+A)^{-2}\|_{\mathcal{L}(X)}+
\left\|\frac{(-\lambda)(\lambda+A)^{-1}}{(2\pi-i\log(-\lambda))^{c}(1-\lambda)^{1+\beta}}\right\|_{\mathcal{L}(X)}+1\lesssim 1.
\end{eqnarray*}

\

$\bullet$~{\bf{Case 2: $\alpha>1$}}. Let $c>0$, and notice that
  \begin{eqnarray*}
(\lambda+ A)^{-1}A^{\alpha}(1+A)^{-\alpha-\beta}(2\pi-i\log(A))^{-c}&=&(\lambda+A)^{-1}(A+1)A^{\alpha}(1+A)^{-\alpha-\beta-1}(2\pi-i\log(A))^{-c}\\
&=& A^{\alpha}(1+A)^{-\alpha-\beta-1}(2\pi-i\log(A))^{-c}\\
&+& (1-\lambda)A^{\alpha}(1+A)^{-\alpha-\beta-1}(2\pi-i\log(A))^{-c}
\end{eqnarray*}

Hence, by Remark~\ref{remark2.2}, one has

\begin{eqnarray*}
(\lambda+A)^{-1}A^{\alpha}(1+A)^{-\alpha-\beta-1}(2\pi-i\log(A))^{-c} &=& \frac{1}{2\pi i}\int_{\Gamma} \frac{z^{\alpha}R(z,A)}{(1+z)^{\alpha+\beta+1}(2\pi-i\log(z))^{c}}dz (\lambda+A)^{-1}\\
&=&\frac{(-\lambda)^{\alpha}}{(2\pi-i\log(-\lambda))^{c}(1-\lambda)^{\alpha+\beta+1}}\\
&+& \frac{1}{2\pi i}\int_{\Gamma} \frac{z^{\alpha}R(z,A)}{(1+z)^{\alpha+\beta+1}(2\pi-i\log(z))^{c}(\lambda+z)}dz.
\end{eqnarray*}
It follows from~\eqref{eqprop3.3} that the first term in the right-hand side of the previous relation is bounded. As for the second term, let $\varepsilon\in (0,\min\{\alpha-1,1\})$ and consider the map $z\mapsto \dfrac{z^{\varepsilon}}{(1+z)^{2\varepsilon}}R(z,A)$, which in integrable over $\Gamma$; then, by letting $\gamma=\alpha-\varepsilon$ and $\delta=\beta+1-\varepsilon$ in Lemma A.1 in \cite{rozendaal}, one gets
\begin{equation*}
    \sup\left\{\frac{|z|^{\alpha-\varepsilon}|1-\lambda|}{|(1+z)^{\alpha+\beta+1-2\varepsilon}(2\pi-i\log(z))^c(\lambda+z)|}\mid z\in \Gamma, \lambda \in i \mathbb{R}\setminus\{0\}\right\}<\infty
\end{equation*}
(note that $(2\pi-i\log(z))^{-c}$ is uniformly bounded over $\Gamma$). Therefore, 
\begin{equation*}
    \sup_{\lambda \in i \mathbb{R}\setminus\{0\}} \|(\lambda+A)^{-1}A^{\alpha}(1+A)^{-\alpha-\beta}(2\pi-i\log(A))^{-c}\|_{\mathcal{L}(X)}<\infty.
\end{equation*}
    \end{proof}

\begin{proof4} We follow the same arguments presented in the proof of Theorem~\ref{theo4.4}. Let $n\in \mathbb{N}\cup\{0\}$ and set $\mu:=(n+1)\alpha-1$, $\nu:=(n+1)\beta+\frac{1}{r}$ and $\upsilon:=c(n+1)+1/r+\delta$. % if $p\neq 2$,  $\mu:=(n+1)\alpha-1$, $\nu:=(n+1)\beta$ and $\upsilon:=c(n+1)+\delta$ otherwise. Since the proof in both cases is identical, we assume that $p\neq 2$.
  By Proposition~\ref{prop3.3} one has, for each $k\in\{1,\ldots,n\}$,
\begin{equation}\label{eq29}
    \sup_{\xi \in \mathbb{R}}\|R(i\xi,A)^{k}\|_{\mathcal{L}(X^{\alpha n}_{\beta n}(cn),X)}<\infty.
\end{equation}

Let  $h_{r,\delta}:\mathbb{R}\rightarrow\mathbb{R}$ be given by the law $h_{r,\delta}(\xi)=\dfrac{|\xi|}{(1+|\xi|)^{1-\frac{1}{r}}}(2\pi+|\log(|\xi|)|)^{\frac{1}{r}+\delta}$, and note that for each $\xi\in\mathbb{R}$,
\begin{eqnarray*}   
h_{r,\delta}(\xi)\|R(i\xi,A)^{k}\|_{\mathcal{L}(X^\mu_{\nu}(\upsilon),X)} &\lesssim& h_{r,\delta}(\xi)\|R(i\xi,A)^{k}B^{-(\alpha(n+1)-1+\beta(n+1)+\frac{1}{r})}(2\pi-i\log(A))^{-\upsilon}\|_{\mathcal{L}(X)}\\
&\le& h_{r,\delta}(\xi)\|R(i\xi,A)A^{\alpha-1} B^{-(\alpha-1+\beta+\frac{1}{r})}(2\pi-i\log(A))^{-c-\frac{1}{r}-\delta}\|_{\mathcal{L}(X)}\\
&\cdot&\|R(i\xi,A)^{k-1}A^{\alpha n}B^{-(\alpha+\beta) n}(2\pi-i\log(A))^{-cn}\|_{\mathcal{L}(X)},
\end{eqnarray*}
where $B:=A+1$. It follows from Proposition~\ref{prop3.2} and relation~\eqref{eq29} that for each $k\in\{1,\ldots,n+1\}$,
\begin{equation}\label{eq28}
\sup_{\xi \in \mathbb{R}} h_{r,\delta}(\xi)\|R(i\xi,A)^{k}\|_{\mathcal{L}(X^\mu_{\nu}(\upsilon),X)} <\infty.
\end{equation}

As in the proof of Theorem~\ref{theo4.4}, let $\psi \in C_{c}(\mathbb{R})$ with be such that $\psi \equiv 1$ on $[-1, 1]$. It follows from~\eqref{eq28} that for each $k\in \{1,\ldots, n\}$,
\begin{equation*}   
\psi(\cdot)R(i\cdot,A)^{k}\in L^{1}(\mathbb{R},\mathcal{L}(X^\mu_{\nu}(\upsilon),X))\subset \mathcal{M}_{1,\infty}(\mathbb{R},\mathcal{L}(X^\mu_{\nu}(\upsilon),X)),
\end{equation*}
and $\|(1-\psi(\cdot))R(i\cdot,A)^{k}\|_{\mathcal{L}(X^\mu_{\nu}(\upsilon),X)} \in L^{r}(\mathbb{R})$. Note that $X^\mu_{\nu}(\upsilon)$ has Fourier type $p$, since $X^\mu_{\nu}(\upsilon)$ is isomorphic to $X$. Then,
\begin{equation*}
(1-\psi(\cdot))R(i\cdot, A)^{k} \in \mathcal{M}_{p,p'}\left(\mathbb{R},\mathcal{L}(X^\mu_{\nu}(\upsilon),X)\right),
\end{equation*}
and by Theorem~\ref{theor4.3}, there exists $c_n\geq 0$ such that for each $t\ge 1$,
%\begin{equation*}
%\|T(t)A^{\mu}(1+A)^{-(\mu+\nu)}(2\pi-i\log(A))^{-\upsilon}\|_{\mathcal{L}(X)}\leq c_n t^{-n},
%\end{equation*}
%or
\begin{equation*}
\|T(t)\|_{\mathcal{L}(X^\mu_{\nu}(\upsilon),X)}\le c_n t^{-n}.
\end{equation*}

In general, for each $s\geq 0$, set $\mu:=\alpha(s+1)-1$, $\nu:=\beta(s+1)+1/r$ and $\Tilde{\upsilon}:=c(\lceil{s} \rceil+1)+1/r+\delta$; by following the same argument presented in the proof of Proposition~\ref{theor3.3}, one concludes that for each $t\ge 1$, %with $n:=\lfloor s\rfloor \in \mathbb{N}\cup\{0\}$, then exists $\theta:=\theta(s)\in (0,1)$ with $s=(1-\theta)n+\theta(n+1)$. Therefore, by moment inequality (see Proposition 4.1 in \cite{chill}) 
\begin{eqnarray*}
\|T(t)A^{\mu}B^{-\mu-\nu}(2\pi-i\log(A))^{-\Tilde{\upsilon}}\|_{\mathcal{L}(X)}&\lesssim& %\|T(t)A^{\alpha(n+1)-1}B^{-(\alpha+\beta)(n+1)-1/r+1}(2\pi-i\log(A))^{-\Tilde{\upsilon}}\|^{1-\theta}_{\mathcal{L}(X)}\\
%&\cdot& \|T(t)A^{\alpha(n+2)-1}B^{-(\alpha+\beta)(n+2)-1/r+1}(2\pi-i\log(A))^{-\Tilde{\upsilon}}\|^{\theta}_{\mathcal{L}(X)}\\
t^{-s}.
\end{eqnarray*}
\end{proof4}

\subsection{Proof of Theorem~\ref{teo4.5}}
\begin{lemma}\label{lemma3.2}
\begin{rm}
Let $A$ be an injective sectorial operator defined in a Banach space $X$.  Let $\alpha\geq 1$, $\beta,s\geq 0$, $r\in[1,\infty]$ and $x\in \mathcal{D}(A)\subset X$. Then,
 \[A^{\alpha(s+1)-1}(1+A)^{-(\alpha+\beta)(s+1)+1-\frac{1}{r}}x\in \mathcal{D}(A)\cap \Ran(A).\] 
\end{rm}
\end{lemma}
%\noindent\textbf{Proof:} 
\begin{proof}
 
Let $x\in X$ and set $\mu:=\alpha(s+1)-1$, $\nu:=\beta(s+1)+\frac{1}{r}$ and $B:=1+A$; it follows from Proposition 3.1.1 in \cite{haase} (see items (c) and (f)) that
\begin{eqnarray*}
A^{\mu-1}B^{-(\mu+\nu)}x&=& A^{\mu-1}B^{-(\mu-1+1+\nu)}x\\
&=& B^{-1}\left(A^{\mu-1}B^{-(\mu-1+\nu)}x\right)\in \mathcal{D}(A).
\end{eqnarray*}
Thus, it follows that for each $x\in X$,
\begin{equation}\label{eq40}  
A^{\mu}B^{-(\mu+\nu)}x=A\left(A^{\mu-1}B^{-(\mu+\nu)}x\right) \in \Ran(A).
\end{equation}
Now one has, for each $x\in \dom(A)$, 
\begin{eqnarray*}   
A^{\mu}B^{-(\mu+\nu)}x&=&A^{\mu}B^{-(\mu+\nu)}B^{-1}Bx\\
&=& B^{-1}\left(A^{\mu}B^{-(\mu+\nu)}\right)Bx\in \mathcal{D}(A).
\end{eqnarray*}
\end{proof}
\begin{proof2}
Let $\delta>1-1/r$,  set $B:=A+1$ and for each $\varepsilon,s>0$, set $\mu:=\alpha(s+1)-1$, $\nu:=\beta(s+1)+1/r$,  $\upsilon:=c(n+1)+1/r+\delta$ (with $n=\lceil s\rceil$) and
\[Q_{\upsilon}(A_\varepsilon,B_{\varepsilon}):=A^{\mu-\varepsilon}B^{-\mu-\nu+\varepsilon}(2\pi-i\log(A))^{-\upsilon} \in \mathcal{L}(X).\]
%Since $1<1/r+\delta<2$, there exists  $\eta\in (0,1)$ so that $1/r+\delta=1+\eta$.
%As in the proof of Theorem \ref{teo4.5}
Set  $m:=\lceil\upsilon\rceil$, so $m\in\mathbb{N}\setminus\{1\}$ and $m-1<\upsilon\le m$.  %$\eta:=\{\upsilon\}$, and
We divide the proof into the cases $\upsilon=m$ and $\upsilon\in(m-1,m)$. %In both of them, we proceed recursively over $m\in\mathbb{N}$.

\

$\bullet$~\textbf{Case $\upsilon=m$.}

\

\textbf{Step 1:} estimating $\Vert T(t)\log(A)Q_\upsilon(A_0,B_0)\Vert$. For each $s>0$, let $\varepsilon=\dfrac{\min\{\alpha,\beta,1\}\theta}{2}>0$, where $\theta\in(0,\min\{1,s\})$; then, %by  combining Lemmas~\ref{lemma2.2},~\ref{lemma2.3} and~\ref{lemma3.2},
one has for each $x\in \dom(A)$,
\begin{eqnarray}\label{eq65a}
\nonumber T(t)\log(A)Q_{\upsilon}(A_0,B_{0})x  %&=& T(t)(2\pi-i\log(A))Q_{\upsilon}(A_0,B_{0})x\\
% &=& 2\pi T(t)Q_{\upsilon}(A_0,B_{0})x-iT(t)\log(A)Q_{\upsilon}(A_0,B_{0})x\\
 &=& T(t)\log(1+A)Q_{\upsilon}(A_0,B_{0})x -T(t)\log(1+A^{-1})Q_{\upsilon}(A_0,B_{0})x\\
 \nonumber&=& \frac{T(t)}{\varepsilon}\log((1+A)^{\varepsilon})Q_{\upsilon}(A_0,B_{0})x-\frac{T(t)}{\varepsilon}\log((1+A^{-1})^{\varepsilon})Q_{\upsilon}(A_0,B_{0})x\\
\nonumber&=& \frac{T(t)}{\varepsilon}  \int_{0+}^{\infty}f_{\varepsilon}(A)(\lambda+f_{\varepsilon}(A))^{-1}Q_{\upsilon}(A_0,B_{0})x d\mu(\lambda)-\\
\nonumber &-&\frac{T(t)}{\varepsilon}  \int_{0+}^{\infty}f_{\varepsilon}(A^{-1})(\lambda+f_{\varepsilon}(A^{-1}))^{-1}Q_{\upsilon}(A_0,B_{0})x d\mu(\lambda)\\
 &=& I_1-I_2,
\end{eqnarray}
with $f_\varepsilon(\lambda)=(1+\lambda)^\varepsilon-1$, where we have applied Lemmas~\ref{lemma2.2} and~\ref{lemma3.2} in the first identity,  Lemma~\ref{lemma2.3} in the second identity and relation~\ref{cbf} in the third identity.

\

\noindent {\textit{\bf{Estimating $I_1$}}}.  Let $\tau:=\dfrac{\|T(t)Q_{\upsilon}(A_0,B_{\varepsilon})\|_{\mathcal{L}(X)}}{\|T(t)Q_{\upsilon}(A_0,B_{0})\|_{\mathcal{L}(X)}}$; by following the same arguments presented in the proof of Theorem~\ref{theo4.4}, one gets
\begin{eqnarray*}
\left\|T(t)\int_{0+}^{\tau}f_{\varepsilon}(A)(\lambda+f_{\varepsilon}(A))^{-1}Q_{\upsilon}(A_0,B_{0})x d\mu(\lambda)\right\|
&\leq& 4\|T(t)Q_{\upsilon}(A_0,B_{0})\|_{\mathcal{L}(X)}M_{f_{\varepsilon}(A)}\int_{0+}^{\tau} \frac{\tau}{\lambda+\tau} d\mu(\lambda)\|x\|,
\end{eqnarray*}
and
\begin{eqnarray*}
\left\|T(t)\int_{\tau}^{\infty}f_{\varepsilon}(A)(\lambda+f_{\varepsilon}(A))^{-1}Q_{\upsilon}(A_0,B_{0})x d\mu(\lambda)\right\|
&\leq & 2\|T(t)f_{\varepsilon}(A)Q_{\upsilon}(A_0,B_{0})\|_{\mathcal{L}(X)}M_{f_{\varepsilon}(A)}\int_{\tau}^{\infty} \frac{1}{\lambda+\tau} d\mu(\lambda)\|x\|.
\end{eqnarray*}
Note that %since for each $t\geq 0$, $T(t)$ commutes with $(1+A)^{\varepsilon}$, it follows that
for each $t\geq 0$ and each $x\in X$,
\begin{eqnarray}\label{eq47}
\nonumber\|T(t)f_{\varepsilon}(A)Q_{\upsilon}(A_0,B_{0})x\|&=& %\|T(t)((1+A)^{\varepsilon}-1)Q_{\upsilon}(A_0,B_{0})x\|\\
%\nonumber&=& \|T(t)((1+A)^{\varepsilon}-(1+A)^{-\varepsilon}(1+A)^{\varepsilon})Q_{\upsilon}(A_0,B_{0})x\|\\
%\nonumber&=&\|T(t)(1-B^{-\varepsilon})B^{\varepsilon}Q_{\upsilon}(A_0,B_{0})x\|\\
%\nonumber&=&
\|(1-B^{-\varepsilon})T(t)B^{\varepsilon}Q_{\upsilon}(A_0,B_{0})x\|\\
\nonumber&\leq& \|1-B^{-\varepsilon}\|_{\mathcal{L}(X)} \|T(t)B^{\varepsilon}Q_{\upsilon}(A_0,B_{0})x\|\\
%&=& \|1-B^{-\varepsilon}\|_{\mathcal{L}(X)} \|T(t)B^{\varepsilon}A^{\alpha(s+1)-1}B^{-\varepsilon}B^{-(\alpha+\beta)(s+1)+\varepsilon}(2\pi-i\log(A))^{-\upsilon}x\|\\
&\leq& C_{\varepsilon} \|T(t)Q_{\upsilon}(A_0,B_{\varepsilon})\|_{\mathcal{L}(X)}\Vert x\Vert.
\end{eqnarray}
%Therefore, $\forall~t\geq 0$, 
%\begin{equation}\label{eq47}
%\|T(t)f_{\varepsilon}(A)Q_{\delta'}(A_0,B_{0})\|_{\mathcal{L}(X)}\leq C_{\varepsilon}\|T(t)Q_{\delta'}(A_0,B_{\varepsilon})\|_{\mathcal{L}(X)}.
%\end{equation}

Then, by \eqref{eq47}, it follows that
\begin{eqnarray}\label{eq47a}
\nonumber\|I_1\|&\lesssim& \varepsilon^{-1}\|T(t)Q_{\upsilon}(A_0,B_{0})\|_{\mathcal{L}(X)} \left( \int_{0+}^{\tau} \frac{\tau}{\tau+\lambda} d\mu(\lambda)+ C_\varepsilon\tau \int_{\tau}^{\infty} \frac{1}{\tau+\lambda} d\mu(\lambda)\right)\|x\|\\
%&\leq& \frac{4M_{f_{\varepsilon}(A)}}{\varepsilon}\|T(t)Q_{\delta'}(A_0,B_{0})\|_{\mathcal{L}(X)} \left( \int_{0+}^{\tau} \frac{\tau}{\tau+\lambda} d\mu(\lambda)+ \int_{\tau}^{\infty} \frac{C_{\varepsilon}\tau}{\tau+\lambda} d\mu(\lambda)\right)\|x\|\\
 &\lesssim & \varepsilon^{-1} \|T(t)Q_{\upsilon}(A_0,B_{0})\|_{\mathcal{L}(X)}\log(1+\tau)\|x\|.
\end{eqnarray}

\

\noindent{\textit{\bf{Estimating $I_2$}}}.
\
\

Let $\sigma:=\dfrac{\|T(t)Q_{\upsilon}(A_\varepsilon,B_{\varepsilon})\|_{\mathcal{L}(X)}}{\|T(t)Q_{\upsilon}(A_0,B_{0})\|_{\mathcal{L}(X)}}$, so
\begin{eqnarray*}
\left\|T(t)\int_{0+}^{\sigma}f_{\varepsilon}(A^{-1})(\lambda+f_{\varepsilon}(A^{-1}))^{-1}Q_{\upsilon}(A_0,B_{0})x d\mu(\lambda)\right\|
&\lesssim& \|T(t)Q_{\upsilon}(A_0,B_{0})\|_{\mathcal{L}(X)}\int_{0+}^{\sigma} \frac{\sigma}{\lambda+\sigma} d\mu(\lambda)\|x\|,
\end{eqnarray*}
and
\begin{eqnarray*}
\left\|\int_{\sigma}^{\infty}f_{\varepsilon}(A^{-1})(\lambda+f_{\varepsilon}(A))^{-1}T(t)Q_{\upsilon}(A_0,B_{0})x d\mu(\lambda)\right\|
&\lesssim& \|T(t)f_{\varepsilon}(A^{-1})Q_{\upsilon}(A_0,B_{0})\|_{\mathcal{L}(X)}\int_{\sigma}^{\infty} \frac{1}{\lambda+\sigma} d\mu(\lambda)\|x\|.
\end{eqnarray*}
Note that for each $x\in X$, 
\begin{eqnarray*}%\label{eq47b}
\nonumber\|T(t)f_{\varepsilon}(A^{-1})Q_{\upsilon}(A_0,B_{0})x\|%&=& \|T(t)((1+A^{-1})^{\varepsilon}-1)A^{\mu}B^{-(\mu+\nu)}(2\pi-i\log(A))^{-(m+\delta')}x\|\\
%&=& \|T(t)((1+A^{-1})^{\varepsilon}-(1+A^{-1})^{-\varepsilon}(1+A)^{\varepsilon})A^{\mu}B^{-(\mu+\nu)}(2\pi-i\log(A))^{-(m+\delta')}x\|\\
%&=& \|T(t)(1-(1+A^{-1})^{-\varepsilon})(1+A^{-1})^{\varepsilon}A^{\alpha(s+1)-1}B^{-(\alpha+\beta)(s+1)}(2\pi-i\log(A))^{-\upsilon}x\|\\
&=&\|(1-(1+A^{-1})^{-\varepsilon})T(t)(1+A^{-1})^{\varepsilon}A^{\mu}B^{-(\mu+\nu)}(2\pi-i\log(A))^{-\upsilon}x\|\\
&\leq& \Tilde{C_\varepsilon}\|T(t)(1+A^{-1})^{\varepsilon}A^{\mu}B^{-(\mu+\nu)}(2\pi-i\log(A))^{-\upsilon}x\|
\end{eqnarray*}

Now, by relation~\eqref{eqq2.7} one has $(1+A^{-1})^{-1}=1-(1+A)^{-1}=A(1+A)^{-1}$, so it follows from Propositions 3.1.1 (e) and 3.1.9 (b) in \cite{haase} that
\begin{eqnarray*}
[(1+A^{-1})^{\varepsilon}]^{-1}= %[(1+A^{-1})^{-1}]^{\varepsilon}\\
(A(1+A)^{-1})^{\varepsilon}=%\\
%&=& A^{\varepsilon}((1+A)^{-1})^{\varepsilon}\\
 A^{\varepsilon}((1+A)^{\varepsilon})^{-1}.
\end{eqnarray*}
Then, $(1+A^{-1})^{\varepsilon}=(1+A)^{\varepsilon}(A^{\varepsilon})^{-1}=(1+A)^{\varepsilon}A^{-\varepsilon}$ (see Proposition 3.2.1 (a) in \cite{haase}). Therefore, by the previous discussion,
\begin{eqnarray}\label{eq47c}
\nonumber\|T(t)f_{\varepsilon}(A^{-1})Q_{\upsilon}(A_0,B_{0})x\|&\leq& %\Tilde{C}_\varepsilon\|T(t)(A+1)^{\varepsilon}(A^{-1})^{\varepsilon}A^{\mu}B^{-(\mu+\nu)}(2\pi-i\log(A))^{-\upsilon}x\|\\
\Tilde{C}_\varepsilon\|T(t)A^{\mu-\varepsilon}B^{-(\mu+\nu)+\varepsilon}(2\pi-i\log(A))^{-\upsilon}x\|\\
&\leq& \Tilde{C}_\varepsilon\|T(t)Q_{\upsilon}(A_\varepsilon,B_{\varepsilon})\|_{\mathcal{L}(X)}\|x\|.
\end{eqnarray}

%Therefore, 
%\begin{equation*}
%\|T(t)f_{\varepsilon}(A^{-1})Q_{\delta'}(A_0,B_{0})\|_{\mathcal{L}(X)}\leq \Tilde{C}_{\varepsilon}\|T(t)Q_{\delta'}(A_\varepsilon,B_{\varepsilon})\|_{\mathcal{L}(X)},
%\end{equation*}
%where $C_{\varepsilon}:=\|(1-(1+A^{-1})^{-\varepsilon})\|_{\mathcal{L}(X)}$.  
Thus, by~\eqref{eq47c},
\begin{eqnarray}\label{eqq48a}
\nonumber\|I_2\|&\lesssim& \varepsilon^{-1}\|T(t)Q_{\upsilon}(A_0,B_0)\|_{\mathcal{L}(X)}\left( \int_{0+}^{\sigma} \frac{\sigma}{\sigma+\lambda} d\mu(\lambda)+\Tilde{C}_{\varepsilon}\sigma\int_{\sigma}^{\infty} \frac{1}{\sigma+\lambda} d\mu(\lambda)\right)\|x\|\\
&\lesssim& \varepsilon^{-1}\|T(t)Q_{\upsilon}(A_0,B_0)\|_{\mathcal{L}(X)}\log(1+\sigma)\|x\|.
\end{eqnarray}
%Therefore,
%By combining the previous estimates, we have, 
%\begin{align}
%\left\|\frac{T(t)}{\varepsilon}\log((1+A^{-1})^{\varepsilon})Q_{\delta'}(A_0,B_0)\right\|_{\mathcal{L}(X)}\lesssim \|T(t)Q_{\delta'}(A_0,B_0)\|_{\mathcal{L}(X)}\log\left(1+\frac{\Tilde{C}_{\varepsilon}\|T(t)Q_{\delta'}(A_\varepsilon,B_\varepsilon)\|_{\mathcal{L}(X)}}{\|T(t)Q_{\delta'}(A_0,B_0)\|_{\mathcal{L}(X)}}\right)
%\end{align}
%where the constant depends of $\varepsilon>0$ and $M_{f_{\varepsilon}(A^{-1})}$. 

\

Finally, by combining relations~\eqref{eq65a},~\eqref{eq47a} and~\eqref{eqq48a}, and by the density of $\dom(A)$, one gets for each sufficiently large $t$, %Proceeding as in the Proof of the Theorem $\ref{theo4.5}$ and by Theorem $\ref{theor2}$, 
\begin{eqnarray*}
\|T(t)\log(A)Q_{\upsilon}(A_0,B_0)\|_{\mathcal{L}(X)}&\leq& C_{\varepsilon}\|T(t)Q_{\upsilon}(A_0,B_0)\|_{\mathcal{L}(X)}\log\left(1+\frac{C_{\varepsilon}\|T(t)Q_{\upsilon}(A_0,B_\varepsilon)\|_{\mathcal{L}(X)}}{\|T(t)Q_{\upsilon}(A_0,B_0)\|_{\mathcal{L}(X)}}\right)\\
&+& \Tilde{C}_\varepsilon\|T(t)Q_{\upsilon}(A_0,B_0)\|_{\mathcal{L}(X)}\log\left(1+\frac{\Tilde{C}_{\varepsilon}\|T(t)Q_{\upsilon}(A_\varepsilon,B_\varepsilon)\|_{\mathcal{L}(X)}}{\|T(t)Q_{\upsilon}(A_0,B_0)\|_{\mathcal{L}(X)}}\right)\\
&\leq& 2C_{\varepsilon,s}t^{-s}\log(1+c_st^{s}),
\end{eqnarray*}
with $C_{\varepsilon,s}$ and $c_s$ positive constants, where in the last inequality we have applied Proposition~\ref{cor2} to $\log(1+\lambda)$ (see the proof of Theorem~\ref{theo4.5} for details).

\

\noindent{\bf{Step 2:}} removing $m$. The idea is to apply {\bf{Step 1}} recursively in order to obtain an estimate for $\|T(t)A^\mu B^{-\mu-\nu}\|_{\mathcal{L}(X)}$.

%Note that 

First of all, note that for each $k\in \mathbb{N}$ and each~$y\in \mathcal{D}(\log(A)^k)$, one has
\begin{equation}\label{eqq68a}
(2\pi-i\log(A))^{k}y=\sum_{n=0}^{k}\binom{k}{j}(2\pi)^{k-j}(i\log(A))^{j}y.
\end{equation}

Now, note that for each $n\in \{1,\ldots,m\}$ and each $x\in X$,
$(2\pi-i\log(A))^{-m}A^{\mu}B^{-(\mu+\nu)}x \in \dom((2\pi-i\log(A))^{m})\subset \mathcal{D}(\log(A)^{n})$, and so by~\eqref{eq40}, for each $x\in \mathcal{D}(A)$, one has
\begin{eqnarray*}
\mathcal{D}(A)&\ni& B^{-1}(2\pi-i\log(A))^{n}Q_{\upsilon}(A_0,B_0)Bx \\
&=& A^{\mu}B^{-(\mu+\nu)}(2\pi-i\log(A))^{n}(2\pi-i\log(A))^{-m}x\in \Ran(A).
\end{eqnarray*}

Therefore, it follows from relation~\eqref{eqq68a} that for each $x\in \dom(A)$,
\begin{eqnarray}\label{eqq71}
T(t)(2\pi-i\log(A))^{m}Q_{\upsilon}(A_0,B_0)x %&=& T(t)(2\pi+i\log(A))(2\pi+i\log(A))^{m-1}Q_{\upsilon}(A_0,B_0)x\\
= \sum_{n=0}^{m}\binom{m}{n}(-i)^n(2\pi)^{m-n}T(t)(\log(A))^{n}Q_{\upsilon}(A_0,B_0)x.
\end{eqnarray}

The next step consists in estimating the norm of each one of the terms presented in relation~\eqref{eqq71}.

\begin{itemize}
\item $n=0$. It follows from Theorem~\ref{teo4.4} that $\Vert T(t)Q_{\upsilon}(A_0,B_0)\Vert_{\mathcal{L}(X)}\lesssim t^{-s}$.

\item $1\le n\le m$. We proceed by induction over $n$. Case $n=1$ is just Step 1. If $n>1$, for each $0\le\varepsilon<1$ let
  \[\tau(t,A_\varepsilon,B_\varepsilon)=\frac{\|T(t)(\log(A))^{n-1}Q_{\upsilon}(A_\varepsilon,B_\varepsilon)\|_{\mathcal{L}(X)}}{\|T(t)(\log(A))^{n-1}Q_{\upsilon}(A_0,B_0)\|_{\mathcal{L}(X)}}\] and note that, by proceeding as in {\bf{Step 1}}, one gets
%\item $1<n\le m$.  We proceed by induction over $n$. Namely, %Note that, since in the Step 1, by Lemmas~\ref{lemma2.2},~\ref{lemma2.3} and ~\ref{lemma2.3}. Again by Corollary~\ref{cor2}, one gets for each sufficiently large $t$,
\end{itemize}
\begin{eqnarray*}
\|T(t)(\log(A))^{n}Q_{\upsilon}(A_0,B_0)\|_{\mathcal{L}(X)} &=& \Vert T(t)\log(A)(\log(A))^{n-1}Q_{\upsilon}(A_0,B_0)\Vert_{\mathcal{L}(X)}\\
&\lesssim& \varepsilon^{-1}\|T(t)(\log((1+A)^{\varepsilon})(\log(A))^{n-1}Q_{\upsilon}(A_0,B_0)\|_{\mathcal{L}(X)}\\
&+& \varepsilon^{-1}\|T(t)(\log((1+A^{-1})^{\varepsilon})(\log(A))^{n-1}Q_{\upsilon}(A_0,B_0)\|_{\mathcal{L}(X)}\\
&\lesssim& \|T(t)(\log(A))^{n-1}Q_{\upsilon}(A_0,B_0)\|_{\mathcal{L}(X)}\log\left(1+\tau(t,A_0,B_\varepsilon)\right)\\
&+&  \|T(t)(\log(A))^{n-1}Q_{\upsilon}(A_0,B_0)\|_{\mathcal{L}(X)}\log\left(1+\tau(t,A_\varepsilon,B_\varepsilon)\right);%\\
%&\lesssim& t^{-s}\log(1+t^s)\log\left(1+\frac{t^s}{\log(1+t^s)}\right),
\end{eqnarray*}
then, by the inductive hypothesis, it follows that %there exists a positive constant $c_{\varepsilon,n-1}$ such that,
for each $t\ge 1$,
\[\|T(t)(\log(A))^{n-1}Q_{\upsilon}(A_0,B_0)\|_{\mathcal{L}(X)}\lesssim t^{-s}\log(1+t)^{n-1}.\]

%It follows from relations ?? and Corollary~\ref{cor2} that for each sufficiently large $t$,
%\[\|T(t)(\log(A))^{n}Q_{\upsilon}(A_0,B_0)\|_{\mathcal{L}(X)}\lesssim t^{-s}\log(1+t)^n.
%\]

Now, by replacing the previous estimates on~\eqref{eqq71}, it follows that for each $t\ge 1$,
\begin{eqnarray*}
\|T(t)A^{\mu}B^{-(\mu+\nu)}\|_{\mathcal{L}(X)}&=& \|T(t)(2\pi-i\log(A))^{m}Q_{\upsilon}(A_0,B_0)\|_{\mathcal{L}(X)}\\
&\lesssim & t^{-s}\log(1+t)^{m}.
\end{eqnarray*}
%\end{itemize}
%\

\noindent\textbf{Case $\upsilon\neq m$.}

\

Since $\upsilon\in (m-1,m)$, it follows from the moment inequality (recall that $(2\pi-i\log(A))^{-1}$ is a sectorial operator) and from the previous case applied to $\upsilon=m$ and $\upsilon=m-1$ (recall that $m\ge2$) that for each $t\ge 1$,
\begin{eqnarray*}
\|T(t)A^{\mu}B^{-(\nu+\mu)}\|_{\mathcal{L}(X)}&=& \|T(t)(2\pi-i\log(A))^{\upsilon}Q_{\upsilon}(A_0,B_0)\|_{\mathcal{L}(X)}\\
&\lesssim&\|T(t)(2\pi-i\log(A))^{m-1}Q_{\upsilon}(A_0,B_0)\|^{m-\upsilon}_{\mathcal{L}(X)}\\
&\cdot&\|T(t)(2\pi-i\log(A))^{m}Q_{\upsilon}(A_0,B_0)\|^{\upsilon-m+1}_{\mathcal{L}(X)}\\
&\lesssim& (t^{-s}\log(1+t)^{m-1})^{m-\upsilon}(t^{-s}\log(1+t)^m)^{\upsilon-m+1}\\
&\lesssim& t^{-s}\log(1+t)^\upsilon.
%&\cdot& \log\left(1+\frac{t^{s}}
%{\log(1+c_st^{s})}\right)\log\left(1+\frac{t^{s}}{\log\left(1+\frac{t^{s}}{\log(1+c_st^{s})}\right)}\right)\cdots
%&\cdots& \log\left(1+\cfrac{t^s}{\log\left(1+\frac{t^{s}}
%{\log\left(1+\frac{t^s}{\log(1+c_st^{s})\right)})}\right)\
%\cdots\vphantom{\cfrac{t^s}{\log(1+c_st^{s})}}}\right)
\end{eqnarray*}
%Then, for all sufficiently large $t$,
%\begin{eqnarray*}
%\|T(t)A^{\alpha(s+1)-1}B^{-(\alpha+\beta)(s+1)+1-1/r}\|_{\mathcal{L}(X)}
%&\leq& C_{\varepsilon,\delta',s,a}t^{-s}\log(1+t)^{1+\delta'}=C_{\varepsilon,\delta',s,a}t^{-s}\log(1+t)^{c(\lfloor s\rfloor+1)+1/r+\delta}
%\end{eqnarray*}
    
\end{proof2}

\section{Singularity at zero}\label{sec5}
\zerarcounters

Let $\mu,\upsilon\geq 0$ and let $A\in \text{Sect}(\omega_A)$ be an injective operator over the Banach space $X$ (by Lemma~\ref{lemma21}, $A^{-1}$ is a sectorial operator); since $\lambda\mapsto \log(1+\lambda)\in\mathcal{CBF}$ (see Example~\ref{ex3.1}-(b)), it follows from Theorem~\ref{The2.3} that the operator $\log(2+A^{-1})$ is sectorial, hence $(\log(2+A^{-1}))^{-\upsilon}\in \mathcal{L}(X)$ is well-defined. % (see definition of fractional powers of sectorial operators in \cite{haase,mart}). 
Define the bounded operator
\begin{equation*}
\Phi^{\mu}(\upsilon)=\Phi^{\mu}(A,\upsilon):= A^{\mu}(1+A)^{-\mu}\log(2+A^{-1})^{-\upsilon}
\end{equation*}
and set $X^{\mu}(\upsilon):=\Ran(\Phi^{\mu}(\upsilon))$. The space $X^{\mu}(\upsilon)$  is a Banach space with respect to the norm
 \begin{eqnarray*}
\|x\|_{X^{\mu}(\upsilon)}&=&\|x\|+\|\Phi^{\mu}(\upsilon)^{-1}x\|=\|x\|+\|\log(2+A^{-1})^{\upsilon}(1+A)^{\mu}A^{-\mu}x\|, \ \ \ x \in X^{\mu}(\upsilon).
 \end{eqnarray*}

 Note that $\Phi^{\mu}(\upsilon):X\rightarrow  X^{\mu}(\upsilon)$ is an isomorphism, so for each $T\in \mathcal{L}(X^{\mu}(\upsilon),X)$,
 \begin{equation}\label{eqq73}
     \|T\|_{\mathcal{L}(X^{\mu}(\upsilon),X)} \leq  \|T\Phi^{\mu}(\upsilon)\|_{\mathcal{L}(X)}\leq \|\Phi^{\mu}(\upsilon)\|_{\mathcal{L}(X)}\|T\|_{\mathcal{L}(X^{\mu}(\upsilon),X)}. %, \ \ \  T\in \mathcal{L}(X_{\nu}(\upsilon),X).
 \end{equation}
 
%Note that $\Phi^{\mu}(0)=\Phi^{\mu}(A)$ and $X^{\mu}(0)=X^{\mu}$, where $\Phi^{\mu}(A)$ and $X^{\mu}$ are defined in \cite{rozendaal}. 

\DEFI\label{ULTIMA}
Let $(T(t))_{t\geq 0}$ be a $C_0$-semigroup $(T(t))_{t\geq 0}$ on a Banach space $X$ with generator $-A$. One defines the non-analytic growth bound $\zeta(T)$ of $(T(t))_{t\geq 0}$ as
\[\zeta(T):=\inf\{w\in \mathbb{R} \mid \sup_{t>0}e^{-tw}\|T(t)-S(t)\|_{\mathcal{L}(X)}<\infty\  \text{for some} \ S\in \mathcal{H}(\mathcal{L}(X))\},\]
where $\mathcal{H}(\mathcal{L}(X))$ is the set of the operators $S:(0,\infty)\rightarrow \mathcal{L}(X)$ having an exponentially bounded
analytic extension to some sector containing $(0,\infty)$.  One says that $(T(t))_{t\geq 0}$ is asymptotically analytic if $\zeta(T)<0$.
\DEFF

\OBSI Let 
\begin{eqnarray*}s^{\infty}_0(-A)&:=&\inf\Biggl\{w\in \mathbb{R}\mid \exists \ R\ge 0  \ \text{ such that} \ \{\text{Re}(\lambda)\geq w \ \text{and} \  |\text{Im}(\lambda)|\geq R\}\subset \rho(-A) \ \text{and} \\
&& \left. \sup_{\text{Re}(\lambda)\geq w, |\text{Im}(\lambda)|\geq R} \|(\lambda+A)^{-1}\|_{\mathcal{L}(X)}<\infty \right\}.
\end{eqnarray*}

It is shown in~\cite{black} (Proposition 2.4) that $\zeta(T)\geq s^{\infty}_0(-A)$. So, if  $(T(t))_{t\geq 0}$ is asymptotically analytic, then  $s^{\infty}_0(-A)<0$; more generally, Theorem 3.6 in~\cite{srivastava} states that $\zeta(T)<0$ if, and only if, $s^{\infty}_0(-A)<0$. In our strategy, we use the fact that $s^{\infty}_0(-A)<0$.
\OBSF

\subsection{Proof of Theorem~\ref{theor4.7}}

\begin{proof5} 
\noindent{\textbf{Step 1:}} Here, we use the same ideas presented in the proof of Theorem~\ref{theo4.5}. Let $n\in \mathbb{N}$ and set $\mu:=\alpha(n+1)-1$,  $\upsilon:=a(n+1)+1+\delta$. For each $x\in X^{\mu}(\zeta)$, let
\begin{eqnarray*}
y:=(\Phi^{\mu}_1(\upsilon))^{-1}x&=&\log(1+A^{-1})^{\upsilon}(1+A)^{\mu}A^{-\mu}x \\
&=& \log(2+A^{-1})^{\upsilon}(1+A)^{\mu}A^{-\mu}\left(A^{\mu}(1+A)^{-\mu-1}\log(2+A^{-1})^{-\upsilon}z\right)\\
&=& B^{-1}z\in \mathcal{D}(A),
\end{eqnarray*}
where $z:=(\Phi^{\mu}(\upsilon))^{-1}x$. Let $g:[0,\infty)\rightarrow X$ be defined by the law
\begin{equation*}
g(t):=\frac{1}{2\pi i} \int_{i\infty}^{-i\infty} e^{-\lambda t} \frac{\lambda^{\mu}}{(1+\lambda)^{\mu}\log(2+\lambda^{-1})^{\upsilon}}R(\lambda,A)yd\lambda.
\end{equation*}
Note that for each $t\ge 0$, $g(t)$ is indeed an element of $X$ (which follows from relation~\eqref{eeq42a} and from $s_{0}^{\infty}(-A)<0$) and it is differentiable. Namely, since $y\in \mathcal{D}(A)$, then $\lambda\mapsto \dfrac{\lambda^{\mu+1}}{(1+\lambda)^{\mu}\log(2+\lambda^{-1})^{\upsilon}}R(\lambda,A)y$ is integrable in $i\mathbb{R}$. Therefore, by dominated convergence,
\begin{equation*}
g'(t)=-\frac{1}{2\pi i}\int_{i\infty}^{-i\infty}e^{-\lambda t} \frac{\lambda^{\mu+1}}{(1+\lambda)^{\mu}\log(2+\lambda^{-1})^{\upsilon}}R(\lambda,A)yd\lambda.
\end{equation*}
Moreover, by Lemma \ref{lemmaB.1},
%\begin{eqnarray*}
%g'(t)&=& \frac{1}{2\pi i}\int_{i\infty}^{-i\infty}e^{-\lambda t} \frac{\lambda^{\mu}}{(1+\lambda)^{\mu}\log(1+\lambda^{-1})^{\upsilon}}(-AR(\lambda,A)y-y)d\lambda\\
%&=& -A\frac{1}{2\pi i}\int_{i\infty}^{-i\infty}e^{-\lambda t} \frac{\lambda^{\mu}}{(1+\lambda)^{\nu+\mu}\log(1+\lambda^{-1})^{\upsilon}}R(\lambda,A)yd\lambda-\\
%&-&\frac{1}{2\pi i}\int_{i\infty}^{-i\infty}e^{-\lambda t} \frac{\lambda^{\mu}}{(1+\lambda)^{\mu}\log(1+\lambda^{-1})^{\upsilon}}yd\lambda \\
%&=& -Ag(t)
%\end{eqnarray*}
%Then,
$g'(t)=-Ag(t)$ for each $t\geq 0$, and $g(0)=x$. Therefore, $g(t)=T(t)x$, by the uniqueness of the solution to the Cauchy problem associated with $-A$.

Now, integration by parts yields
\begin{equation*}
    t^n T(t)=\frac{1}{2\pi i} \int_{i\mathbb{R}}e^{-\lambda t} q(\lambda,A)y d\lambda,
\end{equation*}
where $q(\lambda,A)$ is a finite linear combination of terms of the form
\begin{equation*}
\frac{\lambda^{\mu-j}R(\lambda,A)^{n-k+1}}{(1+\lambda)^{\mu+k-j}(2\lambda+1)^{i}\log(2+\lambda^{-1})^{\upsilon+i}} \ \ \text{and} \ \ \frac{\lambda^{\mu-j}R(\lambda,A)^{n-k+1}}{(1+\lambda)^{\mu+i}(2\lambda+1)^l\log(2+\lambda^{-1})^{\upsilon+j}},
\end{equation*}
where $0 \leq i \le j \leq k \leq  n$ and $k-j\leq l \leq k-j+1$.

Then, for each $t>0$,
\begin{eqnarray*}
\|t^{n}T(t)x\|&\leq& \frac{1}{2\pi} \int_{i\mathbb{R}}|e^{-\lambda t}| \|q(\lambda,A)y\| d\lambda\\
&\leq& \frac{1}{2\pi} \int_{i\mathbb{R}} \|q(\lambda,A)\|_{\mathcal{L}(X)} d\lambda \|y\|\leq C\|(\Phi^{\mu}(\upsilon))^{-1}x\|\lesssim\|x\|_{X^{\mu}(\upsilon)}.
\end{eqnarray*}

Since $X^\mu_{1}(\upsilon)$ is dense in $X^{\mu}(\upsilon)$, it follows from the previous discussion that for each $t\ge 1$,
\begin{equation*}
    \|T(t)\|_{\mathcal{L}(X^\mu(\upsilon),X)}\lesssim t^{-n}. % \ \ \ (t\geq 1).
\end{equation*}

It remains to prove the result for any $s>0$. So, for each fixed $s>0$, let $n\in \mathbb{N}$ be such that $n\le s<n+1$. Let also define $\theta:=\theta(s)\in [0,1)$ by the relation $s=(1-\theta)n+\theta(n+1)$.

  Set $a_1:=\frac{\alpha}{\alpha+a}$ and $a_2:=\frac{a}{\alpha+a}$ and note that $a_1+a_2=1$; then, by Proposition~\ref{theor2.3}-(c), $f(\lambda)=(1+\lambda)^{a_1}\log(2+\lambda)^{a_2} \in \mathcal{CBF}$, where $\lambda>0$. Now, by  Lemma 3.2 in \cite{chill}, the operator
\begin{eqnarray*}
(f(A^{-1}))^{-1}&=&(1+A^{-1})^{-a_1}\log(2+A^{-1})^{-a_2}\\
&=& A^{-a_1}(1+A)^{-a_1}\log(2+A^{-1})^{-a_2}
\end{eqnarray*}
is sectorial, given that $f(A^{-1})$ is sectorial, by Theorem~\ref{The2.3}.

Since $(f(A^{-1}))^{-1}$ is sectorial, it follows from the moment inequality (see Proposition 4.6 in \cite{haase}) and Theorem 2.4.2 in \cite{haase} that
\begin{eqnarray}\label{eqq77a}
\nonumber\|T(t)[(f(A^{-1}))^{-1}]^{\theta(\alpha+a)}\Phi^{\mu}(\upsilon)\|_{\mathcal{L}(X)} &\lesssim& \|T(t)\Phi^{\mu}(\upsilon)\|^{1-\theta}_{\mathcal{L}(X)}\|T(t)[(f(A))^{-1}]^{\alpha+a}\Phi^{\mu}(\upsilon)\|^{\theta}_{\mathcal{L}(X)}\\
\nonumber&=& \|T(t)\Phi^{\mu}(\upsilon)\|^{1-\theta}_{\mathcal{L}(X)}\|T(t)A^{\alpha}(1+A)^{-\alpha}\log(2+A^{-1})^{-a}\Phi_{\nu}(a)\|^{\theta}_{\mathcal{L}(X)}\\
\nonumber&=& \|T(t)\Phi^{\mu}(\upsilon)\|^{1-\theta}_{\mathcal{L}(X)}\|T(t)\Phi^{\alpha(n+2)-1}(a(n+2)+1+\delta)\|^{\theta}_{\mathcal{L}(X)}\\
&\lesssim& t^{-n(1-\theta)}t^{-\theta(n+1)}=t^{-s}.
\end{eqnarray}

\noindent \textbf{Step 2.}
\
For each $\varepsilon>0$, set
\[W_{\upsilon}(A_{\varepsilon},B_{\varepsilon}):=A^{\mu-\varepsilon}B^{-\mu+\varepsilon}\log(1+A^{-1})^{-\upsilon}\in \mathcal{L}(X).\]
Set $m:=\lfloor\upsilon\rfloor$ and $\eta:=\{\upsilon\}$. As in the proof of Theorem~\ref{theo4.5}, we divide the proof into the cases where $\eta=0$ and $\eta>0$. In both of them, we proceed recursively over $m\in\mathbb{N}$.

\

\noindent{\textbf{Case $\eta>0$.}}
 
$\bullet$~\textbf{Removing $\eta$}.

Let $\varepsilon=\frac{\alpha\theta }{2}>0$, where $\theta \in (0,\min\{1,s\})$. Note that for each $x\in \dom (A)$, one has
\begin{eqnarray*}
T(t)W_{m}(A_0,B_0)x &=& \frac{1}{\varepsilon^{\eta}}T(t)\log(1+f_\varepsilon(A^{-1}))^{\eta}W_{m+\eta}(A_0,B_0))x\\
&=& \frac{1}{\varepsilon^{\eta}} \int_{0}^{\infty} f_\varepsilon (A^{-1})(s+f_{\varepsilon}(A^{-1}))^{-1}T(t)W_{m+\eta}(A_0,B_0)xd\mu(s).
\end{eqnarray*}
%where $\mu$ now stands for the Borel measure related to the integral representation of $\log(1+\lambda)^{\eta}$ (which is a complete Bernstein function).

Let 
\[\tau:=\frac{\|T(t)W_{m+\eta}(A_\varepsilon,B_\varepsilon)\|_{\mathcal{L}(X)}}{\|T(t)W_{m+\eta}(A_0,B_0)\|_{\mathcal{L}(X)}};\]
then, by proceeding as in the proof of Theorem~\ref{theo4.5}, one gets
\begin{eqnarray}\label{eqq79}
\|T(t)W_{m}(A_0,B_0)\|_{\mathcal{L}(X)} &\lesssim& \|T(t)W_{m+\eta}(A_0,B_0)\|_{\mathcal{L}(X)}\log(1+\tau)^{\eta}.  
\end{eqnarray}
Again, by combining the estimates~\eqref{eqq77a} and~\eqref{eqq79} with the arguments presented in the proof of Theorem~\ref{theo4.5}, it follows that for each $t\ge 1$,  
\begin{eqnarray}\label{eqq77b}
\|T(t)W_{m}(A_0,B_0)\|_{\mathcal{L}(X)} \lesssim t^{-s}\log(1+t)^{\eta}.
\end{eqnarray}

\

$\bullet$~{\bf{Removing $m$}}. It follows from the discussion presented in the previous item that for each $x\in X$, 
\begin{eqnarray*}
T(t)A^{\mu}B^{-\mu}\log(1+A^{-1})^{-m+1}x&=&\frac{1}{\varepsilon}T(t)\log(1+f_{\varepsilon}(A^{-1}))A^{\mu}B^{-\mu}\log(1+A^{-1})^{-m}x\\
&=& \frac{1}{\varepsilon}T(t)\int_{0+}^{\infty}f_{\varepsilon}(B) (\lambda+f_{\varepsilon}(B))^{-1}W_{m}(A_0,B_0)xd\mu(\lambda).
\end{eqnarray*}
%where $\mu$ now stands for the Borel measure related to the integral representation of $\log(1+\lambda)$ (which is a complete Bernstein function).

Let $\tau:=\dfrac{\|T(t)W_{m}(A_\varepsilon,B_\varepsilon)\|_{\mathcal{L}(X)}}{\|T(t)W_{m}(A_0,B_0)\|_{\mathcal{L}(X)}}$; then, by proceeding as in the proof of Theorem~\ref{theo4.5}, it follows from relation~\eqref{eqq77b} that for each $t\ge 1$,
\begin{eqnarray*}
\|T(t)W_{m-1}(A_0,B_0)\|_{\mathcal{L}(X)}\lesssim t^{-s}\log(1+t)^{1+\eta}.
\end{eqnarray*}

By proceeding recursively over $m$ (see the proof of Theorem~\ref{theo4.5} for details), it follows from the previous discussion that for each $t\ge 1$,
\begin{eqnarray*}
\nonumber\|T(t)A^{\mu}B^{-\mu}\|_{\mathcal{L}(X)}\lesssim t^{-s}\log\left(1+ t\right)^{m+\eta}.
\end{eqnarray*}

\

\noindent{\textbf{Case $\eta=0$.}} Since in this case $\upsilon = m\in\mathbb{N}$, one just needs to proceed as in the previous item in order to conclude that for each $t\ge 1$,
\begin{eqnarray*}
\nonumber\|T(t)A^{\mu}B^{-\mu}\|_{\mathcal{L}(X)}\lesssim t^{-s}\log\left(1+ t\right)^{m}.
\end{eqnarray*}
\end{proof5}

%\begin{proof}
%\end{proof}

\begin{remark}
\begin{rm}
Suppose that $\alpha=1$ in the statement of Corollary~\ref{cor3.2}, so $\|R(\lambda,A)\|_{\mathcal{L}(X)}\lesssim |\lambda|^{-1}$ with $\text{Re}(\lambda)\leq 0$. Then, by relation \eqref{teo43}, for each $\sigma>0$ there exists $C_{\delta,\sigma}>0$ so that for each $t\ge 1$,
\begin{equation}
\|T(t)A^{\sigma}(1+A)^{-\sigma}\|_{\mathcal{L}(X)}\leq C_{\sigma,\delta} t^{-\sigma}\log(1+t)^{1+\delta}.
\end{equation}
Now, for each $\sigma>1$, take $\delta=\sigma-1>0$, and so for each $t\ge 1$,

\begin{equation*}
\|T(t)A^{\sigma}(1+A)^{-\sigma}\|_{\mathcal{L}(X)}\leq C_{\sigma,\delta} t^{-\sigma}\log(1+t)^{\sigma}=O\left(\left(\frac{1}{M^{-1}_{\log}(t)}\right)^{\sigma}\right).
\end{equation*}

This shows that in case $\alpha=1$ and $\sigma>1$, one gets the same estimate as in Corollary 2.12 in~\cite{chill1}  for bounded $C_0$-semigroups, and so for these particular parameters, the result is optimal.
\end{rm}
\end{remark}

\appendix
\section*{Appendix}
\section{Proof of Proposition \ref{prop3.2}}\label{proofprop}
\zerarcounters
\noindent {\bf{Item (a).}} Let $\zeta>1$ and set $\Tilde{c}:=\zeta+a$.
%\begin{itemize}

\
 
$\bullet$~\textbf{Case 1: $\alpha=1$.} 

\

{\bf{Case 1(a): $\Tilde{c} \in (1,2]$}}.   
Note that in this case, %$c=\max\{a,b\}\in[0,1)$ 
$a\in [0,1)$. Set $h_{\alpha,\zeta}(\lambda)=\lambda^{\alpha}(2\pi-i\log(\lambda))^{\zeta}$, with $\lambda \in i\mathbb{R}\setminus\{0\}$, and define the operator $L_{\nu,\Tilde{c}}(A):=(1+A)^{-\nu}(2\pi-i\log(A))^{-\Tilde{c}}\in \mathcal{L}(X)$. Since $(\lambda+A)^{-1}$ commutes with $L_{\nu,\Tilde{c}}(A)$, it follows from the moment inequality that 
\begin{equation}\label{eqq53}
\|h_{1,\zeta}(\lambda)(\lambda+A)^{-1}L_{\nu,\Tilde{c}}(A)\|_{\mathcal{L}(X)} \lesssim \|h_{1,1-a}(\lambda)(\lambda+A)^{-1}L_{\nu,1}(A)\|^{2-\Tilde{c}}_{\mathcal{L}(X)}\|h_{1,2-a}(\lambda)(\lambda+A)^{-1}L_{\nu,2}(A)\|^{\Tilde{c}-1}_{\mathcal{L}(X)}. 
\end{equation}

Let $\varepsilon>0$, set $A_\varepsilon:=(A+\varepsilon)(1+\varepsilon A)^{-1}$ and note that $A^{-1}_\varepsilon \in \mathcal{L}(X)$. For each  $\lambda \in i\mathbb{R}\setminus\{0\}$, let $r\in  (0,|\lambda|/2]$ and $R\geq 2|\lambda|+2$ be such that $\sigma(A_\varepsilon)\subset\{z\in\mathbb{C}\mid r<|z|<R\}$, let $\theta\in(\pi/2,\pi)$ and set $\gamma_{+}=\{se^{i\theta}\mid s\in [r,R]\}$, $\gamma_{-}=\{te^{-i\theta}\mid t\in [r,R]\}$, $\gamma_{r}=\{re^{is}\mid s\in [-\theta,\theta]\}$, $\gamma_{R}=\{Re^{is}\mid s\in [-\theta,\theta]\}$ and $\gamma:=\gamma_{+}\cup \gamma_{-}\cup \gamma_{r}\cup \gamma_{R}$. Then, by the Riesz-Dunford functional calculus (see~\eqref{eqq14}), for each $x\in X$ (here, $y:=(1+A)^{-\nu}x$),
\begin{eqnarray*}
\nonumber h_{1,1-a}(\lambda)(\lambda+A_\varepsilon)^{-1}(2\pi-i\log(A_\varepsilon))^{-1}y &=& \frac{h_{1,1-a}(\lambda)}{2\pi i} \int_{\gamma} \frac{1}{(2\pi-i\log(z))}R(z,A_\varepsilon) (\lambda+A_\varepsilon)^{-1}ydz\\
&=& \frac{h_{1,1-a}(\lambda)}{2\pi i} \int_{\gamma} \frac{1}{(2\pi-i\log(z))(\lambda+z)}dz(\lambda+A_\varepsilon)^{-1} y+\\
&+& \frac{h_{1,1-a}(\lambda)}{2\pi i} \int_{\gamma} \frac{1}{(2\pi-i\log(z))(\lambda+z)}R(z,A_\varepsilon) y dz
\end{eqnarray*}
\begin{eqnarray*}
&=&\frac{h_{1,1-a}(\lambda)(\lambda+A_\varepsilon)^{-1}y}{2\pi-i\log(-\lambda)}
+\frac{1}{2\pi i} \int_{r}^{R}\frac{h_{1,1-a}(\lambda)e^{-i\theta}R(te^{-i\theta},A_\varepsilon)y}{(2\pi-\theta-i\log(t))(\lambda+te^{-i\theta})} dt\\
&-& \frac{h_{1,1-a}(\lambda)}{2\pi i} \int_{r}^{R}\frac{e^{i\theta}}{(2\pi+\theta-i\log(t))(\lambda+te^{i\theta})}R(te^{i\theta},A_\varepsilon)y dt\\
&+& \frac{h_{1,1-a}(\lambda)}{2\pi i} \int_{-\theta}^{\theta}\frac{iRe^{is}}{(2\pi-s+i\log(R))(\lambda+Re^{is})}R(Re^{is},A_\varepsilon)y ds\\
&-& \frac{h_{1,1-a}(\lambda)}{2\pi i} \int_{-\theta}^{\theta}\frac{ire^{is}}{(2\pi-s+i\log(r))(\lambda+re^{is})}R(re^{is},A_\varepsilon)yds,
\end{eqnarray*}
where we have used the residue theorem in the third identity. By taking the limit $\theta \to \pi$ on both sides of the identity above, one gets 
\begin{eqnarray*}
  \nonumber h_{1,1-a}(\lambda)(\lambda+A_\varepsilon)^{-1}(2\pi-i\log(A_\varepsilon))^{-1}y&=& \frac{h_{1,1-a}(\lambda)}{2\pi-i\log(-\lambda)}(\lambda+A_\varepsilon)^{-1}y\\
  &+&\frac{1}{2\pi i} \int_{r}^{R}\frac{h_{1,1-a}(\lambda)}{(\pi-i\log(t))(\lambda-t)}(t+A_\varepsilon)^{-1}y dt\\
  &-& \frac{1}{2\pi i} \int_{r}^{R}\frac{h_{1,1-a}(\lambda)(t+A_\varepsilon)^{-1}y}{(3\pi-i\log(t))(\lambda-t)} dt\\
 &+& \frac{1}{2\pi i} \int_{-\pi}^{\pi}\frac{ih_{1,1-a}(\lambda)Re^{is}  R(Re^{is},A_\varepsilon)y}{(2\pi-s-i\log(R))(\lambda+Re^{is})} ds\\
&-& \frac{h_{1,1-a}(\lambda)}{2\pi i} \int_{-\pi}^{\pi}\frac{ire^{is}}{(2\pi+s-i\log(r))(\lambda+re^{is})}R(re^{is},A_\varepsilon)y ds
\end{eqnarray*}
Now, by taking the limits $r\to 0$ and $R\to \infty$ on both sides of the last identity,  one gets for each $x\in X$,
\begin{eqnarray*}
  h_{1,1-a}(\lambda)(\lambda+A_\varepsilon)^{-1}(2\pi-i\log(A_\varepsilon))^{-1}y&=& \frac{h_{1,1-a}(\lambda)(\lambda+A_\varepsilon)^{-1}y}{2\pi-i\log(-\lambda)}\\
  &+& \int_{0}^{\infty} \frac{ih_{1,1-a}(\lambda)}{(3\pi^2-4\pi i\log(t)-\log(t)^2)(\lambda-t)}(t+A_\varepsilon)^{-1}y\,dt.
\end{eqnarray*}
Finally, by taking the limit $\varepsilon\to 0^{+}$ on both hands of the identity above, one gets%by dominated convergence, by remark ~\ref{re2.1}..., we obtained
\begin{eqnarray}\label{eq77}
\nonumber h_{1,1-a}(\lambda)(\lambda+A)^{-1}(2\pi-i\log(A))^{-1}y&=&\frac{h_{1,1-a}(\lambda)(\lambda+A)^{-1}y}{2\pi-i\log(-\lambda)}\\
&+&\int_{0}^{\infty} \frac{ih_{1,1-a}(\lambda)(t+A)^{-1}y}{(3\pi^2-4\pi i\log(t)-\log(t)^2)(\lambda-t)}dt,
\end{eqnarray}
where we have used on the left-hand side that $(\lambda+A_\varepsilon)^{-1}\rightarrow (\lambda+A)^{-1}$
 uniformly (by Lemma~\ref{lemma21}), $(2\pi-i\log(A_\varepsilon))^{-1}\rightarrow (2\pi-i\log(A))^{-1}$ strongly (see the proof of Lemma 3.5.1~\cite{haase}), and on the right-hand side dominated convergence.

Then, by~\eqref{eq77}, one gets
\begin{eqnarray} \label{eqq54}
  \nonumber &&|h_{1,1-a}(\lambda)|\left\|(\lambda+A)^{-1}(2\pi-i\log(A))^{-1}(1+A)^{-\nu}\right\|_{\mathcal{L}(X)}\\
\nonumber  &&\lesssim \left\|\frac{h_{1,1-a}(\lambda)(\lambda+A)^{-1}}{2\pi-i\log(-\lambda)}\right\|_{\mathcal{L}(X)}+ \int_{0}^{\infty} \frac{|h_{1,1-a}(\lambda)| }{(\pi^2+\log(t)^{2})|\lambda-t|}\|(t+A)^{-1}\|_{\mathcal{L}(X)} dt\\
\nonumber &&\lesssim \left\|\frac{h_{1,1-a}(\lambda)(\lambda+A)^{-1}}{2\pi-i\log(-\lambda)}\right\|_{\mathcal{L}(X)}+\int_{0}^{\infty} \frac{|h_{1,1-a}(\lambda)| }{t(\pi^2+\log(t)^{2})|(\lambda|+t)} dt\\
\nonumber &&\lesssim \left\|\frac{h_{1,1-a}(\lambda)(\lambda+A)^{-1}}{2\pi-i\log(-\lambda)}\right\|_{\mathcal{L}(X)}+\int_{0}^{\infty}\frac{|h_{1,1-a}(\lambda)|(t+1)}{t(\pi^2+\log(t)^2)(|\lambda|+t)}dt\\
&&=\left\|\frac{h_{1,1-a}(\lambda)(\lambda+A)^{-1}}{2\pi-i\log(-\lambda)}\right\|_{\mathcal{L}(X)}+
 \frac{|h_{1,1-a}(\lambda)|(|\lambda|-1)}{|\lambda|\log(|\lambda|)},
\end{eqnarray}
where we have used relation~\eqref{loginver} in the last identity.

Note that for each $\lambda \in i\mathbb{R}\setminus\{0\}$ with $|\lambda|\leq 1$, it follows from~\eqref{eqq42} that
\begin{equation*}
    \left\|\frac{h_{1,1-a}(\lambda)(\lambda+A)^{-1}}{2\pi-i\log(-\lambda)}\right\|_{\mathcal{L}(X)}\lesssim 1,
\end{equation*}
and since for each $\eta>0$, $\displaystyle{\lim_{|\lambda|\to 0^{+}} |\lambda|\log(|\lambda|)^{\eta}=0}$, one gets
\[\frac{|h_{1,1-a}(\lambda)|(|\lambda|-1)}{|\lambda|\log(|\lambda|)}\le\frac{(2\pi+|\log(|\lambda|)|)^{1-a}(|\lambda|-1)}{\log(|\lambda|)}
\lesssim |\lambda||\log(|\lambda|)|^{-a}
\stackrel{|\lambda|\to 0^+}{\longrightarrow} 0\]
and $|h_{1,1-a}(\lambda)|\rightarrow 0$ as $|\lambda|\rightarrow 0^{+}$. Hence, one concludes that %Then, by~\eqref{eqq54}, 
\begin{equation}\label{eqq55}
\sup\{\left\|h_{1,1-a}(\lambda)(\lambda+A)^{-1}(1+A)^{-\nu}(2\pi-i\log(A))^{-1}\right\|_{\mathcal{L}(X)} \mid  \lambda \in i\mathbb{R}\setminus\{0\}, |\lambda|\leq 1\}<\infty.
\end{equation}

Now, by using the same ideas as before, one has for each $\varepsilon>0$ and each $x\in X$,
\begin{eqnarray*}
&&h_{1,2-a}(\lambda)(\lambda+A_\varepsilon)^{-1}(2\pi-i\log(A_\varepsilon))^{-2}(1+A)^{-\nu}x = \frac{h_{1,2-a}(\lambda)(\lambda+A_\varepsilon)^{-1}(1+A)^{-\nu}x}{(2\pi-i\log(\lambda))^2}\\
&&-\int_{0}^{\infty} \frac{2ih_{1,2-a}(\lambda)(2\pi-i\log(t))(t+A_\varepsilon)^{-1}(1+A)^{-\nu}x}{(3\pi^2-4\pi i\log(t)-\log(t)^2)^2(\lambda-t)} dt.
\end{eqnarray*}

So, by taking the limit $\varepsilon\rightarrow 0^+$ on both sides of the identity, one gets %by Dominated Convergence, we have
\begin{eqnarray*}
&&  h_{1,2-a}(\lambda)(\lambda+A)^{-1}(2\pi-i\log(A))^{-2}(1+A)^{-\nu}x =  \frac{h_{1,2-a}(\lambda)(\lambda+A)^{-1}(1+A)^{-\nu}x}{(2\pi-i\log(\lambda))^2}\\
  &&-\int_{0}^{\infty} \frac{2ih_{1,2-a}(\lambda)(2\pi-i\log(t))(t+A)^{-1}(1+A)^{-\nu}x}{(3\pi^2-4\pi i\log(t)-\log(t)^2)^2(\lambda-t)} dt.
\end{eqnarray*}
Then,
\begin{eqnarray*}
&&\left\|h_{1,2-a}(\lambda)(\lambda+A)^{-1}(2\pi-i\log(A))^{-2}\right\|_{\mathcal{L}(X)}\lesssim 
  \left\|\frac{h_{1,2-a}(\lambda)(\lambda+A)^{-1}}{(2\pi-i\log(\lambda))^2}\right\|_{\mathcal{L}(X)}\\
  &&+ \int_{0}^{e^{-2\pi}} \frac{|h_{1,2-a}(\lambda)||\log(t)|}{t(\pi^{2}+\log(t)^2)^2(|\lambda|+t)}dt+\int_{e^{-2\pi}}^{e^{2\pi}} \frac{|h_{1,2-a}(\lambda)|(|\log(t)|+2\pi)}{|(3\pi^2-4\pi i\log(t)-\log(t)^2)^2|}\frac{dt}{t}\\
&&+\int_{e^{2\pi}}^{\infty} \frac{|h_{1,2-a}(\lambda)| |\log(t)|}{t(\pi^{2}+\log(t)^2)^2(|\lambda|+t)}dt
\end{eqnarray*}
\begin{eqnarray*}
&\lesssim& \left\|\frac{h_{1,2-a}(\lambda)(\lambda+A)^{-1}}{(2\pi-i\log(\lambda))^2}\right\|_{\mathcal{L}(X)} +|h_{1,2-a}(\lambda) |\int_{0}^{\infty} \frac{\pi^2-2(1+1/t)\log(t)+\log(t)^2}{(\pi^2+(\log(t))^{2})^2(|\lambda|+t)} dt +|h_{1,2-a}(\lambda)|\\ 
&=& \left\|\frac{h_{1,2-a}(\lambda)(\lambda+A)^{-1}}{(2\pi-i\log(\lambda))^2}\right\|_{\mathcal{L}(X)}+ \frac{|h_{1,2-a}(\lambda)|(|\lambda|\log(|\lambda|)-|\lambda|+1 +|\lambda|\log(|\lambda|)^2)}{|\lambda|\log(|\lambda|)^2},
\end{eqnarray*}
where we have used relation~\eqref{logquad} in the last identity.

By using the same reasoning as before, one concludes that 
\begin{equation}\label{eqq56}
\sup\{\left\|h_{1,2-a}(\lambda)(\lambda+A)^{-1}(2\pi-i\log(A))^{-2}(1+A)^{-\nu}\right\|_{\mathcal{L}(X)}\mid \lambda \in i \mathbb{R}\setminus\{0\}, |\lambda|\leq 1\}<\infty. 
\end{equation}

Finally, by combining~\eqref{eqq53},~\eqref{eqq55} and~\eqref{eqq56}, it follows that
\begin{equation*}
 \sup\{\|h_{1,\zeta}(\lambda)(\lambda+A)^{-1}L_{\nu,\Tilde{c}}(A)\|_{\mathcal{L}(X)}\mid \lambda \in i \mathbb{R}\setminus\{0\}, |\lambda|\leq 1\}<\infty.    
\end{equation*}

\

\textbf{Case 1(b): $\Tilde{c}\in(2,3]$}. In this case, $a\in[1,2)$; then, by the moment inequality, one gets for each $\lambda\in i\mathbb{R}\setminus\{0\}$,
  \begin{eqnarray*}
\|h_{1,\zeta}(\lambda)(\lambda+A)^{-1}L_{\nu,\Tilde{c}}(A)\|_{\mathcal{L}(X)} &\lesssim& \|h_{1,2-a}(\lambda)(\lambda+A)^{-1}L_{\nu,2}(A)\|^{2-\Tilde{c}}_{\mathcal{L}(X)}\|h_{1,3-a}(\lambda)(\lambda+A)^{-1}L_{\nu,3}(A)\|^{\Tilde{c}-1}_{\mathcal{L}(X)},
\end{eqnarray*}
and it remains to estimate $\|h_{1,3-a}(\lambda)(\lambda+A)^{-1}L_{\nu,3}(A)\|^{\Tilde{c}-1}_{\mathcal{L}(X)}$.
  Note that for each $\lambda\in i\mathbb{R}\setminus\{0\}$, $\varepsilon>0$ and each $x\in X$, one has (here, $y=(1+A)^{-\nu}x$)
\begin{eqnarray*}
  \nonumber h_{1,3-a}(\lambda)(\lambda+A_\varepsilon)^{-1}(2\pi-i\log(A_\varepsilon))^{-3}y &=&  \frac{h_{1,3-a}(\lambda)(\lambda+A_\varepsilon)^{-1}y}{(2\pi-i\log(-\lambda))^{3}}\\
  &+&ih_{1,3-a}(\lambda)\int_{0}^{\infty} \frac{(26\pi^3-24\pi^2 i \log(t)+6\pi \log(t))(t+A_\varepsilon)^{-1}y}{(3\pi^2-4\pi i\log(t)-\log(t)^2)^{3}(\lambda-t)} dt,
\end{eqnarray*}
and then, by taking the limit $\varepsilon\to 0^+$ on both sides of the last identity, one gets %where we have let $\theta \to \pi$ in $(a)$ and $r\to 0$, $R\to \infty$ in $(b)$.Now, note that by Dominated Convergence, we have
\begin{eqnarray*}
h_{1,3-a}(\lambda)(\lambda+A)^{-1}L_{\nu,3}(A)x&=&\frac{h_{1,3-a}(\lambda)(\lambda+A)^{-1}y}{(2\pi-i\log(-\lambda))^{3}}+ \\
&+& \int_{0}^{\infty} \frac{ih_{1,3-a}(\lambda)(26\pi^3-24\pi^2 i \log(t)+6\pi \log(t)}{(3\pi^2-4\pi i\log(t)-\log(t)^2)(\lambda-t)}(t+A)^{-1}y\,dt.
\end{eqnarray*}

Thus, by relation~\eqref{cubo},
\begin{eqnarray*}
&&\left\|h_{1,3-a}(\lambda)(\lambda+A)^{-1}L_{\nu,3}(A)\right\|_{\mathcal{L}(X)}\lesssim\left\|\frac{h_{1,3-a}(\lambda)(\lambda+A)^{-1}}{(2\pi-i\log(-\lambda))^{3}}\right\|_{\mathcal{L}(X)}\\
&&+\int_{0}^{e^{-\sqrt{3}\pi}} \frac{|h_{1,3-a}(\lambda)| (26\pi^3+24\pi^2|\log(t)|+6\pi\log(t)^2)\|(t+A)^{-1}\|_{\mathcal{L}(X)}}{(\pi^2+\log(t)^{2})^3|\lambda-t|} dt\\
  &&+ \int_{e^{-\sqrt{3}\pi}}^{e^{\sqrt{3}\pi}} \frac{|h_{1,3-a}(\lambda)|}{|3\pi^2-4\pi i\log(t)-\log(t)^2|||\lambda-t|}\|(t+A)^{-1}\|_{\mathcal{L}(X)} dt
  \end{eqnarray*}

\begin{eqnarray*}
  &&+|h_{1,3-a}(\lambda)|\int_{e^{\sqrt{3}\pi}}^{\infty} \frac{26\pi^3+24\pi^2|\log(t)|+6\pi\log(t)^2}{(\pi^2+\log(t)^{2})^{3}|\lambda-t|}\|(t+A)^{-1}\|_{\mathcal{L}(X)} dt\\
&&\lesssim \left\|\frac{h_{1,3-a}(\lambda)}{(2\pi-i\log(-\lambda))^{3}}(\lambda+A)^{-1}\right\|_{\mathcal{L}(X)}+ \int_{0}^{\infty} \frac{|h_{1,3-a}(\lambda)|f(t)}{t(\pi^2+\log(t)^{2})^{3}(|\lambda|+t)} dt\\
&&+ \int_{e^{-\sqrt{3}\pi}}^{e^{\sqrt{3}\pi}} \frac{|h_{1,3-a}(\lambda)|}{t^2|3\pi^2-4\pi i\log(t)-\log(t)^2|}dt\\%+\int_{e^{\sqrt{3}\pi}}^{\infty} \frac{|h_{1,3-c}(\lambda)|f(t)}{t(\pi^2+\log(t)^{2})(|\lambda|+t)} dt
%\end{eqnarray*}
%\begin{equation*}
&&\lesssim\left\|\frac{h_{1,3-a}(\lambda)}{(2\pi-i\log(-\lambda))^3}(\lambda+A)^{-1}\right\|_{\mathcal{L}(X)}+|h_{1,3-a}(\lambda)|
  \frac{(|\lambda|\log(|\lambda|)^2-2(|\lambda|\log(|\lambda|)-|\lambda|+1)}{|\lambda|\log(|\lambda|)^3}\\
  &&+ |h_{1,3-a}(\lambda)|,
\end{eqnarray*}
where for each $t>0$,
\begin{equation*}
f(t)=\pi^2((-2+\pi^2)t-2)+t\log(t)^4-4t\log(t)^3+2((3+\pi^2)t+3)\log(t)^2-4\pi^2t\log(t).
\end{equation*}

By proceeding as in {\bf{Case 1(a)}}, one concludes that 
\begin{equation*}
 \sup\{\|h_{1,\zeta}(\lambda)(\lambda+A)^{-1}L_{\nu,\Tilde{c}}(A)\|_{\mathcal{L}(X)}\mid \lambda \in i \mathbb{R}\setminus\{0\}, |\lambda|\leq 1\}<\infty.    
\end{equation*}

{\bf{Case 1(c): $\Tilde{c}>3$}}.  In this case, $a\geq 2$. Let $\zeta=\zeta_1+\zeta_2$, with $\zeta_2\in(1,2)$. Again, by applying the moment inequality over $\zeta_2$, one gets
\begin{eqnarray*}\label{eq20}
  &&\nonumber\|h_{1,\zeta}(\lambda)(\lambda+A)^{-1}L_{\nu,a+\zeta_1+\zeta_2}(A)\|_{\mathcal{L}(X)} \lesssim   \|h_{1,\zeta}(\lambda)(\lambda+A)^{-1}L_{\nu,a+\zeta_2}(A)\|_{\mathcal{L}(X)}\\
 & &\lesssim\|h_{1,1}(\lambda)(\lambda+A)^{-1}L_{\nu,1+a}(A)\|^{2-\zeta_2}_{\mathcal{L}(X)} \|h_{1,2}(\lambda)(\lambda+A)^{-1}L_{\nu,2+a}(A)\|^{\zeta_2-1}_{\mathcal{L}(X)}.
\end{eqnarray*}
Let $\gamma$ be the same path as presented in {\bf{Case 1(a)}}. Then, for each $\varepsilon>0$ and each $x\in X$,
\begin{eqnarray*}
\nonumber&& h_{1,1}(\lambda)(\lambda+A_\varepsilon)^{-1}(2\pi-i\log(A_\varepsilon))^{-(1+a)}x = \frac{h_{1,1}(\lambda)}{2\pi i} \int_{\gamma} \frac{1}{(2\pi-i\log(z))^{1+a}}R(z,A_\varepsilon) (\lambda+A_\varepsilon)^{-1} x dz\\
&&\stackrel{\theta \rightarrow \pi}{\longrightarrow}
\frac{h_{1,1}(\lambda) (\lambda+A_\varepsilon)^{-1}x}{(2\pi-i\log(-\lambda))^{1+a}}+\frac{1}{2\pi i} \int_{r}^{R}\frac{h_{1,1}(\lambda)(t+A_\varepsilon)^{-1}}{(2\pi-i\log(t))^{1+a}(\lambda-t)} x dt\\
&&- \frac{1}{2\pi i} \int_{r}^{R}\frac{h_{1,1}(\lambda)}{(3\pi-i\log(t))^{1+a}(\lambda-t)} (t+A_\varepsilon)^{-1} x dt\\
&&+ \frac{1}{2\pi i} \int_{-\pi}^{\pi}\frac{ih_{1,1}(\lambda)Re^{is}}{(2\pi-s-i\log(R))^{1+a}(\lambda+Re^{is})}R(Re^{is},A_\varepsilon)x ds\\
&&- \frac{h_{1,1}(\lambda)}{2\pi i} \int_{-\pi}^{\pi}\frac{ire^{is}}{(2\pi+s-i\log(r))^{1+a}(\lambda+re^{is})}R(re^{is},A_\varepsilon) x ds\\
%\end{eqnarray*}
%\begin{equation*}
&&\stackrel{r\to 0,R\to\infty}{\longrightarrow}\frac{h_{1,1}(\lambda)}{(2\pi-i\log(-\lambda))^{1+a}}(\lambda+A_\varepsilon)^{-1}x+ \frac{1}{2\pi i} \int_{0}^{\infty}\frac{h_{1,1}(\lambda)(t+A_\varepsilon)^{-1}}{(\pi-i\log(t))^{1+a}(\lambda-t)}xdt\\
&&- \frac{1}{2\pi i} \int_{0}^{\infty}\frac{h_{1,1}(\lambda)(t+A_\varepsilon)^{-1}}{(3\pi-i\log(t))^{1+a}(\lambda-t)}x dt.
\end{eqnarray*}
Now, it follows from dominated convergence that for each $x\in X$,
\begin{eqnarray*}
  \nonumber &&h_{1,1}(\lambda)(\lambda+A)^{-1}(2\pi-i\log(A))^{-(1+a)}x = \frac{h_{1,1}(\lambda)(\lambda+A)^{-1}x}{(2\pi-i\log(-\lambda))^{1+a}}\\
  &+& \frac{1}{2\pi i} \int_{0}^{\infty}\frac{h_{1,1}(\lambda)}{(\pi-i\log(t))^{1+a}(\lambda-t)}(t+A)^{-1} xdt- \frac{1}{2\pi i} \int_{0}^{\infty}\frac{h_{1,1}(\lambda)(t+A)^{-1}}{(3\pi-i\log(t))^{1+a}(\lambda-t)}x dt.
\end{eqnarray*}

Therefore,
\begin{eqnarray*}
  &&\| h_{1,1}(\lambda)(\lambda+A)^{-1}(2\pi-i\log(A))^{-(1+a)}\|_{\mathcal{L}(X)}\leq   \left\|\frac{h_{1,1}(\lambda)(\lambda+A)^{-1}}{(2\pi-i\log(-\lambda))^{1+a}}\right\|_{\mathcal{L}(X)}\\
  &&+ \frac{1}{2\pi} \int_{0}^{\infty}\frac{|h_{1,1}(\lambda)|\|(t+A)^{-1}\|_{\mathcal{L}(X)}}{(\pi^2+\log(t)^2)(|\lambda|+t)} dt+ \frac{1}{2\pi} \int_{0}^{\infty}\frac{|h_{1,1}(\lambda)|\|(t+A)^{-1}\|_{\mathcal{L}(X)}}{(\pi^2+\log(t)^2)(|\lambda|+t)} dt\\
&&= \left\|\frac{h_{1,1}(\lambda)(\lambda+A)^{-1}}{(2\pi-i\log(-\lambda))^{1+a}}\right\|_{\mathcal{L}(X)}+2\frac{|h_{1,1}(\lambda)|(|\lambda|-1)}{|\lambda|\log(|\lambda|)}.
\end{eqnarray*}
Now, by the same reasoning as before, one gets 
\begin{eqnarray*}
  \nonumber &&h_{1,2}(\lambda)(\lambda+A)^{-1}(2\pi-i\log(A))^{-(2+a)} = \frac{h_{1,2}(\lambda)(\lambda+A)^{-1}}{(2\pi-i\log(-\lambda))^{2+a}}\\
  &&+ \frac{1}{2\pi i} \int_{0}^{\infty}\frac{h_{1,2}(\lambda)}{(\pi-i\log(t))^{2+a}(\lambda-t)}(t+A)^{-1} dt- \frac{1}{2\pi i} \int_{0}^{\infty}\frac{h_{1,2}(\lambda)(t+A)^{-1}}{(3\pi-i\log(t))^{2+a}(\lambda-t)} dt, 
\end{eqnarray*}
so
\begin{eqnarray*}
\| h_{1,2}(\lambda)(\lambda+A)^{-1}L_{\nu,2+a}(A)\|_{\mathcal{L}(X)}&\lesssim& \left\|\frac{h_{1,2}(\lambda)(\lambda+A)^{-1}}{(2\pi-i\log(-\lambda))^{2+a}}\right\|_{\mathcal{L}(X)}+ \\
&+& \frac{|h_{1,2}(\lambda)|}{\pi} \int_{0}^{\infty}\frac{(\pi^2-2(1+1/t)\log(t)+\log(t)^2)}{(\pi^2+\log(t)^2)^2(|\lambda|+t)}dt\\
%\end{eqnarray*}
%\begin{equation*}
&=&\left\|\frac{h_{1,2}(\lambda)(\lambda+A)^{-1}}{(2\pi-i\log(-\lambda))^{2+a}}\right\|_{\mathcal{L}(X)}+\frac{|h_{1,2}(\lambda)|(|\lambda|\log(|\lambda|)-|\lambda|+1)}{\pi|\lambda|\log(|\lambda|)^2}.
\end{eqnarray*}

Again, by proceeding as in {\bf{Case 1(a)}}, one concludes that 
\begin{equation*}
 \sup\{\|h_{1,\zeta}(\lambda)(\lambda+A)^{-1}L_{\nu,a+\zeta_1+\zeta_2}(A)\|_{\mathcal{L}(X)}\mid \lambda \in i \mathbb{R}\setminus\{0\}, |\lambda|\leq 1\}<\infty.    
\end{equation*}

\

$\bullet$~{\textbf{Case 2: $\alpha \geq 2$.}} By using the functional calculus for $H_0^\infty$ functions (see Remark~\ref{remark2.2}), one gets for each $x\in X$, 
\begin{eqnarray*}
h_{1,\zeta}(\lambda)(\lambda+A)^{-1}A^{\alpha-1}(1+A)^{-(\alpha-1)}(2\pi-i\log(A))^{-\Tilde{c}} x&=& \frac{h_{1,\zeta}(\lambda)}{2\pi i}\int_{\Gamma}\frac{z^{\alpha-1}(\lambda+A)^{-1}}{(1+z)^{\alpha-1}h_{0,\Tilde{c}}(z)}R(z,A)xdz\\
&=& \frac{h_{1,\zeta}(\lambda)(-\lambda)^{\alpha-1}}{(1-\lambda)^{\alpha-1}(2\pi-i\log(-\lambda))^{\Tilde{c}}}(\lambda+A)^{-1}x\\
&+&h_{1,\zeta}(\lambda)S^{''}_{\lambda}x,
\end{eqnarray*}
where
\begin{equation*}
S^{''}_{\lambda}:=\frac{1}{2\pi i}\int_{\Gamma}\frac{z^{\alpha-1}}{(1+z)^{\alpha-1}h_{0,\Tilde{c}}(z)(z+\lambda)}R(z,A)dz.
\end{equation*}

The function $z\mapsto (2\pi-i\log(z))^{-\Tilde{c}}R(z,A)$ is integrable on $\Gamma$ and by Lemma 5.9 in \cite{rozendaal}, for $z\in \Gamma$ and $|\lambda|\leq 1$, one has
\begin{equation*}
\left|\frac{z^{\alpha-1} h_{1,\zeta}(\lambda)}{(1+z)^{\alpha-1}(z+\lambda)}\right|\leq \frac{C}{|1-\lambda|}\leq C;
\end{equation*}
hence, $\sup\{\Vert h_{1,\zeta}(\lambda)S^{''}_{\lambda} \Vert_{\mathcal{L}(X)}\mid \lambda \in i \mathbb{R}\setminus\{0\}, |\lambda|\leq 1\}<\infty$, and since $\left\Vert \dfrac{h_{1,\zeta}(\lambda)(-\lambda)^{\alpha-1}(\lambda+A)^{-1}}{(1-\lambda)^{\alpha-1}(2\pi-i\log(-\lambda))^{\Tilde{c}}}\right\Vert_{\mathcal{L}(X)}$ is also bounded (by hypothesis), then
\begin{equation*}
 \sup\{\|h_{1,\zeta}(\lambda)(\lambda+A)^{-1}A^{\alpha-1}(1+A)^{-(\alpha-1)}(2\pi-i\log(A))^{-\Tilde{c}}\|_{\mathcal{L}(X)}\mid \lambda \in i \mathbb{R}\setminus\{0\}, |\lambda|\leq 1\}<\infty.    
\end{equation*}

\

$\bullet$~{\textbf{Case 3: $\alpha \in (1,2)$.}} By the moment inequality (applied over $\alpha-1\in(0,1)$), one gets 
\begin{eqnarray*}
  &&\|h_{1,\zeta}(\lambda)(\lambda+A)^{-1}(A(1+A)^{-1})^{\alpha-1}L_{\nu,\Tilde{c}}(A)\|_{\mathcal{L}(X)}\\
 && \lesssim\|h_{1,\zeta}(\lambda)(\lambda+A)^{-1}L_{\nu,\Tilde{c}}(A)\|^{2-\alpha}_{\mathcal{L}(X)} \|h_{1,\zeta}(\lambda)(\lambda+A)^{-1}A(1+A)^{-1}L_{\nu,\Tilde{c}}(A)\|^{\alpha-1}_{\mathcal{L}(X)}.
\end{eqnarray*}
The first factor is treated as in {\bf Case 1}, and the second factor is treated as in {\bf Case 2}.

\

{\bf{Item (b)}}

$\bullet$~{\bf{Case 1: $\alpha=1.$}} Let $\zeta>1$ and set $\Tilde{c}:=\zeta+a>1$.

Given that the operator $(\log(2+A))^{\Tilde{c}}(2\pi-i\log(A))^{-\Tilde{c}}$ is closed, it follows from the Closed Graph Theorem that it is bounded; hence, 
\begin{equation*}
    \|(\lambda+A)^{-1}(1+A)^{-1}(2\pi-i\log(A))^{-\Tilde{c}}\|_{\mathcal{L}(X)}\lesssim \|(\lambda+A)^{-1}(1+A)^{-1}\log(A+2)^{-\Tilde{c}}\|_{\mathcal{L}(X)}.
\end{equation*}
Now, by Proposition~\ref{prop3.1}, one gets 
\begin{eqnarray*}
\sup\left\{\frac{|\lambda|}{(1+|\lambda|)^{1-\beta_0}}
|(2\pi-\log(\lambda))|^{\zeta}\|(\lambda+A)^{-1}(1+A)^{-1}\log(A+2)^{-\Tilde{c}}\|_{\mathcal{L}(X)}\mid \lambda \in i \mathbb{R}, |\lambda|\geq 1\right\}<\infty.
\end{eqnarray*}

\

$\bullet$~{\textbf{Case 2: $\alpha \geq 2$.}}
Let $g_{\alpha,\zeta}(\lambda)=\dfrac{\lambda^{\alpha}}{(1-\lambda)^{1-\beta_0}}(2\pi-i\log(\lambda))^\zeta$, with $\lambda\in i\mathbb{R}\setminus\{0\}$; then, by the functional calculus for $H_0^\infty$ functions (see Remark~\ref{remark2.2}), for each $x\in X$, one has 
\begin{eqnarray*}
  &&g_{1,\zeta}(\lambda)(\lambda+A)^{-1}A^{\alpha-1}(1+A)^{-(\alpha+\beta+\beta_0-1)}(2\pi-i\log(A))^{-\Tilde{c}}x\\
  &&=\frac{g_{1,\zeta}(\lambda)}{2\pi i}\int_{\Gamma}\frac{z^{\alpha-1}(\lambda+A)^{-1}}{(1+z)^{\alpha+\beta+\beta_0-1}(2\pi-i\log(z))^{\Tilde{c}}}R(z,A)xdz\\
&&= \frac{g_{1,\zeta}(\lambda)(-\lambda)^{\alpha-1}}{(1-\lambda)^{\alpha+\beta+\beta_0-1}(2\pi-i\log(-\lambda))^{\Tilde{c}}}(\lambda+A)^{-1}x+ g_{1,\zeta}(\lambda)T^{''}_{\lambda}x,
\end{eqnarray*}
where
\begin{equation*}
T^{''}_{\lambda}:=\frac{1}{2\pi i}\int_{\Gamma}\frac{z^{\alpha-1}}{(1+z)^{\alpha+\beta+\beta_0-1}(2\pi-i\log(z))^{\Tilde{c}}(z+\lambda)}R(z,A)dz,
\end{equation*}
with $\Gamma$ the path defined in the proof of Proposition~\ref{prop3.1}.
The function $z\mapsto (2\pi-i\log(z))^{-\Tilde{c}}R(z,A)$ is integrable on $\Gamma$ and by Lemma 5.9 in \cite{rozendaal}, for $z\in \Gamma$ and $|\lambda|\geq 1$, one has
\begin{equation*}
\left|\frac{z^{\alpha-1} g_{1,\zeta}(\lambda)}{(1+z)^{\alpha+\beta+\beta_0-1}(z+\lambda)}\right|\lesssim \frac{ |g_{1,\zeta}(\lambda)|}{|1-\lambda|}\leq C;
\end{equation*}
thus, $\sup\{\Vert g_{1,\zeta}(\lambda)T^{''}_{\lambda}\Vert_{\mathcal{L}(X)}\mid \lambda\in i\mathbb{R},\;|\lambda|\ge 1\}<\infty$, and since
\[\sup\left\{\left\Vert \frac{g_{1,\zeta}(\lambda)(-\lambda)^{\alpha-1}}{(1-\lambda)^{\alpha+\beta+\beta_0-1}(2\pi-i\log(-\lambda))^{\Tilde{c}}}(\lambda+A)^{-1}\right\Vert_{\mathcal{L}(X)}\mid \lambda\in i\mathbb{R}, |\lambda|\geq 1\right\}<\infty,\] by hypothesis, it follows that
\begin{eqnarray*}
\sup\left\{|g_{\alpha,\zeta}(\lambda)|\|(\lambda+A)^{-1}A^{\alpha-1}(1+A)^{-\beta-\beta_0-\alpha+1}(2\pi-i\log(A))^{-\Tilde{c}}\|_{\mathcal{L}(X)}\mid \lambda \in i \mathbb{R}, |\lambda|\geq 1\right\}<\infty.
\end{eqnarray*}

\

$\bullet$~{\bf{Case 3: $\alpha \in (1,2)$.}} It follows from the moment inequality (applied to $\alpha-1\in(0,1)$) that
\begin{eqnarray*}
  &&\|g_{1,\Tilde{c}}(\lambda)(\lambda+A)^{-1}(A(1+A)^{-1})^{\alpha-1}L_{\beta+\beta_0,\Tilde{c}}(A)\|_{\mathcal{L}(X)}\\
  &&\lesssim \|g_{1,\Tilde{c}}(\lambda)(\lambda+A)^{-1}L_{\beta+\beta_0,\Tilde{c}}(A)\|^{2-\alpha}_{\mathcal{L}(X)} \|g_{1,\Tilde{c}}(\lambda)(\lambda+A)^{-1}A(1+A)^{-1}L_{\beta+\beta_0,\Tilde{c}}(A)\|^{\alpha-1}_{\mathcal{L}(X)}.
\end{eqnarray*}

The first factor must be treated as in {\bf Case 1}, and the second one as in {\bf Case 2}.

\section{Estimates}

\begin{lemma}\label{lemmaB.1}
\begin{rm}
Let $\mu,\zeta \geq 0$ and $\nu \geq 1$; then, for each $t\geq 0$, 
\begin{enumerate}
\item \hspace{2cm} $\dint_{i\infty}^{-i\infty}e^{-\lambda t} \dfrac{1}{(1+\lambda)^{\nu}(\log(2+\lambda))^{\zeta}}d\lambda=0$. 
%\end{equation*}
\item %\begin{equation*}
\hspace{2cm} $\dint_{i\infty}^{-i\infty}e^{-\lambda t} \dfrac{\lambda^{\mu}}{(1+\lambda)^{\nu+\mu}(2\pi-i\log(\lambda))^{\zeta}}d\lambda=0$. 
%\end{equation*}
\end{enumerate}

\end{rm}
\end{lemma}

\begin{proof} We just present the proof of the first equality, since the proof of the other one is analogous. Let us first show the following statement.

\noindent{\textbf{Claim:}}
\begin{equation}\label{eqClaim}
\frac{1}{2\pi i}\int_{i\infty}^{-i\infty}e^{-\lambda t} \frac{1}{(1+\lambda)^{\nu}(\log(2+\lambda))^{\zeta}}d\lambda= \frac{1}{2\pi i}\int_{\Gamma_\varphi}e^{-\lambda t} \frac{1}{(1+\lambda)^{\nu}(\log(2+\lambda))^{\zeta}}d\lambda,
\end{equation}
where $\Gamma_\varphi=\{re^{i\varphi}\mid r\in[0,\infty)\}\cup \{re^{-i\varphi}\mid r\in[0,\infty)\}$ and $0<\varphi<\frac{\pi}{2}$.

Namely, for $t\geq 0$, set $i\mathbb{R}\ni \lambda \mapsto h_t(\lambda):=e^{-\lambda t} \dfrac{1}{(1+\lambda)^{\nu}(\log(2+\lambda))^{\zeta}}$, and for each $R,r>0$ and each $\eta\in[\varphi,\pi/2]$, set $\Gamma^{+}_{R,\varphi}=\{Re^{i\theta}\mid \theta \in(\varphi,\frac{\pi}{2})\}$, $\Gamma^{+}_{r,\varphi}=\{re^{i\theta}\mid \theta \in(\varphi,\frac{\pi}{2})\}$, $\Gamma^{-}_{R,\varphi}=\{Re^{-i\theta}\mid \theta \in(\varphi,\frac{\pi}{2})\}$, $\Gamma^{-}_{r,\varphi}=\{re^{-i\theta}\mid \theta \in(\varphi,\frac{\pi}{2})\}$, $\gamma^{+}_{\eta}=\{se^{i\eta}\mid s\in [r,R]\}$ and $\gamma^{-}_{\eta}=\{se^{-i\eta}\mid s\in [r,R]\}$. By Cauchy’s Integral Theorem,
\begin{equation}\label{eqIC1}
  -\int_{\Gamma^{+}_{R,\varphi}}h_t(\lambda)d\lambda+\int_{\gamma^{+}_{\frac{\pi}{2}}}h_t(\lambda)d\lambda+\int_{\Gamma^{+}_{r,\varphi}}h_t(\lambda)d\lambda-\int_{\gamma^{+}_{\varphi}}h_t(\lambda)d\lambda=0,
\end{equation}
and
\begin{equation}\label{eqIC2}    \int_{\Gamma^{-}_{R,\varphi}}h_t(\lambda)d\lambda-\int_{\gamma^{-}_{\frac{\pi}{2}}}h_t(\lambda)d\lambda-\int_{\Gamma^{-}_{r,\varphi}}h_t(\lambda)d\lambda+\int_{\gamma^{-}_{\varphi}}h_t(\lambda)d\lambda=0.
\end{equation}

Note that
\begin{eqnarray*}
  \left|\int_{\Gamma^{\pm}_{R,\varphi}}h_t(\lambda)d\lambda\right|&\leq& %\int_{\Gamma^{+}_{R,\varphi}}|h_t(\lambda)|d|\lambda|=
  \int_{\varphi}^{\frac{\pi}{2}}\frac{R e^{-t\cos \theta}}{|(1+Re^{\pm i\theta})|^{\nu}|\log(2+Re^{\pm i\theta})|^{\zeta}}d\theta\\
&\leq& 2^{\nu/2} \int_{\varphi}^{\frac{\pi}{2}}\frac{Re^{-t\cos \theta}}{(1+R)^{\nu}(1+\cos(\theta))^{\nu/2}\left(\log(2+R)+\frac{1}{2}\log\left(\frac{1+\cos(\theta)}{2}\right)\right)^{\zeta}}d\theta\\
&\lesssim & \frac{R^{1-\nu}}{\log(2+R)} 
\end{eqnarray*}
and
\begin{equation*}
\left|\int_{\Gamma^{\pm}_{r,\varphi}}h_t(\lambda)d\lambda\right| \lesssim r
\end{equation*} 
By adding the equations \eqref{eqIC1} and \eqref{eqIC2}, and by taking the limits $R\to \infty$, $r\to 0$, one gets \eqref{eqClaim}.

%\noindent{\textbf{Proof of the Lemma \ref{lemma5.1} }}

\

By Claim, it suffices to prove that 
\begin{equation*}
\frac{1}{2\pi i}\int_{\Gamma_\varphi}e^{-\lambda t} \frac{1}{(1+\lambda)^{\nu}\log(2+\lambda)^{\zeta}}d\lambda=0.
\end{equation*}
It follows from Cauchy's Integral Theorem that for each $0<r<R$,
\begin{equation}\label{acabou}
\frac{1}{2\pi i}\int_{\Gamma_\varphi}h_t(\lambda)d\lambda+\frac{1}{2\pi i}\int_{\gamma_{R,\varphi}}h_t(\lambda)d\lambda+\frac{1}{2\pi i}\int_{\gamma_{r,\varphi}}h_t(\lambda)d\lambda=0,
\end{equation}
with $\gamma_{R,\varphi}:=\{Re^{i\theta}\mid \theta\in [-\varphi,\varphi]\}$ and $\gamma_{r,\varphi}:=\{r e^{-i\theta}\mid \theta\in [-\varphi,\varphi]\}$.

Note that for each sufficiently large $R$, 
\begin{equation*}
\left| \int_{\gamma_{R,\varphi}} e^{-\lambda t} \frac{1}{(1+\lambda)^{\nu}\log(2+\lambda)^{\zeta}}d\lambda\right| \lesssim \frac{R^{1-\nu}}{\log(2+R)},
\end{equation*}
%tends to zero as $R\to \infty$
and for each sufficiently small $r$,
\begin{equation*}
\left|\int_{\gamma_{r,\varphi}} e^{-\lambda t} \frac{1}{(1+\lambda)^{\nu}\log(2+\lambda)^{\zeta}}d\lambda\right| \lesssim r.
\end{equation*}
The result follows by taking the limits $r\to 0$ and $R\to\infty$ in relation~\eqref{acabou}. % Which proves the Lemma.
\end{proof}

\begin{lemma}\label{lemmaB2}
\begin{rm}
 Let $\varphi \in (0,\frac{\pi}{2}]$ and $\theta \in (\pi-\varphi,\pi)$. Set $\Omega:= \overline{\mathbb{C}_{+}}\setminus (S_\varphi\cup \{0\})$ and let $\Gamma:=\{re^{i\theta}\mid r \in [0,\infty)\}\cup \{re^{i\theta}\mid r \in [0,\infty)\}$ be oriented from $\infty e^{i\theta}$ to $\infty e^{-i\theta}$. Then, for each $\alpha \in [0,\infty)$, $\beta \in (0,\infty)$, $\eta \in (0, 1]$ and each $\lambda \in \Omega$, one has
\begin{enumerate}[a)]
\item $\dint_{\Gamma} \dfrac{1}{(\eta+z)^{\beta}(\log(1+\eta+z))^{\zeta}(z+\lambda+\eta-1)} dz= \dfrac{1}{(1-\lambda)^{\beta}(\log(2-\lambda))^{\zeta}}$.
%\end{equation*}
\item 
%\begin{equation*}
$\dint_{\Gamma} \dfrac{z^{\alpha}}{(\eta+z)^{\alpha+\beta}(2\pi-i\log(-1+\eta+z))^{\zeta}(z+\lambda+\eta-1)} dz= \dfrac{(1-\lambda-\eta)^{\alpha}}{(1-\lambda)(2\pi-i\log(-\lambda))^{\zeta}}$.
%\end{equation*}
\end{enumerate}
\end{rm}
\end{lemma}  

\begin{proof} We just present the proof of item~a). Let $\lambda \in \Omega$. For each $r \in (0, \text{Im}(\lambda)/2]$ and each $R\geq 2|\lambda| + 2$, set $\gamma_{+}:=\{se^{i\theta}\mid s\in [r,R] \}$, $\gamma_{-}:=\{se^{-i\theta}\mid s\in [r,R]\}$, 
$\gamma_{r}:=\{re^{i\nu}\mid \nu\in [-\theta,\theta]\}$, $\gamma_{R}:=\{Re^{i\nu}\mid \nu\in [-\theta,\theta]\}$ and $\gamma_{r,R}:=(-\gamma_{+})\cup \gamma_{-}\cup (-\gamma_{r})\cup \gamma_{R}$. Let $f_{\beta,\zeta,\lambda}:\overline{\mathbb{C}_{+}}\rightarrow \mathbb{C}$ be given by the law $f_{\beta,\zeta,\lambda}(z)=\dfrac{1}{(\eta+z)^{\beta}(\log(1+\eta+z))^{\zeta}(z+\lambda+\eta-1)}$; then,
\begin{eqnarray*}
   \left| \int_{\gamma_R} f_{\beta,\zeta,\lambda}(z) dz\right|&\leq& %\int_{\gamma_R} \frac{1}{|\eta+z|^{\beta}|\log(1+\eta+z)|^{\zeta}|z+\lambda+\eta-1|} d|z|\\
    \int_{-\theta}^{\theta} \frac{R}{|\eta+Re^{i\nu}|^{\beta}\log(|1+\eta+Re^{i\nu}|)^{\zeta}|Re^{i\nu}+\lambda+\eta-1|} d\nu\\
   &\lesssim& \frac{R^{-\beta}}{\log(1+R)^{\zeta}},
\end{eqnarray*}
which goes to zero as $R\rightarrow \infty$. Similarly, one can show that 
\begin{equation*}
   \lim_{r\to 0} \left| \int_{\gamma_r} f_{\beta,\zeta,\lambda}(z) dz\right|=0.
\end{equation*}

On the other hand, by the Residue Theorem, one has 
\begin{equation*}
\int_{\gamma_{r,R}} \frac{1}{(\eta+z)^{\beta}(\log(1+\eta+z))^{\zeta}(z+\lambda+\eta-1)} dz =\frac{1}{(1-\lambda)^{\beta}\log(2-\lambda)^{\zeta}}.
\end{equation*}
Thus, it follows that
\begin{eqnarray*}
  &&\int_{\Gamma} \frac{1}{(\eta+z)^{\beta}(\log(1+\eta+z))^{\zeta}(z+\lambda+\eta-1)} dz\\
 && =\lim_{r\to 0, R\to \infty} \int_{\gamma_{r,R}} \frac{1}{(\eta+z)^{\beta}(\log(1+\eta+z))^{\zeta}(z+\lambda+\eta-1)} dz=\frac{1}{(1-\lambda)^{\beta}\log(2-\lambda)^{\zeta}}.
\end{eqnarray*}
\end{proof}

\begin{center} \Large{Acknowledgments} 
\end{center}

GSS thanks the partial support by CAPES (Brazilian agency). SLC thanks the partial support by Fapemig (Minas Gerais state agency; Universal Project under contract 001/17/CEX-APQ-00352-17).

%%%%%%%%%%%%%%%%%%%%%%%%%%%%%%%%%%%%%%%%%%%%%%%%%%%%%%%%%%%%%%%%%%%%%%%%%%%%%%

\

\noindent  Email: gesoares2017@gmail.com, Departamento de Matem\'atica, UFMG, Belo Horizonte, MG, 30161-970 Brazil

\noindent  Email: silas@mat.ufmg.br, Departamento de Matem\'atica, UFMG, Belo Horizonte, MG, 30161-970 Brazil

\end{document}